\documentclass[a4paper,11pt,reqno]{amsart}
\usepackage{inputenc}
\usepackage{amsmath,amsthm,amsfonts,amssymb}
\usepackage[left=2.20cm,right=2.25cm,top=3cm,bottom=3cm]{geometry}
\usepackage[dvips]{graphicx}
\usepackage[english]{babel}
\usepackage{fouridx}
\usepackage{dsfont}
\usepackage{xcolor}
\usepackage{enumerate}
\usepackage{todonotes}
\newcommand{\toup}{\nearrow}
\newcommand{\todown}{\searrow}
\newcommand{\divv}{\mathrm{div}\,}

\newcommand\delb[1]{} 

\newcommand{\hA}{\hat{\mathrm{A}} }
\newcommand{\hV}{\hat{\mathrm{V}} }

\newcommand{\rA}{\mathrm{A} }
\newcommand{\rB}{\mathrm{B}}

\newcommand{\rH}{\mathrm{H} }
\newcommand{\rV}{\mathrm{V} }

\newcommand{\lb}{\langle}
\newcommand{\rb}{\rangle}

\newcommand\delh[1]{} 

\newcommand\del[1]{}

\theoremstyle{definition}
\newtheorem{Def}{Definition}[section]

\newtheorem{Rem}[Def]{Remark}

\newtheorem{example}[Def]{Example}

\theoremstyle{plain}

\newtheorem{assume}[Def]{Assumption}
\newtheorem{Clm}[Def]{Claim}
\newtheorem{Thm}[Def]{Theorem}
\newtheorem{Prop}[Def]{Proposition}
\newtheorem{proposition}[Def]{Proposition}

\newtheorem{Lem}[Def]{Lemma}
\newtheorem{Cc}[Def]{Corollary}

\newcommand{\rh}{\mathrm{H}}
\newcommand{\ve}{\mathrm{V}}
\newcommand{\eps}{\varepsilon\,}

\DeclareMathOperator{\Div}{div}

\numberwithin{equation}{section}

\title[2D Ericksen-Leslie with anisotropic energy ]{On the 2D Ericksen-Leslie equations with anisotropic energy and external forces }
\author[Z. Brze\'zniak]{Zdzis{\l}aw Brze{\'z}niak}
\address{Department of Mathematics \\
	University of York, Heslington, York YO10
	5DD, UK} \email{zdzislaw.brzezniak@york.ac.uk}

\author[
G. Deugou\'e]{Gabriel Deugou\'e}

\address{Department of Mathematics and Computer Science\\ University of Dschang, Dschang, Cameroon} \email{
	agdeugoue@yahoo.fr}

\author[P. Razafimandimby]{Paul Andr\'e  {Razafimandimby}}
\address{Department of Mathematics\\University of York, Heslington, York YO10
	5DD, UK} \email{paul.razafimandimby@york.ac.uk}
\thanks{This article is part of a project that is currently funded by the  European Union's Horizon 2020 research and innovation programme under the Marie Sk\l{}odowska-Curie grant agreement No. 791735 ``SELEs".\newline
		P. Razafimandimby is very grateful to the warm hospitality of the Lehrstuhl f\"ur Angewandte Mathematik, Montanuniversit\"at Leoben, Leoben (Austria), where part of the paper was written during the period 16 December 2019 till 30 January 2020.\newline
	G. Deugou\'e is thankful  to the financial support and the hospitality of the Department of Mathematics at the University of York during his visit in June 2019. He is also very grateful to the financial support from the Marie Sk\l{}odowska-Curie grant (agreement No. 791735) ``SELEs".
}
\begin{document}
\maketitle
	\begin{abstract}
	In this paper we consider the 2D Ericksen-Leslie equations which describes the hydrodynamics of nematic Liquid crystal with external body forces and anisotropic energy modeling the energy of applied external control such as magnetic or electric field. Under general assumptions on the initial data, the external data and the anisotropic energy, we prove the existence and uniqueness of global weak solutions with finitely many singular times. If the initial data and the external forces are sufficiently small, then we establish that the global weak solution does not have any singular times and is regular as long as the data are  regular.
	\end{abstract}
	\tableofcontents
	
	\section{Introduction}
	
	We  consider a  hydrodynamical system modeling the flow of liquid crystal materials with anisotropic energy in a 2D  bounded domain. More precisely, let $T>0$ and $\Omega\subset \mathbb{R}^2$ be a bounded domain with a smooth boundary $\partial \Omega$ and let us  consider
	\begin{subequations} \label{1a}
	\begin{align}
	&\partial_t v+ v\cdot \nabla v-\Delta v +\nabla \mathrm{p}=-\Div~ (\nabla d\odot\nabla d) + f ,\text{in $ [0,T)\times \Omega $}\\
&	\partial_t d + v\cdot \nabla d=-d\times(d\times(\Delta d-\phi^\prime(d))) + d\times g,~\text{in~$[0,T)\times \Omega $},\\
&	\Div v=0,~~\text{in~$[0,T)\times \Omega $},\\
&	v =	{\frac{\partial d}{\partial \nu}}=0,\text{ on~$[0,T)\times \partial \Omega $},\\
&	\lvert d \lvert=1,\text{in~$[0,T)\times \Omega $},\\
&	(v(0),d(0))=(v_{0},d_{0}),\text{in~$ \Omega $},
	\end{align}
	\end{subequations}
	where $v:[0,T)\times \Omega \to\mathbb{R}^{2}$, $d:[0,T)\times \Omega  \to\mathbb{S}^{2}$, where  $\mathbb{S}^2$ is  the unit sphere in $\mathbb{R}^{3}$, and $P:[0,T)\times \Omega  \to\mathbb{R}$ represent the velocity field of the flow, the macroscopic molecular orientation of the liquid crystal material  and the pressure function, respectively. In the system (\ref{1a}), the function $\phi:\mathbb{R}^{3}\to\mathbb{R}^{+}$ is a given map, $f:[0,T)\times \Omega  \to \mathbb{R}^2 $ and $g:[0,T)\times \Omega \to \mathbb{R}^3$ are given external forces. The symbol $\nu(x)$ is a unit outward normal vector at each point $x\in \partial \Omega$. The matrix $\nabla d\odot\nabla d$ is defined by
	\begin{equation*}
	[\nabla d\odot\nabla d]_{ij}= \sum_{k=1}^3 \partial_i d_k \partial_j d_k~ \text{for}~ 1\le i,j\le 3.
	\end{equation*}
	
	 Using the identities \begin{align*}
	  -\Div (\nabla d \odot \nabla d) = -\nabla d \Delta d + \frac12 \nabla \lvert \nabla d \rvert^2, \\
	 a\times (b\times c)=(a\cdot c)b-(a\cdot b)c  \text{ for } a, b, c\in\mathbb{R}^{3},
	 \end{align*}
	  we can rewrite system (\ref{1a}) as follows
	\begin{subequations}\label{1b-0}
	\begin{align}
	& \partial_t v+v\cdot \nabla v-\Delta v+\nabla \tilde{\mathrm{p}}=-\nabla d \Delta d + f,\\
& 	\partial_t d+v\cdot \nabla d=\Delta d+ |\nabla d|^{2}d-\phi^\prime(d)+(\phi^\prime(d)\cdot d)d+d\times g,\\
&	\Div v=0,\\
	&	v_{\lvert_{\partial\Omega}} =	{\frac{\partial d}{\partial \nu}}{\Big\lvert_{\partial\Omega}}=0, \label{Eq:BC-DirNeu}\\
&	\lvert d \rvert=1,\\
&	(v(0),d(0))=(v_{0},d_{0}),
	\end{align}
\end{subequations}
where
\begin{equation*}
\tilde{\mathrm{p}}= \mathrm{p} + \frac12 \lvert \nabla d \rvert^2 .
\end{equation*}
	While we focus our mathematical analysis on the system  \eqref{1b-0} with the Dirichlet and the Neumann boundary conditions \eqref{Eq:BC-DirNeu},  our results remain valid for the case of the periodic boundary conditions.  That is, our results remain true in the  case that $\Omega$ is a 2D torus $\mathbb{T}^2$ and  \eqref{Eq:BC-DirNeu} is replaced by
	\begin{equation*}
	\int_\Omega v(t, x) \;dx=0, \; \forall t \in (0,T].
	\end{equation*}
	
	The model \eqref{1b-0} is an oversimplification of a Ericksen-Leslie model of nematic liquid crystal with a simplified energy density
	\begin{equation*}
	\frac12 \lvert  \nabla d \rvert^2 + \phi(d).
	\end{equation*}
	The term $\frac12 \lvert  \nabla d \rvert^2 $ represents the one-constant simplification of the Frank-Oseen energy density and $\phi(d)$ represents an anisotropic energy density. One example of such anisotropic energy density is the magnetic energy density
	$$ \phi(n) = (n\cdot H)^2,$$
	when the nematic liquid crystal is subjected to the action of a constant magnetic field $H \in \mathbb{R}^3$. We also give different examples of mathematical models of anisotropic energy density later on.   For more details on physical modeling of liquid crystal under the action of external control such as magnetic or electric field we refer to the books  \cite{Gennes} and \cite{Stewart} and the papers \cite{Ericksen} and \cite{Leslie}.
	
	We should note that since \eqref{1b-0} was obtained by neglecting several terms such as the viscous Leslie stress tensor in the equation for $v$(see for instance \cite{Lin-Liu,Lin-Liu2}), the stretching and rotational effects for $d$,  one does not known whether it is  thermodynamically stable or consistent with the laws of thermodynamics. However, this model still  retains many  mathematical and essential features of the hydrodynamic equations for nematic liquid crystals. In recent years, several liquid crystal models which are thermodynamically consistent and stable have been recently developed and analyzed, see for instance the  \cite{Feireisl-2019}, \cite{Feireisl-2012}, \cite{Feiresil-2011}, \cite{MH+JP-2017}, \cite{MH+JP-2018}, \cite{MH et al-2014}, \cite{Sun+Liu}, \cite{Lin+Wang-2014} and references therein.
	
	In the absence of external forcings $f$, $g$ and the anisotropic energy potential $\phi(d)$, the system \eqref{1b-0} has extensively studied and several important results have been obtained. In addition to the papers we cited above we refer, among others, to \cite{Hong,Hong12,LLW,Lin-Liu,
		Lin-Liu2,Lin-Wang,WZZ} for results obtained prior to 2013, and to \cite{YC+SK+YY-2018,Hong14,MCH+YM-2019,Huang14,Huang16,JL+ET+ZX-2016,Lin+Wang16,Wang2014,Wang2016,WZZ15} for results obtained after 2014. For detailed reviews of the literature about the mathematical theory of nematic liquid crystals and other related models, we recommend the review articles \cite{Lin+Wang-2014,MH+JP-2018,Climent} and the recent papers \cite{MCH+YM-2019,Lin+Wang16}.
	
Let us now outline the contributions of our manuscript.
\begin{trivlist}
\item[\;\;\;(1)] 	In Section \ref{Sec:ExistRegularSol}, we prove by using Banach fixed point theorem that if $(v_0,d_0)\in D(\rA^\frac12)\times (D(\hA) \cap \mathcal{M} )$ and $(f,g)\in L^2(0,T; \rH\times D(\hA))$, then there exists a unique local regular solution $(v,d): [0,T_0] \to D(\rA^\frac12)\times D(\hA) \cap \mathcal{M}$ such that $C([0,T_0]; D(\rA^\frac12) \times (D(\hA)\cap \mathcal{M})) \cap L^2(0,T; D(\rA)\times  D(\hA^\frac32)$, see Theorem \ref{th}. Here
\[
\mathcal{M}=\{ d: \Omega \to \mathbb{R}{^3}: \lvert d(x) \rvert=1 \;\;\mathrm{Leb}\text{-a.e.} \},
\]
$\rA$ and $\hA$ are respectively the Stokes operator and the Neumann Laplacian,
see Section \ref{Sec:NotaPrelim-Sec} for the definitions of these operators and the space $\rH$.

\item[\;\;\;(2)] We exploit this result and the  local energy method developed in \cite{Struwe}, \cite{LLW}  and \cite{Hong}  to show in Section \ref{Sec:LocRegSolSmallENERG} that there exists universal constants $\eps_0>0$ and $r_0$ such that the following statement hold.

\textit{ If  $(f,g)\in L^2(0,T; \rh\times D(\hA^{1/2}))$ and $(v_0, d_0)$ belongs to $\ve\times (D(\hA)\cap \mathcal{M}) $  with small energy, i.e., there exists  $R_0\in (0,r_0]$ such that
\begin{equation*}
\sup_{x\in \Omega}\int_{B(x, R_0)} [\lvert v_0(x) \rvert^2+ \lvert \nabla d_0(x) \rvert^2 + \phi(d_0(x))]<\eps_0^2.
\end{equation*}
Then, there exists a time $T_0>0$ a unique $(v,d): [0,T_0] \to D(\rA^\frac12)\times (D(\hA) \cap \mathcal{M})$ such that $C([0,T_0]; D(\rA^\frac12) \times (D(\hA)\cap \mathcal{M})) \cap L^2(0<T; D(\rA)\times  D(\hA^\frac32)$ such that
\begin{equation*}
\frac12 \sup_{0\le t\le T_0 } \sup_{x\in \Omega}\int_{B(x,R_0)} \left(\lvert v(t,y)\rvert^2 + \lvert \nabla d(t,y)  \rvert^2 + 2 \phi(d(t,y))  \right) dy\le 2\eps_1^2.
\end{equation*}
}
We refer to Proposition \ref{Prop:LocalSolwithSmallEnergy} and its proof for more details.

\item[\;\;\;(3)] The two  results  above are exploited in Section \ref{Sec:ExistMaxLocStrongSol} in order to prove the global existence of our problem. This is our main  result. It  holds under weaker assumptions than those listed in  (1) and (2) above,  and is presented  in Theorem \ref{thm-main}.  It  can be summarized as follows.

\begin{Thm}\label{thm-main-intro}
	Let $(v_0,d_0)\in \rh\times (\rH^1\cap \mathcal{M})$. Then, there exist constants $\varrho_0>0 $ and $\eps_0>0 $ such that the following hold.
	If $(f,g)\in L^2(0,T; \rH^{-1} \times L^2)$, then
\begin{enumerate}[(i)]
	\item 	a number $L\in \mathbb{N}$, depending only on the norms of $(v_0,d_0)\in \rH \times \rH^1$ and $(f,g)\in L^2(0,T; \rH^{-1}\times L^2)$,
 a finite sequence   $0=T_0<T_1<\cdots<T_L\leq T$ and,
	\item a function $(\mathbf{u},\mathbf{d})\in C_{w}([0, T]; \mathrm{H} \times \rh^1) \cap L^2(0, T; \ve\times D(\hA) )$
 such that  {for all $t\in [0,T], \mathbf{d}(t)\in \mathcal{M} $} and
	\begin{enumerate}
		\item for every $i\in \{1, \ldots, L\}$, $(\mathbf{u},\mathbf{d})_{\lvert_{[T_{i-1}, T_i)}}\in C([T_{i-1}, T_i); \mathrm{H} \times \rh^1)$ with the left-limit at $T_i$,
which  satisfies  the variational form  of  problem \eqref{1b-0} on the interval $[T_{i-1}, T_i)$  with initial data $(v(T_{i-1}), d(T_{i-1}) )$.
\item  If $T_L<T$, then
$(\mathbf{u},\mathbf{d})_{\lvert_{[T_{L}, T]}}$ belongs to $ C([T_{L}, T]; \mathrm{H} \times \rh^1)$ and  satisfies  the variational form  problem \eqref{1b-0} on the interval $[T_{L}, T]$  with initial data $(v(T_{L}), d(T_{L}) )$.

		\item {There exists $\eps_1\in (0,\eps_0)$ such that for all $i\in \{1, \ldots, L\}$ and all $R\in (0,\varrho_0]$}
		\begin{equation*}
		\lim_{t\toup T_i} \sup_{x\in \Omega} \int_{B(x, R)}\left( \frac12 \lvert \mathbf{u}(t,y)\rvert^2 +  \frac12\lvert \nabla \mathbf{d}(t,y) \rvert^2 + \phi(d(t,y)) \right) dy\ge \eps_1^2.
		\end{equation*}
		\item At every time  $T_i$, $i\in \{1, \ldots, L\}$, there is a loss of energy at least $\eps_1^2\in (0, \eps_0^2)$, \textit{i.e.},
		\begin{equation*}
	\begin{split}
	& \hspace{1truecm}	\int_{\Omega }\left( \frac12 \lvert \mathbf{u}(T_i,y)\rvert^2 +  \frac12 \lvert \nabla \mathbf{d}(T_i,y) \rvert^2 + \phi(d(T_i,y)) \right) dy \\
	& \hspace{1truecm}   \le \int_{\Omega }\left( \frac12 \lvert \mathbf{u}(T_{i-1},y)\rvert^2 +  \frac12 \lvert \nabla \mathbf{d}(T_{i-1},y) \rvert^2 + \phi(d(T_{i-1},y)) \right) dy
		 +\frac12 \int_{T_{i-1}}^{T_1} \left[\lvert f\rvert_{\rH^{-1}}^2 + \lvert g\rvert^2_{L^2}\right] dt -\eps_1^2 .
	\end{split}
		\end{equation*}
	\end{enumerate}
\end{enumerate}
The numbers $T_1,\cdots, T_L$ are called the  \textbf{singular times} of the solution $(\mathbf{u},\mathbf{d})$.
\end{Thm}
Because of the presence of the anisotropic energy and the external forces, this result is a generalization  of the global existence of weak solution obtained in \cite{Hong} and \cite{LLW}.
\item[(4)] Finally, in  Section \ref{Sec:RegCompactSmallData} we prove that the set of singular times is empty  when the data 	$(v_0,d_0)\in \rh\times \rH^1$ and $(f,g)\in L^2(0,T; \rH^{-1}\times L^2)$ are sufficiently small.  We also  show that if the data are sufficiently regular and small, \textit{i.e.} $(v_0,d_0)\in D(\rA^\frac12 )\times D(\hA)$ and $(f,g)\in L^2(0,T; \rH \times \rH^1)$, then the weak solution becomes regular for all time. Moreover, for all $t \in [0,T)$ $(u(t),d(t))$ lies in a compact set of  $\rH\times \rH^1$. We refer the reader to Theorems \ref{Thm:NoSingSmall} and \ref{Thm:PrecompactInHH1} for more detail about these results.
\end{trivlist}	
	We close this introduction with the presentation of the  layout of the present paper. In Section \ref{Sec:NotaPrelim-Sec} we fix the frequently used notation in the manuscript. We also state and prove some auxiliary results which are essential to our analysis. Section \ref{Sec:ExistRegularSol} is devoted to the existence and uniqueness of of a regular solution to Problem \eqref{1b}. In Section \ref{Sec:LocRegSolSmallENERG} we prove that one  can find a small number $R_0>0$ and a unique maximal local regular solution $((v,d);T_0)$ such that its energy at any time $t\in [0,T_0)$ does not exceed  twice the supremum of all energies on $B(x,2R_0)$, $x\in \Omega$ of the initial data. This result will play a pivotal role in the proof of the existence of a maximal local strong solution to Problem \eqref{1b} in Section \ref{Sec:ExistMaxLocStrongSol}.

	\section{Notation and preliminaries}\label{Sec:NotaPrelim-Sec}
	Let $\Omega\subset\mathbb{R}^{2}$ be an open and bounded set. We denote by  $\Gamma=\partial\Omega$ the  boundary of $\Omega$. We assume that the closure $\overline{\Omega}$ of the set $\Omega$ is a manifold with $C^\infty$ boundary  $\Gamma:=\partial\Omega$ which   is a
	$1$-dimensional infinitely differentiable manifold being
	locally on one side of $\Omega$.
	
	Throughout this paper  $L^p(\Omega; \mathbb{R}^\ell)$,  $\mathrm{W}^{p,k}(\Omega; \mathbb{R}^\ell)$( $\rH^k(\Omega; \mathbb{R}^\ell)= \mathrm{W}^{2,p}(\Omega; \mathbb{R}^\ell)$), $p\in[1,\infty], k \in \mathbb{N}$,  and  $\rH^\alpha(\Omega; \mathbb{R}^\ell)$, $\alpha \in (0,\infty)$,  denote the  Lebesgue and Sobolev spaces whose elements take values in $\mathbb{R}^\ell$, $\ell=2,3$.  To shorten the notation we will just write $L^p$,   $\mathrm{W}^{p,k}$, $\rH^k$ and $\rH^\alpha$ irrespectively if the elements of these spaces  take values in $\mathbb{R}^2$ or $\mathbb{R}^3$.
	
	\subsection{Notations for the velocity field $v$}\label{subsect-Dirichlet}
	
	The following is  an abridged version of notations and preliminary of the paper \cite{Brz+Cer+F_2015}. The facts we enumerate here can be found in \cite[Section 2]{Brz+Cer+F_2015} and references therein.
	%
	%
	%
	
	Let $\mathcal{D}(\Omega)$ (resp. $\mathcal{D
	}(\overline{\Omega})$) be the set of all $C^\infty$ class vector  fields  $u:\mathbb{R}^2\to \mathbb{R}^2$ with compact support   contained in the set $\Omega$ (resp. $\overline{\Omega}$). Then, let us define
	\begin{eqnarray*}
	E(\Omega) &=& \{ u\in L^2(\Omega,\mathbb{R}^2):\divv u \in\,L^2(\Omega,\mathbb{R}^2)\},\\
	\mathcal{V}&=&\big\{  u\in C_0(\Omega,\mathbb{R}^2): \divv u=0\big\} ,\\
	\rH&=& \mbox{the closure of $\mathcal{V}$ in } L^2(\Omega,\mathbb{R}^2),\\
	\rH_0^1(\Omega,\mathbb{R}^2)&=& \mbox{the closure of $\mathcal{D}(\Omega,\mathbb{R}^2)
		$ in } \rH^1(\Omega,\mathbb{R}^2),\\
	\ve&=& \mbox{the closure of $\mathcal{V}$ in } \rH^1_0(\Omega,\mathbb{R}^2).
	\end{eqnarray*}
	The inner products in all  $L^2$ spaces will be denoted by $\langle \cdot,\cdot\rangle$. The space $E(\Omega)$ is a Hilbert space with a scalar product \begin{equation}\label{Temam_1.13}
	\lb u,v\rb_{E(\Omega)}:=\langle u,v\rangle +\langle \divv u, \divv v\rangle.
	\end{equation}
	We endow the set $\rH$ with the inner product $\rangle \cdot,\cdot\rangle$ and the norm   $\left\vert \cdot\right\vert_{\rH} $ induced by $L^2$.
	
	The space $\rH$ can also be characterized in the following way, see
	\cite[Theorem I.1.4]{Temam_2001},
	\[
	\rH=\{ u \in E(\Omega): \divv u=0 \mbox{ and } u\cdot \nu_{\lvert_{\partial \Omega}} =0\}.
	\]
	Let us denote by $\Pi :L^2(\Omega,\mathbb{R}^2) \rightarrow \mathrm{H}$
	the orthogonal projection called usually the Leray-Helmholtz projection. It is known,
	see \cite[Remark I.1.6]{Temam_2001} that $\Pi$ maps continuously the Sobolev space $\rH^{1}$ into itself.
	
	Observe that $\Omega$ is a Poincar\'e
	domain, \textit{i.e.}   there exists a constant $\lambda_1>0$ such that  the following Poincar\'{e} inequality is satisfied
	\begin{equation}
	\lambda_1\int_{\Omega}\varphi^{2}(x)\,\;dx\leq\int_{\Omega}|\nabla
	\varphi(x)|^{2}\,\;dx,\ \ \ \varphi\in \rH^1_0({\Omega}).
	\label{ineq-Poincare}
	\end{equation}
	Because of this  the norms on the space $\rV$ induced by $\lvert \cdot \rvert_\rH^1= \lvert \cdot \rvert_{L^2} +\lvert \nabla \cdot \rvert_{L^2}$ and by $\lvert \nabla \cdot \rvert_{L^2}$ are equivalent.
	%
	Since the space $\mathrm{V} $ is  densely and continuously
	embedded into $\mathrm{H}$, by   identifying   $\mathrm{H}$
	with its dual $\mathrm{H}^\prime$, we have the following embedding
	\begin{equation}
	\label{eqn:Gelfanf}
	\mathrm{V} \hookrightarrow \mathrm{H}\cong\mathrm{H}^\prime \hookrightarrow
	\mathrm{V}^\prime.
	\end{equation}
	Let us observe here  that, in particular,  the spaces $\mathrm{V}$, $\mathrm{H}$ and
	$\mathrm{V}^\prime$ form a Gelfand triple.
	
	We will denote by  $| \cdot |_{\ve^\ast}$   and  $\left\langle \cdot,\cdot\right\rangle $ the norm in
	$\ve^\ast$ and  the
	duality pairing between $\rV$ and $\ve^\ast$, respectively.
	
	Now, define the bilinear
	form $a:\mathrm{V}\times \mathrm{V} \to \mathbb{R}$ by setting
	\begin{equation}
	\label{form-a}
	a(u,v):=\langle \nabla u,\nabla v\rangle, \quad u,v \in \mathrm{V}.
	\end{equation}
	It is well-known that this bilinear map is
	$\mathrm{V}$-continuous and $\mathrm{V}$-coercive, i.e. there exist some $C_0, C_1>0$ such that
	\[ C_0 \lvert u \rvert^2_{\rV}\le \vert a(u,u) \vert
	\leq C \vert u \vert_\ve^2,\ \ \ \ u \in\,\mathrm{V}.\]
	Hence, by the Riesz Lemma and Lax-Milgram theorem, see for instance
	Temam \cite[Theorem II.2.1]{Temam_2001},
	there exists a unique	isomorphism
	$\mathcal{A}:\mathrm{V} \to \mathrm{V}^\prime$,  such that
	$a(u,v)=\lb \mathcal{A}u,v\rb$, for $u, v \in \mathrm{V}$.
	
	Next we  define an unbounded linear operator
	$\mathrm{A}$ in $\mathrm{H}$ as follows
	\begin{equation}
	\label{def-A} \left\{
	\begin{array}{ll}
	D(\mathrm{A}) &= \{u \in \mathrm{V}: \mathcal{A}u \in
	\mathrm{H}\},\\
	&\vspace{.1cm} \\
	\mathrm{A}u&=\mathcal{A}u, \, u \in
	D(\mathrm{A}).
	\end{array}
	\right.
	\end{equation}
	
	Under our assumption on $\Omega$, $\rA$ and $D(\mathrm{A})$ can be characterized as follows
	\begin{equation}
	\label{eqn:4.3} \left\{
	\begin{array}{l}
	D(\mathrm{A}) = \mathrm{V} \cap \rH^2=\mathrm{H} \cap \rH^1_0\cap \rH^2,\\
	\mathrm{A}u=-\mathrm{P}\Delta u, \quad u\in D(\mathrm{A}).
	\end{array}
	\right.
	\end{equation}
	It is also well-known, see \cite[Section]{Brz+Cer+F_2015} and references therein,  that $\mathrm{A}$ is a positive self adjoint operator in
	$\mathrm{H}$ and
	\[D(\rA^{\alpha/2})=[\rH,D(\rA)]_{\frac{\alpha}{2}},\] where
	$[\cdot,\cdot]_\frac{\alpha}{2}$ is the complex interpolation functor  of
	order $\frac{\alpha}{2}$. Furthermore,   for $\alpha  \in (0, \frac12)$
	\begin{equation}\label{eqn-domains}
	D(\rA^{\alpha/2})= \rH \cap \rH^{\alpha}(\Omega,\mathbb{R}^2).
	\end{equation}
	In particular, $\mathrm{V}=D(\mathrm{A}^{1/2})$ and $\lvert \rA^\frac12 u \rvert^2= \lvert \nabla u \rvert^2=: \lVert u \rVert^2$ for $u \in \ve$.
	The  equality \eqref{eqn-domains}  leads to  the following result which was proved in \cite[Proposition 2.1]{Brz+Cer+F_2015}.
	\begin{proposition}\label{prop-Leray-fractional} Assume that $\alpha  \in (0, \frac12)$.
		Then the  Leray-Helmholtz projection $\Pi$  is a well defined and continuous map  from
		$\rH^{\alpha}$ into  $ D(\rA^{\alpha/2})$.
	\end{proposition}
	Let us  finally recall that the projection $\Pi$ extends to a
	bounded linear projection in the space $L^q$, for any $q \in\,(1,\infty)$.

	Now, consider the trilinear form $b$ on $V\times V\times V$ given by
	\[
	b(u,v,w)=\sum_{i,j=1}^{2}\int_{\Omega}u_{i}{\frac{\partial v_{j}%
		}{\partial x_{i}}}w_{j}\,\,{d}x,\quad u,v,w\in \rV.
	\]
	Indeed, $b$ is a continuous trilinear form such that
	\begin{equation}
	\label{eqn:b01}
	b(u,v,w)=-b(u,w,v),
	\quad  \, u\in \mathrm{V}, v, w\in
	\rH_0^{1}(\Omega,\mathbb{R}^2),
	\end{equation}
	or a proof see  for instance \cite[Lemma 1.3, p.163]{Temam_2001} .
	
	Define next the bilinear map $B:\rV\times \rV\rightarrow \rV^{\ast}$ by setting
	\[
	{}_{\ve^\ast}\left\langle B(u,v),w\right\rangle_{\ve}=b(u,v,w),\quad u,v,w\in \rV,
	\]
	and the homogenous polynomial of second degree $B:\rV \rightarrow \ve^{\ast}$ by
	
	\[
	B(u)=B(u,u),\; u\in \rV.
	\]
	Let us observe that if $v \in\,D(\rA)$, then $B(u,v) \in\,H$ and the following identity is a direct consequence of \eqref{eqn:b01}.
	\begin{equation}
	\label{eqn-B02}
	{}_{\ve^\ast}\lb  \rB(u,v),v \rb_{\ve} =0,\;\;  u,v\in \rV.
	\end{equation}

	The restriction of the map $\rB$ to the space $D(\rA)\times D(\rA)$ has also the following representation
	\begin{equation}
	\label{eqn-B-using-LH}
	\rB(u,v)= \Pi( u\cdot \nabla v), \;\; u,v\in D(\rA),
	\end{equation}
	where $\Pi$ is the Leray-Helmholtz projection operator and $u\nabla v=\sum_{j=1}^2 u^jD_jv \in L^2(\Omega,\mathbb{R}^2)$.

	\subsection{The Laplacian for the director field $d$}
	 Throughout this section  we still denote $L^2(\Omega; \mathbb{R}^3)$ and $\rH^k(\Omega; \mathbb{R}^3)$, $k \in \mathbb{N},$ by $L^2$ and $\rH^1$, respectively.
	We aim in this subsection to introduce the Laplacian for the director $d:\Omega \to \mathbb{R}^3$ with the Neumann boundary conditions. We can do this by mimicking the way we define the Stokes operator $\rA$.  We define the bilinear map $\hat{a}: \rH^1\times \rH^1 \to \mathbb{R}$ by
	\begin{equation*}
	\hat{a}(d,n) = \int_\Omega (\nabla\; d  \nabla  n ) \;dx=\sum_{i=1}^3\sum_{j=1}^{3}\int_{\Omega}(\partial_i d_j \partial_in_j)\,dx, \; d, n \in \rH^1.
	\end{equation*}
	It is clear that $\hat{a}$ is continuous, and hemce, by Riesz representation lemma, there exists a unique bounded linear operator $\hat{\mathcal{A}}: \rH^1 \to (\rH^1)^\ast$ such that ${}_{(\rH^1)^\ast} \langle\hat{\mathcal{A}} d, n\rangle_{\rH^1} =\hat{a}(d,n) $, for $d,n\in \rH^1$.  	Next, we  define an unbounded linear operator
	$\hA$ in $L^2$ as follows
	\begin{equation}
	\label{def-ANeum} \begin{cases}
	D(\hA) = \{d \in \rH^1: \hA d \in L^2 \}\\
	\hA =\hat{\mathcal{A}}d, \, d \in
	D(\hA).
	\end{cases}
	\end{equation}
	Under our assumption on $\Omega$, it is known, see for instance \cite[Section2, p. 65]{Temam_1997} that $\hA$ and $D(\hA)$ can be characterized by
	\begin{equation}\label{eqn-def-A}
	\begin{cases}
	D(\hA)&:=\{d\in \mathbb{H}^{2}: \frac{\partial d}{\partial \nu }=0 \;\text{ on }\; \partial \Omega \},\\
	\hA d&:=-\Delta d,\quad d\in D(\hA),
	\end{cases}
	\end{equation}
	where $\nu=(\nu_1,\nu_2,\nu_3)$ is the unit outward normal vector field on $\partial \Omega $ and $\frac{\partial d}{\partial \nu}$
	is the directional derivative of $d$ in the direction $\nu$.
	
	Let us recall that the operator $\hA$ is self-adjoint and nonnegative and
	$D\left(\hA^{1/2}\right)$ when endowed with the
	graph norm coincides with $\rH^1$. Moreover, the operator $(I+\hA)^{-1}$ is
	compact. Furthermore, if we denote
	\[
	\hV:= D(\hA^{1/2}),
	\]
	the $(\hV,L^2, \ve^\ast)$ is a Gelfand triple and
	\[
	\fourIdx{}{\hV^\ast}{}{\hV}{\lb d_1,d_2 \rb}= \int_{\Omega} d_1(x) d_2(x)\, \;dx
	\]

	\subsection{An abstract formulation of problem \eqref{1a}}
	With the notations we have introduced above, we can now rewrite  problem \eqref{1b-0} as an abstract equations. In fact by projecting the first equation in \eqref{1b} onto $\rH$ we obtain the following system
	\begin{equation}
	\begin{cases}
	\dot{v}+\rA v =-B(v,v)-\Pi \left(\Div [ \nabla d\odot \nabla  d]\right)+ \Pi f,\\
	\partial_t d+\hA d = |\nabla d|^{2}d-v\cdot \nabla d-\phi^\prime(d)+(\phi^\prime(d),d)d+d\times g,\\
	\lvert d \rvert=1\\
	(v,d)(0)=(v_{0},d_{0}). \label{1b}
	\end{cases}
	\end{equation}

	\section{The existence and the  uniqueness of a regular solution to Problem \eqref{1b}} \label{Sec:ExistRegularSol}

	Throughout the whole section, we fix a  map $\phi:\mathbb{R}^3\to \mathbb{R}^3$ satisfying the following set of conditions.
	\begin{assume}\label{Assum:EnergyPotential}
		The map $\phi:\mathbb{R}^3 \to \mathbb{R}^3$ is of class $C^2$ and there  exist constants  $M_0>0$, $M_1>0$ and $M_2>0$ such that for all $n, d \in \mathbb{R}^3$
		\begin{align}
		\lvert \phi^\prime(n) \rvert \le & M_0(1+ \lvert n\rvert),\label{Eq:LInearGrowthPhiprime}\\
		\lvert \phi^{\prime \prime} (n)-\phi^{\prime \prime} (d)\rvert \le & M_1\lvert n-d\rvert,
		\label{Eq:LipSchitzPhibis}\\
		\lvert \phi^{\prime \prime} (n)\rvert\le & M_2.\label{Eq:BoundednessPhibis}
		\end{align}
	\end{assume}
\begin{example}
	Let $H\in \mathbb{R}^3$ be a constant vector. Then the anisotropy energy potential $\phi$  due to the action of a magnetic or electric is defined by
	$$ \phi(d)= \frac12[ \lvert H\rvert^2- (d\cdot H)^2], d\in \mathbb{R}^3.$$
	This potential $\phi$ satisfies the Assumption \ref{Assum:EnergyPotential}.  In this case $H$ represents a constant magnetic or electric field applied to the nematic liquid crystal.
	
	\noindent Another mathematical example which satisfies Assumption \ref{Assum:EnergyPotential} is the potential defined by
	$$ \phi(d)= \frac12 \lvert d -\xi \rvert^2, d\in \mathbb{R}^3,$$
	where  $\xi \in \mathbb{R}^3$ is a fixed constant vector.
\end{example}

Next, we consider the problem \eqref{1b} on a finite time horizon $[0,T]$. Throughout the paper we put
\[
\mathcal{M}=\{ d: \Omega \to \mathbb{R}{^3}: \lvert d(x) \rvert=1 \;\;\mathrm{Leb}\text{-a.e.} \},
\]

For this section, we
impose the following set of conditions on the data.
	\begin{assume}
		We assume that $({f},g)\in L^2(0,T; L^2 \times D(\hA^{1/2}))$ and $(v_0, d_0)\in \ve\times (D(\hA)\cap \mathcal{M})$.
	\end{assume}
	Under this assumption we will  prove in this section that Problem \eqref{1b} has a unique  local regular solution.
	 Before stating and proving this result
	 we define what we mean by a  maximal local regular  solution.
	
	\begin{Def}\label{Def:Local-Regular-Sol}
		Let $T_0\in (0, T]$. A function $(v,d): [0,T_0] \to \ve\times D(\hA)$ is called   a local regular  solution to Problem   \eqref{1b} with initial data $(v_0,d_0)$ iff
		\begin{enumerate}[(1)]
			\item  $(v,d)\in C([0,T_0];\ve\times D(\hA))\cap L^2(0,T_0; D(\rA)\times D(\hA^{3/2}))$,
			\item for all $t \in [0,T_0]$ the integral equations
			\begin{align}
			v(t) =& v_0 +\int_0^t[-\rA v(s) -B(v(s),v(s)) - \Pi( \Div[\nabla d(s)\odot \nabla  d(s)] )  ] ds + \int_0^t f(s) ds,\label{eq:LocVelo}\\
			d(t)=& d_0 +\int_0^t[-\hA d(s)+\lvert \nabla d(s)\rvert^2 d(s) -v(s)\cdot \nabla d(s) -\phi^\prime(d(s)) +(\phi^\prime(d(s) )\cdot d(s) ) d(s)   ] ds \nonumber \\
			&\qquad \qquad + \int_0^t (d(s)\times g(s)) ds,\label{eq:LocDir}
			\end{align}
			hold in $\rH$ and $D(\hA^{1/2})$, respectively.
			\item For all $t\in [0,T_0]$ $d(t)\in \mathcal{M}$.
			\item  \label{Item4-DefRegularSol}  $(\partial_tv, \partial_td)\in L^2(0,T_0; \rH\times D(\hA^{1/2}))$.
		\end{enumerate}
		Throughout this paper we will denote by $((v,d);T_0)$ a local regular solution defined on $[0,T_0]$.
	\end{Def}
		We also need the definition of a maximal local regular solution.
	\begin{Def}\label{Def:Maximal-Sol}
		
		A pair  $((v,d);T_0)$ with $T_0\in (0,T)$ and $(v,d):[0,T_0) \to \ve \times D(\hA)$ is called a maximal local regular solution to \eqref{1b} with initial data $(v(0), d(0))=(v_0,d_0)$ if
		\begin{enumerate}
			\item 	$((v,d);T_0)$ defined on $[0,T_0)$ is a local regular solution to \eqref{1b},
			\item for any other local regular solution $((\tilde{v}, \tilde{d}); \tilde{T}_0)$  we have
			\begin{equation}
				\tilde{T}_0 \le T_0  \text{ and }  (v,d)_{\lvert_{[0,\tilde{T}_0)}}= (\tilde{v},\tilde{d}).\notag
			\end{equation}
		\end{enumerate}
	\end{Def}
	We state the following important remark.
	\begin{Rem}
		Let
			\begin{align}
		F(v,d):=&  -B(v, v) -\Pi(\Div [\nabla d \odot \nabla  d ]),\label{Eq:NonlinVelo} \\
		\tilde{G}(v,d)=& \lvert \nabla d \rvert^2 d -v\cdot \nabla d -\phi^\prime(d) + (\phi^\prime(d) \cdot d) d, \\
		G(v,d):=&\lvert \nabla d \rvert^2 d -v\cdot \nabla d -\phi^\prime(d) + (\phi^\prime(d) \cdot d) d + d\times g.\label{Eq:NonlinDir}
		\end{align}
		Then,  the condition  \eqref{Item4-DefRegularSol} of Definition \ref{Def:Local-Regular-Sol} is equivalent to the following
		\begin{equation}
		(F(v,d), G(v,d)) \in  L^2(0,T_0; \rH\times D(\hA^{1/2})  ).\notag
		\end{equation}
	
		We will see in Lemma \ref{Lem:RighthandSideinL2VxH1} that if  $(v,d)\in C([0,T_0];\ve\times D(\hA))\cap L^2(0,T_0; D(A)\times D(\hA^{3/2}))$, then  $(F(v,d), G(v,d)) \in L^2(0,T_0; H\times D(\hA^{1/2})) $.
	\end{Rem}
	With the definitions and remark in mind we are now ready to formulate our first result.
	\begin{Thm}\label{th}
		Let $R_1>0$, $R_2>0$ and $g\in L^2(0,T; D(\hA^{1/2}))$. Then, there exist $T_1(g)$ and $T_2(R_1,R_2)>0$ such that the following holds.
		
		\noindent If  $(v_0, d_0)\in \ve\times (D(\hA)\cap \mathcal{M})$ and  $(f,g)\in L^2(0,T; L^2\times D(\hA^{1/2}))$ are such that
		\begin{equation} \label{Eq:DatainBoundedSet}
		 \left[\lvert v_0\rvert^2_{V} + \lvert d_0\rvert^2_{D(\hA)}\right]^\frac12 \le R_1 \text{ and }
		\left[\int_0^T \left( \lvert f(s)  \rvert^2_{L^2} + \lvert g(s) \rvert^2_{D(\hA^{1/2})} \right)\;ds \right]^\frac12 \le R_2,
		\end{equation}
		then the problem \eqref{1b} has a  local regular solution $((v,d);T_0)$ with $T_0=T_1(g)\wedge T_2(R_1,R_2)$
		
		Moreover, if  $((u,n);T_0)$ is another  local regular  solution, then
		$$(u(t),n(t))=(v(t),d(t)) \text{ for all } t\in [0,T_0] . $$
	\end{Thm}
	
	In order to prove this theorem we shall introduce the following spaces
	\begin{align}
	\mathbf{X}^1_T=& C([0,T]; \ve  ) \cap L^2(0,T; D(\rA)  ),\notag\\
	\mathbf{X}^2_T=& C([0,T]; D(\hA) ) \cap L^2(0,T; D(\hA^\frac32)  ),\textbf{}\notag\\
	\mathbf{Y}^1_T= & L^2(0,T; \rH  ) \notag\\
	\mathbf{Y}^2_T=& L^2(0,T; D(\hA^\frac12))\notag
	\end{align}
	We also set
	\begin{align}
	\mathbf{X}_T=& \mathbf{X}^1_{T} \times \mathbf{X}^2_{T},\notag\\
	\mathbf{Y}_T=&\mathbf{Y}^1_T \times \mathbf{Y}^2_T.\notag
	\end{align}
	Let $(v,n)\in \mathbf{X}_T$ and consider the following decoupled linear problem
	\begin{equation}\notag
	\begin{pmatrix}
	\partial_t {u}\\
	\partial_t {d}
	\end{pmatrix}
	+ \begin{pmatrix}
	\rA u \\
	\hA d
	\end{pmatrix}
	=
	\begin{pmatrix}
	F(v,n)+f\\
	G(v,n)+d\times g
	\end{pmatrix}.
	\end{equation}
	Before continuing further, let us recall the following result.
	\begin{Lem}\label{Lem:ExistenceStokes+Heat}
		If  $(u_0,d_0)\in \ve \times D(\hA)$, $(\mathfrak{f}, \mathfrak{g})\in L^2(0,T; \rH\times D(\hA^\frac12) )$ and $T>0$, then  the problem
		\begin{equation}\label{Eq:SystemStokesheat}
		\begin{pmatrix}
		\partial_t {u}\\
		\partial_t{d}
		\end{pmatrix}
		+ \begin{pmatrix}
		\rA u \\
		\hA d
		\end{pmatrix}
		=
		\begin{pmatrix}
		\mathfrak{f}\\
		\mathfrak{g}
		\end{pmatrix}
		\end{equation}
		has a unique strong solution $(u,d)\in \mathbf{X}_T$. Moreover, there exists a constant $C>0$, independent of $T$, such that
		\begin{equation}\notag
		\lvert (u,d) \rvert^2_{\mathbf{X}_T} \le C \lvert (u_0,d_0) \rvert^2_{\ve\times D(\hA)} + C \lvert (\mathfrak{f}, \mathfrak{g}) \rvert^2_{L^2(0,T; \rH \times D(\hA^\frac12))}.
		\end{equation}
	\end{Lem}
	Now we state the following lemma whose proof will be given in the appendix.
	\begin{Lem}\label{Lem:RighthandSideinL2VxH1}
		There exists a constant $C_0>0$, independent of $T$, such that for all $v_i\in \mathbf{X}^1_T, d_i\in \mathbf{X}^2_T$, $i=1,2$, the following inequalities  hold
		\begin{align}
		\lvert F(v_1,n_1)-F(v_2, n_2) \rvert^2_{L^2(0,T; \rH)} \le & C_0 T^\frac12 \lvert (v_1,n_1) - (v_2,n_2)\rvert^2_{\mathbf{X}_T} [\lvert (v_1,n_1) \rvert^2_{\mathbf{X}_T}+ \lvert (v_2,n_2)\rvert^2_{\mathbf{X}_T}]\label{Eq:FPT-ForcingVelo} \\
		\lvert \tilde{G}(v_1, n_1) - \tilde{G}(v_2,n_2) \rvert^2_{L^2(0,T; D(\hA^\frac12))}\le  & C_0 (T\vee T^\frac12) \lvert (v_1,n_1) - (v_2,n_2)\rvert^2_{\mathbf{X}_T} \Bigl[ 1+\sum_{ i=1}^{2} \lvert (v_i, n_i) \rvert^{6} _{\mathbf{X}_T} \Bigr].   \label{Eq:FPT-ForcingOptDir-1}\\
		\lvert n_1\times g - n_2 \times g \rvert^2_{L^2(0,T; \rH^1)}\le & C_0  \lvert n_1-n_2 \rvert^2_{\mathbf{X}_T^2} \lvert g \rvert^2_{L^2(0,T; \rH^1)}. \label{Eq:FPT-ForcingOptDir-2}
		\end{align}
	\end{Lem}

	Now, we will give the proof of Theorem \ref{th}.
	\begin{proof}[Proof of Theorem \ref{th}]
		Let $\Psi:\mathbf{X}_T \to \mathbf{X}_T$ be the map defined as follows.
		If $(v,n)\in \mathbf{X}_T $, then $\Psi(v,n)=(u,d)$ iff  $(u,d)$ is the unique regular  solution to \eqref{Eq:SystemStokesheat} with right hand side of the form
		\begin{equation}\label{Eq:nonlinear-linearized}
		\begin{pmatrix}
		\mathfrak{f}\\
		\mathfrak{g}
		\end{pmatrix}
		=\begin{pmatrix}
		F(v,n) + f \\
		\tilde{G}(v,n) +n\times g
		\end{pmatrix}.
		\end{equation}
		Let us observe that by  Lemma \ref{Lem:RighthandSideinL2VxH1}
		the term $(\mathfrak{f}, \mathfrak{g}) $ defined  in \eqref{Eq:nonlinear-linearized}  belongs to $L^2(0,T; \rH\times D(\hA^\frac12))$. Hence, by Lemma \ref{Lem:ExistenceStokes+Heat} $(u,d)\in \mathbf{X}_T$, and so the map $\Psi$ is well-defined.
		
		Now, let $R_1>0$, $R_2>0$, $(f,g)\in L^2(0,T; H\times D(\hA^{1/2}))$  and $(v_0, d_0)\in \ve\times (D(\hA)\cap \mathcal{M})$.
		Let
		\begin{equation}
		\mathbf{K}_{R_1,R_2}= \Bigl\{ (v,n)\in \mathbf{X}_T:  \lvert (v, n) \rvert^2_{\mathbf{X}_T}\le R_1^2 + R_2^2   \Bigr\}.
		\end{equation}
		Let $(v_i,n_i)\in \mathbf{K}_{R_1,R_2}$ and $(u_i,d_i)=\Psi(v_i,n_i)$, $i=1,2$. Put  $(u,d)=(u_1-u_2, d_1-d_2)$. Then, it is easy to check that $(u,d)$ solves the following problem
		\begin{equation}\label{Eq:SystemStokesheat-1}
		\begin{pmatrix}
		\partial_t {u}\\
		\partial_t{d}
		\end{pmatrix}
		+ \begin{pmatrix}
		\rA u \\
		\hA d
		\end{pmatrix}
		=
		\begin{pmatrix}
		F(v_1,n_1)-F(v_2,n_2)\\
		\tilde{G}(v_1,n_1)-\tilde{G}(v_1,n_2)+ (n_1-n_2)\times g
		\end{pmatrix}
		.
		\end{equation}
		Hence, by Lemma \ref{Lem:ExistenceStokes+Heat} there exists a constant $C>0$, independent of $T$, such that 	
		\begin{equation}\notag
		\begin{split}
		\lvert (u,d)\rvert^2_{\mathbf{X}_T}\le & C \bigl(\lvert 	F(v_1,n_1)-F(v_2,n_2)\rvert^2_{L^2(0,T;\rH)} +  \lvert 	 \tilde{G}(v_1,n_1)-\tilde{G}(v_2,n_2)\rvert^2_{L^2(0,T;\rH^1)} \\& + \lvert (n_1-n_2)\times g\rvert^2_{L^2(0,T;\rH^1)}  \bigr).
		\end{split}
		\end{equation}
		Then, by plugging \eqref{Eq:FPT-ForcingVelo}, \eqref{Eq:FPT-ForcingOptDir-1} and  \eqref{Eq:FPT-ForcingOptDir-2} in the above inequality and performing elementary calculations imply that there exists a constant $C_1>0$, independent of $T$, such that for all $(v_i,n_i)\in \mathbf{K}_{R_1,R_2}$, $i=1,2$,
		\begin{equation}\notag
		\begin{split}
		\lvert \Psi(v_1,n_1)-\Psi(v_2,n_2) \rvert^2_{\mathbf{X}_T}\le C_1 \lvert (v_1,n_1)-(v_2,n_2) \rvert^2_{\mathbf{X}_T}\Biggl( \Bigl[1+R_1^{6} +R_2^{6}  \Bigr](T\vee T^\frac12) + \int_0^T \lvert g(s)\rvert^2_{\rH^1} ds  \Biggr).
		\end{split}
		\end{equation}
		Since $g\in L^2(0,T; \rH^1) $, for any $\eps>0$ there exists $T_1\in (0,T)$ such that
		\begin{equation}\notag
		C_1 \int_0^{T_1} \lvert g(s)\rvert^2_{\rH^1}\le \eps.
		\end{equation}
		Next, we choose a number $T_2$ such that
		\begin{equation}\notag
		C_1 \Bigl[1+R_1^{6} +R_2^{6}  \Bigr] (T_2 \vee T_2^\frac12) \le \frac14.
		\end{equation}
		Hence, by choosing $\eps=\frac14$ and setting $T_0=T_1\wedge T_2$  we infer that
		for all $(v_i,n_i)\in \mathbf{K}_{R_1,R_2}$, $i=1,2$,
		\begin{equation}\notag
		\begin{split}
		\lvert \Psi(v_1,n_1)-\Psi(v_2,n_2) \rvert^2_{\mathbf{X}_T}\le  \frac12.
		\end{split}
		\end{equation}
		Hence, $\Psi$ has a unique fixed point $(u,d)\in \mathbf{X}_{T_0}$ satisfying
		\begin{equation}\label{Eq:SystemStokesheat-A}
		\begin{pmatrix}
		\partial_t{u}\\
		\partial_t {d}
		\end{pmatrix}
		+ \begin{pmatrix}
		\rA u \\
		\hA d
		\end{pmatrix}
		=
		\begin{pmatrix}
		F(u,d) + f \\
		G(u,d) +d\times g
		\end{pmatrix}.
		\end{equation}
		
		Thus, in order to prove Theorem \ref{th} it remains to prove that
		\begin{equation}\label{Eq:tu5}
		d(t) \in \mathcal{M} \text{ for all } t\in [0,T_0].
		\end{equation}
		For this purpose, let
		$$z(t) =\lvert d(t) \rvert^2-1,\; t\in [0,T_0].$$
	We  recall that there exists a constant $C>0$ such that for all $n\in \rH^3$
	\begin{equation}\label{Eq:InterpolationofH52}
	\lvert n \rvert_{\rh^{\frac52}}\le  C \lvert n \vert^\frac12_{\rh^2} \lvert n \rvert^\frac12_{\rH^3}.
	\end{equation}
Hence, since $d\in \mathbf{X}^2_{T_0}:=C([0,T_0]; D(\hA))  \cap L^2(0,T; D(\hA^\frac32))$, $D(\hA^{\theta})\subset \rh^{2\theta}$ and $\rH^{2\theta}, \theta>\frac12$ is an algebra,  by using the interpolation inequality \eqref{Eq:InterpolationofH52} we easily show that
\begin{equation}\label{Eq:Regularity-of-z-1}
z\in C([0,T_0]; \rh^2) \cap L^2(0,T_0; \rh^{\frac52}).
\end{equation}
Also, since $(u,d)\in X_{T_0}$ we infer from Lemma \ref{Lem:RighthandSideinL2VxH1} that
\begin{equation}\label{Eq:TimeDrivative-of-d}
\partial_t d= -\hA d + \tilde{G}(u,d)+ d\times g \in L^2(0,T_0;\rH^1).
\end{equation}
Using this and $d\in \mathbf{X}^2_{T_0}$ we easily prove that
\begin{equation}\label{Eq:Regularity-of-z-2}
\partial_tz \in L^2(0,T_0; \rH^1).
\end{equation}
Now we will claim that $z$ satisfies the weak form of the following problem
\begin{equation}\label{Eq:ViscousTransport-z}
\begin{cases}
\partial_{t}z-\Delta z +u\cdot \nabla z=2|\nabla d|^{2}z-2(\phi^\prime(d).d)z,\\
{\frac{\partial z}{\partial \nu}}{\Big\lvert_{\partial\Omega}}=0,\\
z(0)=0.
\end{cases}
\end{equation}
Towards this end let $\varphi\in H^{1}(\Omega;\mathbb{R})$ and fix $t\in [0,T_0]$.
		Since $d\in C([0,T_0]; D(\hA))  \cap L^2(0,T; D(\hA^\frac32))$ and  $D(\hA)\subset L^\infty$ we easily prove that $\varphi d \in C([0,T_0]; \rH^1)\subset L^2(0,T_0; \rH^1)$.
		Also, since $(u,d)\in X_{T_0}$ we infer from Lemma \ref{Lem:RighthandSideinL2VxH1} that
	$$\partial_t d= -\hA d + \tilde{G}(u,d)+ d\times g \in L^2(0,T_0;\rH^1).$$
		Hence, in view of the Lions-Magenes lemma, see \cite[Lemma III.1.2]{Temam_2001},  we have
		\begin{align}
		\frac{1}{2}\frac{d}{dt}\int_{\Omega}\varphi(x)|d(t,x))|^{2} \;dx= & \langle \partial_t d(t), \varphi d(t)\rangle \nonumber \\
		=&  -\int_{\Omega}\varphi(x)\rA d(t,x)\cdot d(t,x) \;dx-\int_{\Omega}\varphi(x) [u(t,x)\cdot \nabla d(t,x)]\cdot d(t,x) \;dx \nonumber \\& + \int_{\Omega}\varphi(x) |\nabla d(t,x)|^{2}|d(t,x)|^{2}\;dx  -\int_{\Omega}\varphi(|d(t,x)|^{2}-1)(\phi^\prime(d(t,x))\cdot d(t,x))\;dx, \label{Eq:Weakformz-1}
		\end{align}
		where we used the fact that $\varphi d \perp_{\mathbb{R}^3} d\times g$.
		Since $d(t)\in D(\hA)$ and $\varphi d(t)\in \rH^1$  for all $t \in [0,T_0]$, by using \cite[Equation (2.6)]{ZB+BG+TJ}  we infer that
		\[
		-\int_\Omega 	\hA d(t,x)\cdot \varphi(x)  d(t,x) \;dx= - \int_\Omega (\nabla d(t,x) ) (\nabla[\varphi(x)d(t,x)]) \;dx.
		\]
		Thus, straightforward calculation yields
		\[
		-\int_\Omega 	\hA d(t,x)\cdot \varphi(x)  d(t,x) \;dx= - \int_\Omega \varphi(x) \lvert \nabla d(t,x) \rvert^2  \;dx- \frac12\int_\Omega \varphi(x) \nabla \lvert d(t,x) \rvert^2 \;dx.
		\]
		Hence, recalling the definition of $z$ and using the last identity in \eqref{Eq:Weakformz-1} implies
		\begin{equation}\label{Eq:ProofSphere-1}
		\begin{split}
		\frac{1}{2}\int_{\Omega}\partial_{t}z(t,x) \varphi(x) \;dx+\frac{1}{2}\int_{\Omega}\nabla z(t,x)\nabla\varphi(x) \;dx+\frac{1}{2}\int_{\Omega} u(t,x)\nabla z(t,x)  \varphi(x)\; \;dx\\
		=\int_{\Omega}|\nabla d(t,x)|^{2}z(t,x) \varphi(x) \;dx-\int_{\Omega} z(t,x)(\phi^\prime(d(t,x))\cdot d(t,x))\varphi(x)\;\;dx.
		\end{split}
		\end{equation}
		This is exactly the weak form of \eqref{Eq:ViscousTransport-z}.
		
		By Proposition \ref{tu4} $z$ is the unique solution of \eqref{Eq:ViscousTransport-z} and satisfies
		\begin{equation}\notag
		\sup_{0\le t\le T}|z(t)|^{2}_{L^{2}}+\int_{0}^{T}|\nabla z(t)|^{2}dt\le
		|z(0)|^{2}_{L^{2}}e^{c\int_{0}^{T}\left[|\nabla d|^{4}_{L^{4}}+(1+|d|^{2}_{H^{2}})\right]dt}.
	\end{equation}
	Since $z(0)=0$, we infer that
		\begin{equation}
		\sup_{0\le t\le T}|z(t)|^{2}_{L^{2}}=\sup_{0\le t \le T}\int_{\Omega}(|d|^{2}-1)^{2}\;dx=0,\notag
		\end{equation}
		which implies that $d(t) \in \mathcal{M}$ for all $t\in [0,T_0]$.
			This completes the proof of \eqref{Eq:tu5}. This also completes the proof of Theorem \ref{th}.
	\end{proof}

	\section{The existence and the uniqueness  of a maximal local regular solution to \eqref{1b} } \label{Sec:LocRegSolSmallENERG}\
	The aim of this section is to prove that Problem \eqref{1b} has a unique maximal local regular solution when the initial data has small energy. The main result of the section is Proposition \ref{Prop:LocalSolwithSmallEnergy} and it is a generalization of
	\cite[Lemma 5.2]{LLW}. Before proceeding to a precise statement and a detailed proof of the result let us introduce few notations. For $R>0$ and $(u,n)\in \rH\times \rh^1$ we set
	\begin{equation}\label{eqn-energy-loc}
	\mathcal{E}_{R}(u,n):= \frac12\sup_{x\in \Omega} \int_{B(x, 2R)} \left(\lvert u(y)\rvert^2 + \lvert \nabla n(y)  \rvert^2 + 2 \phi(n(y))  \right) dy,
	\end{equation}
	and
	\begin{equation}\label{eqn-energy-glob}
	\mathcal{E}(u,n):=\frac12 (\lvert u\rvert^2_{\rh} + \lvert \nabla n\rvert^2_{L^2}  ) + \int_\Omega \phi(n(x)) \;dx.
	\end{equation}
		We also  recall the following important lemma, see \cite[Lemma 3.1\& 3.2]{Struwe}.
	\begin{Lem}\label{Lem:Struwe}
		There exist $c_1>0$ and $r_0>0$  such that for every $h\in L^\infty(0,T; L^2)\cap L^2(0,T; \rH^1)$ we have
		\begin{equation}\label{Eq:Ladyzhenskaya-Struwe}
		\int_0^T \lvert h(t) \rvert^4_{L^4}dt \le c_1 \left(\sup_{(t,x) \in [0,T]\times \Omega  } \int_{B(x,r_0)} \lvert h(t,y) \rvert^2 dy \right) \left(\int_0^T \lvert \nabla h(t)\rvert^2_{L^2} dt + \frac{1}{r_0^2} \int_0^T \lvert h(t)\rvert^2_{L^2} dt   \right).
		\end{equation}
	\end{Lem}
	\begin{Rem}
		Let $r_0>0$ be as in Lemma \ref{Lem:Struwe}.
		In view of \ref{Eq:Ladyzhenskaya-Struwe} and \cite[Theorem 3.4]{Simader}, we infer  that there exists $c_2>0$ such that for all  $h\in L^\infty(0,T;\rH^1)\cap L^2(0,T; D(\hA))$ we have
		\begin{equation}\label{Eq:Ladyzhenskaya-Struwe-2}
		\int_0^T \lvert \nabla  h(t) \rvert^4_{L^4} \le c_2 \left(\sup_{(t,x) \in [0,T]\times \Omega  } \int_{B(x,r_0)} \lvert \nabla h(t,y) \rvert^2 dy \right) \left(\int_0^T \lvert \Delta h(t)\rvert^2_{L^2} dt + \frac{1}{r_0^2} \int_0^T \lvert \nabla h(t)\rvert^2_{L^2} dt   \right).
		\end{equation}
	\end{Rem}
	We state and prove the following important result.
	\begin{Prop}\label{Prop:LocalSolwithSmallEnergy}
		There exist a constant $\varepsilon_0>0$ and  a function 	
		\[\theta_0: (0,\eps_0)\times (0,\infty) \to (0,\infty),,
		\]
		which is non-increasing w.r.t. the  second variable and nondecreasing  w.r.t. the first one,
		such that the following holds: \\
		Let $r_0>0$ be as in Lemma \ref{Lem:Struwe}.
		Let $(f,g)\in L^2(0,T; \rh\times D(\hA^{1/2}))$, $(v_0, d_0)\in\ve\times D(\hA) $ and $R_0\in (0, r_0]$  are  such that
		\begin{equation}\label{Eq:AssumSmallInitialNRJ}
		\mathcal{E}_{2R_0}(v_0,d_0)<  \eps_0^2.
		\end{equation}
		Then,   there exists a unique maximal local regular solution $((v,d);T_0)$ to problem \eqref{1b} satisfying
		\begin{align}\label{eqn-t_o}
		T_0 \ge  \frac{R^2_0}{(R_0^\frac12 + 1)^4} \theta_0(\eps_1,E_0),\\
		\sup_{0\le t\le T_0 } \mathcal{E}_{R_0}(v(t),d(t)) \le 2\eps_1^2,\label{eqn-small estimates}
		\end{align}
		where $E_0:= \mathcal{E}(v_0,d_0)$ and  	$\eps_1^2=\mathcal{E}_{2R_0}(v_0,d_0)$.
		
	\end{Prop}
	
	\begin{Rem}
	In this theorem, 	the length $T_0$  is not the length of the existence interval but the length  of the existence interval as long as   the condition \eqref{eqn-small estimates} is satisfied.
	
	Note also that \eqref{eqn-small estimates} is equivalent to
	\begin{equation*}\label{eqn-small estimates-B}
	\frac12 \sup_{0\le t\le T_0 } \sup_{x\in \Omega}\int_{B(x,R_0)} \left(\lvert v(t,y)\rvert^2 + \lvert \nabla d(t,y)  \rvert^2 + 2 \phi(d(t,y))  \right) dy\le 2\eps_1^2.
	\end{equation*}
	\end{Rem}

	In order to prove the above proposition  we need several  results. For $n\in \mathbb{R}^3$
	\begin{equation*}
	\alpha(n)= \phi^\prime(n)\cdot n.
	\end{equation*}
	We state and prove  the following elementary results.
	\begin{Clm}\label{Claim:perp}
		Let $u\in \rH$,  $n\in D(\hA)\cap \mathcal{M}$ and $m\in C([0,T_\ast); (D(\hA)\cap \mathcal{M}))$ such that $\partial_t m \in L^2(0,T_\ast; L^2)$. Then,
		\begin{align}
		&\langle u\cdot \nabla n, \lvert \nabla n\rvert^2 n  -\phi^\prime(n) +\alpha(n)n\rangle=0,\notag\\
		&	\langle \partial_t m,  \lvert \nabla m \rvert^2 m -\alpha(m) m \rangle=0.\notag
		\end{align}
	\end{Clm}
	\begin{proof}
		Let us fix $u\in \ve$, $n\in D(\hA)$ and $m: [0,T_\ast)\to  D(\hA)\cap \mathcal{M}$ satisfying the assumptions of Claim \ref{Claim:perp}. Then, since  $\Div u=0$ and the fact $n(x) \in \mathbb{S}^2$ $x$-a.e. we get
		\begin{equation*}
	\begin{split}
		\langle u\cdot \nabla n, \lvert \nabla n\rvert^2 n  -\phi^\prime(n) +\alpha(n)n\rangle = \frac12 \int_\Omega  u(x)\cdot \nabla \lvert n(x) \rvert^2_{\mathbb{R}^3}( \lvert \nabla n(x) \rvert^2 + \alpha(n(x)) ) \;dx \\- \int_\Omega u(x)\cdot \nabla \phi(n(x)) \;dx =0.
	\end{split}
		\end{equation*}
		Since $ m\in C([0,T_\ast); D(\hA)\cap \mathcal{M}) $ we infer that for all $t\in [0,T_\ast)$
		\begin{align*}
		\langle \partial_t m(t),  \lvert \nabla m(t) \rvert^2 m(t) -\alpha(m(t)) m(t) \rangle= \frac12 \int_\Omega \partial_t \lvert m (t,x)\rvert^2_{\mathbb{R}^3} ( \lvert \nabla m(t,x) \rvert^2 - \alpha(m(t,x)) ) \;dx  =0,
		\end{align*}
		which completes the proof of the claim.
	\end{proof}
	
	We also recall the following result, see \cite[Eq. (2.10)]{ZB+EH+PR-SPDE_2019}.
	\begin{Clm}\label{Claim:Dissip}
		For any $u\in \rh \cap L^4$ and $n \in D(\hA)$
		\begin{equation}\label{Eq:Dissip}
		-\langle\Div (\nabla n \odot \nabla n ), u\rangle + \langle u\cdot \nabla n, \Delta n \rangle =0.
		\end{equation}
	\end{Clm}
	
	By Theorem \ref{th}, Problem \eqref{1b} has  a unique maximal local regular solution $(v,d) \in X_{T_\ast} $ provided that $(v_0,d_0)\in \ve\times D(\hA)$ and $(f,g)\in L^2(0,T; L^2 \times D(\hA^{1/2}))$.
	Hereafter, we fix such a maximal local regular solution and we set,
	\begin{equation}\notag
	\mathcal{E}(v(t),d(t)) = \frac12\left(\lvert v(t) \rvert^2_{L^2} + \lvert \nabla d(t) \rvert^2_{L^2}     \right) + \int_\Omega \phi(d(t,x)) \;dx,\text{ for $t \in [0,T_\ast]$,}.
	\end{equation}
	
	We then prove the following important a' priori  estimates.
	\begin{Lem}
			Let $(v_0,d_0)\in \ve \times D(\hA)$, $(f,g)\in L^2(0,T;L^2\times \rH^1)$ and  $(v,d) \in X_{T_\ast} $ be a maximal local regular solution to the problem \eqref{1b}.
		Then  for all  $s,t\in [0,T_\ast] $ with $s\le t$, the following inequality holds
		\begin{equation}\label{Eq:NRJ-Inequality}
		\mathcal{E}(v(t), d(t) ) +\frac12 \int_s^t\left( \lvert \nabla v(r) \rvert^2_{L^2} +  \lvert R(d(r) ) \rvert^2_{L^2} \right) dr \le \mathcal{E}(v(s), d(s)) + \frac12 \int_s^t \left(\lvert f(r) \rvert^2_{\rH^{-1}}+ \lvert g(r) \rvert^2_{L^2}  \right) dr,
		\end{equation}
		where, for $ n \in D(\hA)$, we put
		\begin{align*}
		&\alpha(n) = \phi^\prime(n)\cdot n,\\
		&R(n)= \Delta n + \lvert \nabla n \rvert^2 n -\phi^\prime(n) + \alpha(n).
		\end{align*}

	\end{Lem}
	\begin{proof}
		Since the maximal smooth solution $(v,d)$ satisfies part (1)  of Definition \ref{Def:Local-Regular-Sol}, it is not difficult to show that $( v(t),R(d(t)))\in L^2(0,T_\ast; D(\rA) \times D(\hA^\frac12))$. We also have  $(\partial_t v,\partial_t d) \in  L^2(0,T_\ast; \rH \times D(\hA^\frac12))$ because $(v,d)$ satisfies. Hence, by applying the  Lions-Magenes Lemma (\cite[Lemma III.1.2]{Temam_2001}) and the Claims \ref{Claim:perp} and \ref{Claim:Dissip} we infer that for all $t\in [0,T_\ast)$
		\begin{align*}
		-\langle \partial_t d(t), R(d(t)) \rangle - \langle  \partial_t v(t), \Delta v\rangle  =& \frac{d}{dt} \mathcal{E}(v(t),d(t)) + \lvert \nabla v(t) \rvert^2_{L^2} + \lvert R(d(t) ) \rvert^2_{L^2} \\
		= & \langle  f(t), v(t) \rangle  +  \langle d(t)\times g(t) , R(d(t) ) \rangle.
		\end{align*}
		From the Cauchy-Schwarz and the Young inequalities and,  the fact $d(t)\in \mathcal{M}, t\in [0, T_\ast)$, which is part (4) of Definition \ref{Def:Local-Regular-Sol}, we infer that there exists a constant $C>0$ such that
		\begin{align*}
		\frac{d}{dt} \mathcal{E}(v(t),d(t) ) + \lvert \nabla v(t) \rvert^2_{L^2} + \lvert R(d(t) ) \rvert^2_{L^2} \le C \vert f(t) \rvert_{\rH^{-1}} \lvert \nabla  v(t)  \rvert_{L^2} + C \vert d(t)\times g(t)  \rvert_{L^2}
		\lvert R(d(t) ) \rvert_{L^2}\\
		\le \frac12 \left(\lvert \nabla v(t) \rvert^2_{L^2} + \lvert R(d(t) ) \rvert^2_{L^2}\right) + \frac12 \left( \lvert f(t) \rvert^2_{\rH^{-1}}+ \lvert g(t) \rvert^2_{L^2}\right).
		\end{align*}
		Absorbing the first term on the right hand side of the last inequality into the left hand side and integrating over $[s,t] \subset [0,T_\ast)$ completes the proof of the lemma.
	\end{proof}

	For any $\eps>0$ and $R>0$ we define the time
	\begin{equation}\label{Eq:DefStoppingTime}
	T(\eps, R)= \inf\left\{t\in [0,T_\ast):  \mathcal{E}_{R}(v(t),d(t)) > 2 \eps^2 \right \}\wedge T_\ast.
	 \end{equation}
	 \begin{Rem}
	 	Let $\eps>0$ and $R>0$. Then, for any $t\in [0,T(\eps, R)]$
	 	\begin{equation}\label{Eq:BelowStoppingTime}
	 	 \mathcal{E}_{R}(v(t),d(t)) \le  2 \eps^2.
	 	\end{equation}
	 \end{Rem}
We state and prove the following lemma.
	\begin{Lem}\label{Lem:EstimateofDeltaD+L4norms}
			Let $(v_0,d_0)\in \ve \times D(\hA)$ and $(f,g)\in L^2(0,T;L^2\times \rH^1)$. 
		There exist  $\eps_0>0$ and $K_0>0$ such that  if $(v,d) \in X_{T_\ast} $ is a maximal local regular solution to the problem \eqref{1b}, then for all	$\eps\in (0, \eps_0)$, $R\in (0,r_0]$, where $r_0>0$ is the constant from Lemma  \ref{Lem:Struwe}, and for all $t\in [0,T(\eps, R)]$
		\begin{align}
		\int_0^{t} \lvert \Delta d(r) \rvert^2_{L^2} dr \le K_0\left[ E_0  + \frac12 \int_0^{t}\left(\lvert f (s) \rvert^2_{\rH^{-1}}+ \lvert g(s) \rvert^2_{L^2}\right) ds + \left(1+\frac{2\eps^2}{R^2}\right)t \right],\label{Eq:EstDeltad}\\
		\int_0^{t } \left(\lvert v(s) \rvert^4_{L^4} + \lvert \nabla d (s) \rvert^4_{L^4}    \right) ds \le K_0\eps^2\left[ E_0  + \frac12 \int_0^{t}\left(\lvert f (s) \rvert^2_{\rH^{-1}}+ \lvert g(s) \rvert^2_{L^2}\right) ds + \left(1+\frac{2\eps^2}{R^2}\right)t \right].\label{Eq:EstL4normsofU+NablaD}
		\end{align}
	\end{Lem}
	\begin{proof}
	Let 	$r_0>0$ be the constant from Lemma  \ref{Lem:Struwe}, $R\in [0,r_0]$ and  $\eps\in (0,\eps_0)$ where $\eps_0$ is number to be chosen later.
	We set
	$$E_0= \mathcal{E}(v_0,d_0).$$
		Since $\phi$ is twice continuously differentiable and the 2-sphere $\mathbb{S}^2$ is compact, we can and will assume throughout that for some constant $M>0$
		\begin{equation}
		2 \lvert \phi^\prime(n) - \alpha(n) n \rvert^2 \le M, \; n \in \mathbb{S}^2.\notag
		\end{equation}
		
		From this observation we infer that for all $t\in [0,T(\eps, R)]$
		\begin{align*}
		\int_0^{t} \lvert \Delta d \rvert^2_{L^2} ds \le 2 \int_0^{t} \lvert \Delta d - \phi^\prime(d) + \alpha(d) d \rvert^2_{L^2} ds + M t\\
		\le 4  \int_0^t \lvert R(d) \rvert^2_{L^2}  ds +  4 \int_0^{t} \lvert \nabla d \rvert^4_{L^4} ds + M t.
		\end{align*}
	The last line of the above inequalities, \eqref{Eq:Ladyzhenskaya-Struwe-2}, \eqref{Eq:NRJ-Inequality} and \eqref{Eq:BelowStoppingTime}  imply that
		\begin{align}
			\int_0^{t} \lvert \Delta d \rvert^2_{L^2} ds\le &  4 c_2 \left(\sup_{(s,x) \in [0,t]\times \Omega  } \int_{B(x,R)} \lvert \nabla d(s,y) \rvert^2 dy \right) \left(\int_0^t \lvert \Delta d(t)\rvert^2_{L^2} dt + \frac{1}{R^2} \int_0^t \lvert \nabla d(s)\rvert^2_{L^2} ds   \right)\nonumber \\
			& + 4  \int_0^t \lvert R(d) \rvert^2_{L^2}  ds \nonumber \\
			\le & 4 \left[E_0 + \frac12  \int_0^{t}\left(\lvert f (s) \rvert^2_{\rH^{-1}}+ \lvert g(s) \rvert^2_{L^2}\right) ds  \right]  + 8 c_2 \eps^2 \left(\int_0^{t} \lvert \Delta d \rvert^2_{L^2}ds  + \frac{2 \eps^2 }{R^2}  t \right)+M t.\notag
		\end{align}
		Now choosing $\eps_0>0$ so that $1-8c_2\eps_0^2\ge \frac12$, we infer that $1-8c_2\eps^2>\frac12$ and for all $t\in [0,T(\eps, R)]$
		\begin{equation}
		\begin{split}
		\int_0^{t} \lvert \Delta d \rvert^2_{L^2} ds \le 8 \left[E_0 + \frac12  \int_0^{t}\left(\lvert f (s) \rvert^2_{\rH^{-1}}+ \lvert g(s) \rvert^2_{L^2}\right) ds  \right]  \\
		+ 16 c_2 \left(\frac{2 \eps^2 }{R^2} +M\right)t,\notag
		\end{split}
		\end{equation}
		which completes the proof of \eqref{Eq:EstDeltad}.
		
		We now proceed to the proof of \eqref{Eq:EstL4normsofU+NablaD}. For this we observe that by Lemma \ref{Lem:Struwe} and  \eqref{Eq:EstDeltad} we infer that for all $t\in [0,T(\eps, R)]$
		\begin{align}
		\int_0^{t} \lvert \nabla d \rvert^4_{L^4} ds \le 4 c_2\eps^2 \left( \int_0^{t} \lvert \Delta d \rvert^2_{L^2}  ds + \frac{1}{R^2} \int_0^{t} \lvert \nabla d \rvert^2_{L^2}  ds \right)\nonumber \\
		\le 4 c_2\eps^2 \left(K_0\left[ E_0  + \frac12 \int_0^{t}\left(\lvert f (s) \rvert^2_{\rH^{-1}}+ \lvert g(s) \rvert^2_{L^2}\right) ds + \left(1+\frac{2\eps^2}{R^2}\right)t \right] +  \frac{2\eps^2}{R^2} t \right).\label{Est:L4NablaD}
		\end{align}
		In a similar way,  we prove that for all $t\in [0,T(\eps, R)]$
		\begin{align}
		\int_0^{t} \lvert v \rvert^4_{L^4} ds \le 4 c_2\eps^2 \left( \int_0^{t} \lvert \nabla v \rvert^2_{L^2}  ds + \frac{1}{r_0^2} \int_0^{t} \lvert v \rvert^2_{L^2}  ds \right)\nonumber \\
		\le 4 c_2\eps^2 \left(K_0\left[ E_0  + \frac12 \int_0^{t}\left(\lvert f (s) \rvert^2_{\rH^{-1}}+ \lvert g(s) \rvert^2_{L^2}\right) ds  \right] +  \frac{2\eps^2}{R^2} t \right),\label{Eq:EstL4u}
		\end{align}
		which altogether with \eqref{Est:L4NablaD} imply \eqref{Eq:EstL4normsofU+NablaD}.
	\end{proof}
	We will need the following estimates which  will be proved in the Appendix \ref{App:EstHighOrder}.
	\begin{Clm}\label{Clm:EstimateHigherorder}
		There exists a constant $K_1>0$ such that for all $v\in D(A)$ and $n \in D(\hA^{3/2})$ we have
		\begin{align}
	&	\lvert \langle A v, -\Pi[\Div(\nabla n \odot \nabla n) ]  \rangle \lvert \le \frac1{12}\left(\lvert Av \rvert^2_{L^2}+ \lvert \nabla \Delta n \rvert^2_{L^2}\right) + K_1 \lvert \nabla n \rvert^4_{L^4} \lvert \Delta n \rvert^2_{L^2},\label{Eq:EstHighOrd-1}\\
&	\lvert {}	_{(\rH^1)^\ast}\langle \hA^2 n,  (v\cdot \nabla n )\rangle_{\rH^1} \lvert \le  	\frac1{12} \left(\lvert \nabla \Delta n\rvert^2_{L^2} +\lvert \rA v\rvert^2_{L^2}\right)+ K_1  \lvert \nabla v\rvert^2_{L^2 } \rvert(\nabla n) \rvert^2_{L^4}+K_1[\lvert v \rvert^2_{L^4}+\lvert v\rvert^4_{L^4} ]\lvert \Delta n\rvert_{L^2},\label{Eq:EstHighOrd-2}\\
&		\lvert {}_{(\rH^1)^\ast}\langle \hA^2 n, (\lvert \nabla n \rvert^2 n ) \rangle_{\rH^1} \lvert \le  \frac1{12} \lvert \nabla \Delta n \rvert^2_{L^2} + K_1 \Big(  [\lvert \nabla n \rvert^4_{L^4} +\lvert \nabla n\rvert^2_{L^4} ]\lvert \Delta n  \rvert^2_{L^2}+  \lvert \nabla n  \rvert^4_{L^4}(\lvert \nabla n\rvert^2_{L^2} + \lvert \Delta n \rvert^2_{L^2})\Big).\label{Eq:EstHighOrd-3}
		\end{align}
	\end{Clm}
	We will also need the following results.
	\begin{Clm}\label{Clm:EstimateDeltaAnisotropy}
		There exists a constant $K_2>0$ such that for all  $n \in D(\hA^{3/2})$ we have
		\begin{align}
		\lvert {}_{ (\rH^1)^\ast }\langle \hA ^2 n, \alpha(n)n-\phi^\prime(n) \rangle_{\rH^1} \lvert \le \frac1{12} \lvert \nabla \Delta n \rvert^2_{L^2}+ K_2 \lvert \nabla n \rvert^2_{L^2}.\notag
		\end{align}
	\end{Clm}
	\begin{proof}
		Using the Cauchy-Schwarz inequality and the fact that there exists a constant $M>0$ such that
		\begin{equation*}
	\lvert \phi^\prime(n) \rvert + \lvert \phi^{\prime\prime}(n) \rvert \le M,  n \in \mathbb{S}^2,
		\end{equation*}
		we infer that there exists a constant $C>0$ such that
		\begin{align*}
		\lvert {}_{(\rH^1)^\ast}\langle \hA^2 n, \alpha(n)n-\phi^\prime(n) \rangle_{\rH^1} \lvert\le \lvert \nabla \Delta n \rvert_{L^2}\left( \lvert \nabla (\alpha(n)n) \rvert_{L^2}+ \lvert \phi^{\prime\prime}(n)\nabla n \rvert_{L^2} \right)\\
		\le C \lvert \nabla \Delta n \rvert_{L^2}\left( \lvert \lvert \nabla n\rvert \; \rvert\phi^\prime(n)\rvert\;  \lvert n\rvert\;  \rvert_{L^2}+\lvert \lvert \nabla n\rvert \; \lvert \phi^{\prime\prime}(n) \rvert\; \lvert n \rvert^2   \rvert_{L^2} + \lvert \phi^{\prime\prime}(n)\nabla n \rvert_{L^2} \right)\\
		\le C \lvert \nabla \Delta n \rvert_{L^2}\lvert \nabla n \rvert_{L^2}.
		\end{align*}
		We now complete the proof using the Young inequality in the last line.
	\end{proof}
Hereafter, we put for all $s,t \in [0,T]$ with $s\le T$
\begin{align}
	\Psi(s,t)=&\frac12 \int_s^t \left(\lvert f \rvert^2_{\rH^{-1}}+ \lvert g \rvert^2_{L^2} \right)dr \text{ and } \Psi(t)=\Psi(0,t), \label{Eq:DefPsi}\\
	\Xi(s,t)=&\frac12 \int_s^t \left(\lvert f \rvert^2_{L^2} + \lvert g \rvert^2_{\rH^1} \right)dr \text{ and } \Xi(t)=\Xi(0,t) \;\; \forall s\le t\in [0,T]. \label{Eq:DefXi}
\end{align}
\begin{Lem}\label{Lem:EstinHigherorderNorms}
	Let $(v_0,d_0)\in \ve \times D(\hA)$, $(f,g)\in L^2(0,T;L^2\times \rH^1)$ and  $(v,d) \in X_{T_\ast} $ be a maximal local regular solution to the problem \eqref{1b}.
	Let  us put
	\begin{equation}\label{Eq:DefSigma}
		\Sigma_0 =[E_0 + \lvert \Delta d_0\rvert^2_{L^2} + \lvert \nabla v_0\rvert^2_{L^2}+ \Psi(T)+ \Xi(T)+ (E_0+1)\Xi(T)]<\infty.
	\end{equation}
	Then, there exist constants  $K_3>0$ and $K_4>0$ such that for all $\tau\in [0,T_\ast]$ we have
	\begin{equation}\label{Eq:EstinHigherorderNorms}
		\begin{split}
			\sup_{0\le s\le \tau} \left(\lvert \nabla v(s) \rvert^2_{L^2} + \lvert \Delta d(s) \rvert^2_{L^2} \right) + 2 \int_0^{\tau} \left( \lvert \nabla \Delta d(s) \rvert^2_{L^2} + \lvert A v(s) \rvert^2_{L^2}\right)ds \\
			\le K_3\Sigma_0 e^{K_4\Xi(T)} e^{K_4 \int_0^{\tau}[\lvert \nabla d(r) \rvert^4_{L^4}+ \lvert v(r) \rvert^4_{L^4} +\lvert \nabla d(r) \rvert^2_{L^4}+ \lvert v(r)\rvert^2_{L^4} ]dr   }.
		\end{split}
	\end{equation}
\end{Lem}
\begin{proof}
	Throughout this proof $C>0$ will denote an universal constant which may change from one term to the other. Let $(v_0,d_0)\in \rH\times D(\hA)$ and $((v,d); T_\ast) $ be a  local regular solution to the problem \eqref{1b}.
	
	By part (1) and (4) of Definition \ref{Def:Local-Regular-Sol} we have $(v,d)\in L^2(0,T_\ast; D(\rA) \times D(\hA^\frac32))$ and $(\partial_t v, \partial d) \in L^2(0,T_\ast; \rH  \times D(\hA^\frac12)) \subset L^2(0,T_\ast; \ve^\ast \times D(\hA))$. Hence, by the Lions-Magenes Lemma (\cite[Lemma III.1.2]{Temam_2001}) we infer that
	\begin{align*}
		\frac12 \frac{d}{dt}\left( \lvert \Delta d(t) \rvert^2_{L^2} + \lvert \nabla v(t) \rvert^2_{L^2} \right)
		&=  \langle \partial_t \Delta d(t), \Delta d(t) \rangle+ \langle \partial_t \nabla v(t), \nabla v(t)\rangle\\
		= & {}_{\rH^{1}}\langle \partial_t  d(t), \hA^2 d(t)\rangle_{(\rH^{1})^\ast} +  (\partial_tv(t) , A v(t) ).
	\end{align*}
	Hence,
	\begin{align}\label{Eq:Derivativehigheroerdernorms}
		\frac12 \frac{d}{dt}\left( \lvert \Delta d(t) \rvert^2_{L^2} + \lvert \nabla v(t) \rvert^2_{L^2} \right) = & -\lvert \nabla \Delta d(t) \rvert^2_{L^2} - \lvert \rA v(t) \rvert^2_{L^2} + \langle f(t), Av(t) \rangle-\langle \nabla (g(t)\times d(t)), \nabla \Delta d(t)\rangle \nonumber \\
		&+ {}_{(\rH^{1})^\ast}\langle  \hA^2 d(t),\lvert \nabla d(t) \rvert^2 d(t)+\alpha(d(t))d(t)
		-\phi^\prime(d(t))-v(t)\cdot \nabla d(t)    \rangle_{\rH^{1}}\nonumber  \\
		&- \langle \Div(\nabla d(t) \odot \nabla d(t))+ v(t)\cdot \nabla v(t), A v(t)\rangle.
	\end{align}
	By using the Cauchy-Schwarz, the Young inequalities  and the Ladyzhenskaya inequality (\cite[Lemma III.3.3]{Temam_2001})
	we obtain
	\begin{align*}
		\langle v\cdot \nabla v, \rA v \rangle \le C  \lvert Av \rvert_{L^2} \lvert v \rvert_{L^4} \lvert \nabla v \rvert_{L^4} \\
		\le C \lvert \rA v \rvert^\frac32 \lvert v \rvert_{L^4} \lvert \nabla v \rvert^\frac12_{L^2}\\
		\le \frac14 \lvert \rA v \rvert^2+ C \lvert v \rvert^4_{L^4} \lvert \nabla v \rvert^2_{L^2}.
	\end{align*}
	

	\noindent Using the H\"older and Young inequalities, ,and the Sobolev embedding $\rH^1 \hookrightarrow L^4$ we also have
	\begin{align*}
		{}_{(\rH^{1})^\ast}\langle \hA^2 d, g\times d \rangle_{\rH^{1}} = & ( \nabla \Delta d , \nabla (g\times d) )\\
		\le & C \lvert \nabla \Delta d \rvert_{L^2} \left(\lvert \nabla g \rvert_{L^2}+ \lvert g \rvert_{L^4} \lvert \nabla d \rvert_{L^4} \right)\\
		\le & C \lvert \nabla \Delta d \rvert_{L^2} \left(\lvert \nabla g \rvert_{L^2}+ \lvert g \rvert_{L^4} \lvert \nabla d \rvert_{L^2}^\frac12 \lvert \Delta d \rvert_{L^2}^\frac12 \right)\\
		\le & \frac{1}{2} \lvert \nabla \Delta d \rvert^2_{L^2} + \frac12 \lvert g \rvert^2_{\rH^1} + C \lvert g\rvert^2_{\rH^1}\left(\lvert \nabla d \rvert^2_{L^2}+\lvert \Delta d \rvert^2_{L^2}\right).
	\end{align*}

	Plugging these estimates and the ones in Claims \ref{Clm:EstimateHigherorder}-\ref{Clm:EstimateDeltaAnisotropy} into \eqref{Eq:Derivativehigheroerdernorms} yield
	\begin{equation}\label{Eq:Derivativehigheroerdernorms-Fin}
		\begin{split}
			\frac12 \frac{d}{dt}\left( \lvert \Delta d\rvert^2_{L^2} + \lvert \nabla v \rvert^2_{L^2} \right) + \frac12 \left(\lvert \rA v \rvert^2_{L^2} + \lvert \nabla \Delta d \rvert^2_{L^2}\right) - C \lvert v \rvert^4_{L^4} \lvert \nabla v \rvert^2_{L^2} \\
			\le  C \left(1+
			\lvert \nabla d \rvert^4_{L^4} + \lvert g \rvert^2_{\rH^1}   \right) \lvert \Delta d \rvert^2_{L^2}
			+ \frac12 \left(\lvert f \rvert^2_{L^2}+ \lvert g \rvert^2_{\rH^1}+ C \lvert g \rvert^2_{\rH^1} \lvert \nabla d \rvert^2_{L^2}    \right)\\
			+ C \lvert \nabla d \rvert^2 \lvert \nabla d \rvert^4_{L^4}.
		\end{split}
	\end{equation}
	Hence,
	\begin{equation}\label{Eq:Derivativehigheroerdernorms-Fin-A}
		\begin{split}
			\frac12 \frac{d}{dt}\left( \lvert \Delta d(t) \rvert^2_{L^2} + \lvert \nabla v(t) \rvert^2_{L^2} \right)  \le  C \left(1+
			\lvert \nabla d(t)\rvert^4_{L^4} + \lvert g(t)\rvert^2_{\rH^1} + \lvert  v(t)\rvert^4_{L^4}  \right) \left( \lvert \Delta d(t) \rvert^2_{L^2} + \lvert \nabla v(t) \rvert^2_{L^2} \right) \\ + \frac12 \left(\lvert f(t) \rvert^2_{L^2}+ \lvert g(t) \rvert^2_{\rH^1}+ C \lvert g(t) \rvert^2_{\rH^1} \lvert \nabla d(t)\rvert^2_{L^2}    \right)+ C \lvert \nabla d(t) \rvert^2 \lvert \nabla d(t)\rvert^4_{L^4}.
		\end{split}
	\end{equation}
	Let us put
	\begin{equation*}
		\Theta(t):=e^{2 C \left( \int_0^t[
			\lvert \nabla d(r)\rvert^4_{L^4} + 	\lvert \nabla d(r)\rvert^4_{L^4} + \lvert g(r) \rvert^2_{\rH^1}  + \lvert v(r) \rvert^2_{L^4} + \lvert v(r) \rvert^4_{L^4}]\,dr  \right) }ds,\; t\in [0,T_\ast).
	\end{equation*}
	Thus, by the Gronwall Lemma, we obtain	
	\begin{equation}\label{Eq:Derivativehigheroerdernorms-Fin-B}
		\begin{split}
			& \lvert \Delta d (t) \rvert^2_{L^2} + \lvert \nabla u (t) \rvert^2_{L^2}   -  \left( \lvert \Delta d_0 \rvert^2_{L^2} + \lvert \nabla v_0 \rvert^2_{L^2} \right)  \Theta(t) \\
			&\le \left(\int_0^t  \frac12 \lvert f(s) \rvert^2_{L^2} ds+ \left(C \sup_{s\in [0,\tau]} \lvert \nabla d (s) \rvert^2_{L^2} +\frac12\right)  \int_0^t [\lvert g (s) \rvert^2_{\rH^1} +\lvert \nabla d(s) \rvert^4_{L^4}] ds \right) \Theta(t),
		\end{split}
	\end{equation}
	which along with the inequality \eqref{Eq:NRJ-Inequality} the fact $\theta \le e^\theta,\; \theta\ge 0$ implies that
	\begin{equation}\label{Eq:EstinHigherorderNorms-0}
		\begin{split}
			\sup_{0\le s\le \tau} \left(\lvert \nabla v(s) \rvert^2_{L^2} + \lvert \Delta d(s) \rvert^2_{L^2} \right)
			\le C \Sigma_0 e^{C \Xi(T)} e^{C \int_0^{\tau}[\lvert \nabla d(r) \rvert^4_{L^4}+ \lvert v(r) \rvert^4_{L^4}+\lvert \nabla d(r) \rvert^2_{L^4}+ \lvert v(r) \rvert^2_{L^4} ]dr   }.
		\end{split}
	\end{equation}
	
	Integrating \eqref{Eq:Derivativehigheroerdernorms-Fin}, using \eqref{Eq:EstinHigherorderNorms-0} and the fact $\theta \le e^\theta,\; \theta\ge 0$ yield
	\begin{align}
			\int_0^{\tau} \left(\lvert \rA v(s) \rvert^2_{L^2} + \lvert \nabla \Delta d(r) \rvert^2 \right) ds\le&    \sup_{s\in [0,\tau]} \lvert \Delta d(s) \rvert^2_{L^2}  C \int_0^{\tau}
			\lvert \nabla d(s)\rvert^4_{L^4} + \lvert g(s) \rvert^2_{\rH^1}  \,ds \nonumber  \\
			& + \int_0^\tau \left(\lvert f(s) \rvert^2_{L^2}+ \lvert g(s)  \rvert^2_{\rH^1}\right) ds +\lvert \nabla v_0\rvert^2_{L^2} +  \lvert \Delta d_0\rvert^2_{L^2} \nonumber \\
			& + C \sup_{s\in [0,\tau)}  \lvert \nabla d(s) \rvert^2_{L^2} \int_0^\tau \lvert g(s)\rvert^2_{\rH^1} ds \nonumber  \\
			\le & C \Sigma_0 e^{C \Xi(T)} e^{C \int_0^{\tau}[\lvert \nabla d(r) \rvert^4_{L^4}+ \lvert v(r) \rvert^4_{L^4} +\lvert \nabla d(r) \rvert^2_{L^4}+ \lvert v(r) \rvert^2_{L^4}]dr   } \label{Eq:EstinHigherorderNorms-1}
	\end{align}
	
	We easily infer from \eqref{Eq:EstinHigherorderNorms-0} and \eqref{Eq:EstinHigherorderNorms-1} that \eqref{Eq:EstinHigherorderNorms} holds. This completes the proof of the lemma.
\end{proof}
We have the following consequence of the above lemma.
\begin{Cc}\label{Cor:EstinHigherorderNorms-CC}
	Let $(v_0,d_0)\in \ve \times D(\hA)$, $(f,g)\in L^2(0,T;L^2\times \rH^1)$, $r_0$ and $\eps_0$ be as in Lemma  \ref{Lem:Struwe} and Lemma \ref{Lem:EstimateofDeltaD+L4norms}.
	Let   $(v,d) \in X_{T_\ast} $ be a maximal local regular solution to the problem \eqref{1b}.
	Let $\Sigma_0$ be defined as in \eqref{Eq:DefSigma}.
	
	Then, there exists a constant $K_3>0$ such that for any $R\in (0, r_0]$, $\eps\in (0, \eps_0)$  and $t\in [0,T(\eps,R)]$,
	\begin{equation}\label{Eq:EstinHigherorderNorms-CC}
		\begin{split}
			\left(\lvert \nabla v(t) \rvert^2_{L^2} + \lvert \Delta d(t) \rvert^2_{L^2} \right) + 2 \int_0^{t } \left( \lvert \nabla \Delta d \rvert^2_{L^2} + \lvert \rA v \rvert^2_{L^2}\right)ds
			\le K_3 \Sigma_0 e^{K_3 [\Xi(T)+\eps^2 (E_0 + \Psi(T)+(1+\frac{2\eps^2}{R^2}) t) ]}.
		\end{split}
	\end{equation}
\end{Cc}
\begin{proof}
	Let   $(v,d) \in X_{T_\ast} $ be a maximal local regular solution to the problem \eqref{1b} with initial data $(v_0,d_0)\in \ve\times D(\hA)$.
	Let us fix $R\in (0,r_0]$ and $\eps\in (0,\eps_0)$.   Since $(v,d)\in X_{T(\eps, R)}$ then we can apply Lemma \ref{Lem:EstinHigherorderNorms}  and infer that \eqref{Eq:EstinHigherorderNorms} holds for $\tau=T(\eps, R)$. Hence, in order to complete the proof of the corollary we need to estimate the exponential term in the right-hand side of \eqref{Eq:EstinHigherorderNorms}. For this purpose, we use   \eqref{Eq:EstL4normsofU+NablaD} and infer that there exists a universal constant $K>0$ such that for any $R\in (0,r_0]$, $\eps(0,\eps_0)$ and $t\in [0,T(\eps,R)]$
	\begin{equation}\label{Eq:EstExpo}
		e^{ K_4\int_0^t\left(
			\lvert \nabla d(s)\rvert^4_{L^4} + \lvert v(s) \rvert^4_{L^4} +\lvert \nabla d(r) \rvert^2_{L^4}+ \lvert v(r) \rvert^2_{L^4}  \right)ds }\le   K  e^{K[\Xi(T)+\eps^2 (E_0 + \Psi(T)+(1+\frac{2\eps^2}{R^2}) t) ] }.
	\end{equation}
	This completes the proof of the corollary.
\end{proof}

\begin{Cc}\label{Cor:SmallNRJBeforemaxTime}
	Let $(v_0,d_0)\in \ve \times D(\hA)$, $(f,g)\in L^2(0,T;L^2\times \rH^1)$ and  $(v,d) \in X_{T_\ast} $ be a maximal local regular solution to the problem \eqref{1b}.
	Let $r_0>0$ and $\eps_0>0$ be as in Lemma \ref{Lem:Struwe} and Lemma \ref{Lem:EstimateofDeltaD+L4norms}, respectively.  Then, for all $\eps\in (0,\eps_0)$ and $R\in (0,r_0]$ we have
	$$ T(\eps, R) < T_\ast .$$
\end{Cc}
\begin{proof}
	We argue by contradiction. Assume that there exists $\eps\in (0,\eps_0)$ and $R\in (0,r_0]$ such that $T(\eps, R) = T_\ast$. 	 	Let us put
	$$ R_2=\int_0^{T_\ast} (\lvert f(r) \rvert^2_{L^2}+\lvert g(r) \rvert^2_{\rH^1}) dr.$$ By Corollary   \ref{Cor:EstinHigherorderNorms-CC}  and \eqref{Eq:NRJ-Inequality} we infer that  there exists a constant $\tilde{K}_3>0$ such that for all $t\in [0,T_\ast)$
	\begin{equation*}
		\lvert (v(t), d(t) \rvert^2_{\ve\times D(\hA) }\le \tilde{K}_3.
	\end{equation*}
	Let $T_0= T_1(\tilde{K}_3, R_2) \wedge T_2(g)\wedge T_\ast >0$ be the time given by Theorem \ref{th}. Let $T_1= \frac{T_0}{2}$.  By Theorem \ref{th} the  problem \eqref{1b} with initial data $(v(T_1), d(T_1))$ has a unique local regular solution $(\tilde{v}(t), \tilde{d}(t) )$  defined on $[T_\ast- T_1, T_\ast-T_1 + T_0 ]$.
	We then define $(\bar{v}, \bar{d}):[0,T_\ast+ T_1]\to \ve\times D(\hA)$ by
	\begin{equation*}
		(\bar{v}(t),\bar{d}(t))=
		\begin{cases}
			(v(t), d(t)  ) \text{ if } t\in [0,T_\ast-T_1]\\
			(\tilde{v}(t), \tilde{d}(t) ) \text{ if } t\in [T_\ast-T_1, T_\ast + T_1].
		\end{cases}
	\end{equation*}
	It is easily seen that  $((\bar{v}, \bar{d});T_1)$ is a local regular solution to \eqref{1b} with initial data $(v_0, d_0)$ and time of existence $T_\ast+T_1>T_\ast$.
	This contradicts the fact that $((v,d);T_\ast)$, with $T_\ast=T(\eps, R)$,  is a maximal smooth solution  to \eqref{1b}. This completes the proof of the corollary.
\end{proof}

	We now state and prove a local energy inequality which will play an important role in the proof of  Proposition \ref{Prop:LocalSolwithSmallEnergy}.
		\begin{Lem}\label{Lem:EstPressure}
		Let $(v_0,d_0)\in \ve \times D(\hA)$, $(f,g)\in L^2(0,T;L^2\times \rH^1)$, $r_0$ and $\eps_0$ be as in Lemma \ref{Lem:EstimateofDeltaD+L4norms}. Also, let $R\in (0, r_0]$, $\eps\in (0, \eps_0)$ and $(v,d) \in X_{T_\ast} $ be a  local regular solution to the problem \eqref{1b}.
		Then, there exists a function $\mathrm{p}:[0,T_\ast )\to L^1$ such that $\mathrm{p}\in L^\frac43(0,T_\ast; L^4)$,  $\nabla \mathrm{p} \in L^\frac43(0,T_\ast ; L^\frac43 ) $ and
		\begin{equation}\label{Eq:pressure}
		v(t) + \int_0^t [v(s) \cdot \nabla v(s) + \nabla \mathrm{p}(s)]\;ds=v_0 + \int_0^t [\Delta v(s) - \nabla d(s)\Delta d(s) +f(s)]\; ds, t\in [0,T_\ast].
		\end{equation}
		Moreover, there exists a constant $K_5>0$, which may depend on the norms of $(v_0,d_0)\in \ve \times D(\hA)$ and $(f,g)\in L^2(0,T;L^2\times \rH^1)$, such that for any $t\in [0,T_\ast)$
		\begin{equation}
		\lvert \nabla \mathrm{p} \rvert_{L^\frac43(0,t;  L^\frac43 )} \le K_5 \eps^\frac12_1 \left[E_0  + \Psi(t) + (1+ \frac{2\eps^2}{R^2}) t\right]^\frac{3}{4}.\notag
		\end{equation}
	\end{Lem}
	\begin{proof}
		Let us fix $\eps\in (0,\eps_0)$ and $R\in (0,r_0]$. We fix $t\in [0,T_\ast)$. From \eqref{Eq:NRJ-Inequality} and \eqref{Eq:EstDeltad} we have $(v,d)\in C([0,t]; \rh\times \rH^1)\cap L^2(0,t; \ve\times D(\hA))$. Hence, one can apply \cite[Lemma 4.4]{LLW}  and infer that there exists function  $\mathrm{p}:[0,T_\ast)\to L^1$ such that $\mathrm{p}\in L^\frac43(0,T_\ast; L^4)$,  $\nabla \mathrm{p} \in L^\frac43(0,T_\ast ; L^\frac43 ) $ and the identity \eqref{Eq:pressure} holds. Moreover,
		\begin{align}\label{Eq:EstPressure-0}
		\lvert \nabla \mathrm{p} \rvert_{L^\frac43(0,t; L^\frac43 )}  \le \lvert f\rvert_{L^\frac43(0,t; L^\frac43 )} + \lvert v\cdot \nabla v\rvert_{L^\frac43(0,t; L^\frac43 )} + \lvert \nabla d\Delta d \rvert_{L^\frac43(0,t; L^\frac43 )}.
		\end{align}
		From the H\"older inequality, \eqref{Eq:NRJ-Inequality} and \eqref{Eq:EstL4normsofU+NablaD} we infer that
		\begin{align}
		\lvert v\cdot \nabla v\rvert_{L^\frac43(0,t; L^\frac43 )} \le \left(\int_0^{t}  \lvert v(r) \rvert^4_{L^4} dr \right)^\frac14 \left(\int_0^{t}\lvert \nabla v(r) \rvert^2_{L^2} dr \right)^\frac12\nonumber \\
		\le C \eps^\frac12 \left[E_0  + \Psi(t) + (1+ \frac{2\eps^2}{R^2}) t\right]^\frac{3}{4}.\label{Eq:EstPressure-1}
		\end{align}
		In a similar way,
		\begin{align}
	 \lvert \nabla d \Delta d \rvert_{L^\frac43(0,t; L^\frac43) } \le \left(\int_0^{t}  \lvert \nabla d(r)  \rvert^4_{L^4} dr \right)^\frac14 \left(\int_0^{t}\lvert \Delta d(r)  \rvert^2_{L^2} dr \right)^\frac12\nonumber \\
		\le C \eps^\frac12 \left[E_0  + \Psi(t) + (1+ \frac{2\eps^2}{R^2}) t\right]^\frac{3}{4}.\label{Eq:EstPressure-2}
		\end{align}
		Plugging \eqref{Eq:EstPressure-1} and \eqref{Eq:EstPressure-2} into \eqref{Eq:EstPressure-0} completes the proof of Lemma \ref{Lem:EstPressure}.
	\end{proof}
We now continue with some estimates of local energy. Hereafter, in order to save space we simply write $ \int_\mathcal{O} \Xi\, d\mu $ instead of  $ \int_\mathcal{O} \Xi(y)\, d\mu(y) $  for an integrable function $\Xi$ defined on a measure space $(\mathcal{O}, \mathcal{A}, \mu)$.
	\begin{Lem}
			Let $(v_0,d_0)\in \ve \times D(\hA)$, $(f,g)\in L^2(0,T;L^2\times \rH^1)$ and $(v,d) \in X_{T_\ast} $ be a maximal local regular solution to the problem \eqref{1b}.
			
		Let $\varphi \in C^\infty_c(\Omega; \mathbb{R})$ and put
		\begin{equation}
		\mathcal{E}_\varphi(v(t),d(t))=\frac12 \int_\Omega \varphi(x) \left( \lvert v(t,x) \rvert^2 + \lvert \nabla d(t,x) \rvert^2 + 2 \phi(d(t,x)) \right)\;dx, \;\; t\in [0,T_\ast) .\notag
		\end{equation}
		We also set
		$$ \mathrm{p}_\Omega(t)= \frac1{\lvert \Omega\rvert} \int_\Omega \mathrm{p}(t,x) \;dx, \; t\in [0,T_\ast).$$
		Then, for any $s,t\in [0,T_\ast)$ with $s\le t$ we have
		\begin{equation}\label{Eq:LocalNRJ-DU}
		\begin{split}
		\mathcal{E}_\varphi(v(t), d(t)) -	\mathcal{E}_\varphi(v(s), d(s)) +\int_s^t \int_\Omega \varphi \left(\lvert \nabla v \rvert^2 + \lvert R(d) \rvert^2 \right) \;dxdr \\
		\le
		\int_s^t \int_\Omega \lvert \nabla \varphi \rvert \Bigl( \frac12 \lvert \nabla d \rvert^2 \lvert v \rvert + \lvert \partial_t \rvert \lvert \nabla d \rvert+ \phi(d) \lvert v \rvert + \lvert v \rvert^3 \Bigr) \;dxdr\\
		+ \int_s^t \int_\Omega \lvert \nabla \varphi\rvert \Bigl(\lvert \nabla v \rvert \lvert v \rvert + \lvert \mathrm{p}- \mathrm{p}_\Omega \rvert \lvert v \rvert + \lvert  g \rvert \lvert \nabla d\rvert    \Bigr) \;dxdr\\
		+ \int_s^t \int_\Omega \lvert \varphi \rvert \Bigl(\lvert d \rvert \lvert \nabla g \rvert \lvert \nabla d \rvert + \lvert d  \rvert \lvert g \rvert \lvert \phi^\prime (d) \rvert  + \lvert f \rvert \lvert v \rvert \Bigr) \;dxdr.
		\end{split}
		\end{equation}
	\end{Lem}
	\begin{proof}
		We fix $\varphi \in C_c^\infty(\Omega;\mathbb{R})$ and $s\le t\in [0,T_\ast)$. We firstly observe that because of the fact $ d(t)\in \mathcal{M}$ for all $t\in [0,T_\ast)$  we have
		\begin{align}\label{Eq:pointwisePerp}
		(\partial_t d + v\cdot \nabla d)\cdot(\lvert \nabla d \rvert^2 d + \alpha(d) d)=0,
		\end{align}
		for all $t\in [0,T_\ast)$ and a.e. $x\in \Omega$.  Multiplying $\partial_t d + v\cdot \nabla d$  by $-\varphi R(d)$ in $L^2$ and using the pointwise orthogonality yields
		\begin{equation}\label{Eq:LocalNRJD-0}
		\begin{split}
		A+B:= & -\int_\Omega \varphi (\partial_td + v\cdot \nabla d)\cdot \Delta d  \;dx + \int_\Omega (\partial_t d + v\cdot \nabla d)\cdot  \phi^\prime(d)  \;dx\\
		= & -\int_\Omega \varphi \lvert R(d) \rvert^2 \;dx - \int_\Omega \varphi(d\times g)\cdot R(d) \;dx .
		\end{split}
		\end{equation}
		Using integration by parts and \cite[Eq. (4.16)]{LLW} we obtain
		\begin{align}
		A=& \frac12 \frac{d}{dt}\int_\Omega \lvert \nabla d \rvert^2 \varphi \;dx + \int_\Omega \partial_t d\cdot (\nabla d \nabla \varphi) \;dx -\int_\Omega \varphi (v\cdot \nabla d)\cdot  \Delta d \;dx\nonumber \\
		=& \frac12 \frac{d}{dt}\int_\Omega \lvert \nabla d \rvert^2 \varphi \;dx + \int_\Omega \partial_t d\cdot (\nabla d \nabla \varphi) \;dx - \int_\Omega \frac12 \lvert \nabla d\rvert^2 v\cdot \nabla \varphi \;dx + \int_\Omega \varphi (\nabla d\odot \nabla d)\nabla u  \;dx \nonumber \\
		&\qquad -\int_\Omega (v\cdot \nabla d)\cdot (\nabla d \nabla \varphi) \;dx.\label{Eq:LocalNRJD-1}
		\end{align}
		For the term $B$ it is easy to show that
		\begin{align}
		B= & \int_\Omega (\partial_t d \cdot  \phi^\prime(d) )\varphi \;dx + \int_\Omega \varphi (v\cdot \nabla d) \cdot \phi^\prime (d) \;dx\nonumber  \\
		=& \frac{d}{dt} \int_\Omega \phi(d) \varphi \;dx + \int_\Omega v\cdot \nabla \phi(d) \varphi \;dx\nonumber  \\
		=& \frac{d}{dt} \int_\Omega \phi(d) \varphi \;dx -\int_\Omega v\cdot \nabla \varphi \phi(d) \;dx.\label{Eq:LocalNRJD-2}
		\end{align}
		Note also that for all $t\in [0,T_\ast)$ and a.e. $x\in \Omega$.
		\begin{equation}
		(d\times g)\cdot (\lvert \nabla d\rvert^2 d + \alpha(d) d)=0.\notag
		\end{equation}
		Hence, using integration by parts and the Cauchy-Schwarz inequality we obtain
		\begin{align}
		-\int_\Omega \varphi (d\times g) \cdot R(d) \;dx= & -\int_\Omega \varphi (d\times g)\cdot (\Delta d- \phi^\prime(d) ) \;dx \nonumber \\
		\le& \int_\Omega \lvert \varphi \rvert  \left(  \lvert d \rvert   \lvert \nabla d\rvert  \nabla g \rvert  \right) \;dx + \int_\Omega \lvert g \rvert \nabla d \rvert \lvert \nabla \varphi \rvert \;dx + \int_\Omega \lvert g \rvert \lvert d \rvert \lvert \phi^\prime(d)\rvert \lvert \varphi \rvert \;dx . \label{Eq:LocalNRJD-3}
		\end{align}
		Plugging \eqref{Eq:LocalNRJD-1},  \eqref{Eq:LocalNRJD-2} and \eqref{Eq:LocalNRJD-3} into \eqref{Eq:LocalNRJD-0} yields
		\begin{equation}
		\begin{split}
	&	\frac12 \frac{d}{dt}\int_\Omega \left( \lvert \nabla d \rvert^2 + 2 \phi(d)\right)\varphi \;dx +\int_\Omega \varphi \lvert R(d) \rvert^2 \;dx \\
	&	\le   \int_\Omega \frac12 \lvert \nabla d\rvert^2 v\cdot \nabla \varphi \;dx+ \int_\Omega (v\cdot \nabla d)\cdot (\nabla d \nabla \varphi) \;dx-\int_\Omega \partial_t d\cdot (\nabla d \nabla \varphi) \;dx\\
		& \qquad - \int_\Omega \varphi (\nabla d\odot \nabla d)\nabla u  \;dx
		 + \int_\Omega \lvert \varphi \rvert  \left(   \lvert d \rvert  \lvert \nabla d\rvert  \nabla g \rvert  \right) \;dx + \int_\Omega \lvert g \rvert \nabla d \rvert \lvert \nabla \varphi \rvert \;dx \\
		&\qquad + \int_\Omega \lvert g \rvert \lvert d \rvert \lvert \phi^\prime(d)\rvert \lvert \varphi \rvert \;dx .
		\end{split}
		\end{equation}
		We can follow the same calculation in \cite{LLW} to derive the following local inequality for the velocity $v$
		\begin{equation}\notag
		\begin{split}
		\frac12 \frac{d}{dt} \int_\Omega \lvert v \rvert^2 \varphi \;dx + \int_\Omega \le \frac14 \int_\Omega \lvert v\rvert^2 v\cdot \nabla \varphi \;dx -\int_\Omega (\nabla u) v\cdot \nabla \varphi \;dx +
		\int_\Omega (\mathrm{p}-\mathrm{p}_\Omega) v\cdot \nabla \varphi \;dx \\
		+ \int_\Omega \varphi (\nabla d \odot \nabla d) \nabla u \;dx + \int_\Omega (  \nabla d) \cdot \nabla d\nabla \varphi \;dx + \int_\Omega \lvert f\rvert \lvert v \rvert \lvert \varphi \rvert \;dx.
		\end{split}
		\end{equation}
		Adding up the last inequalities side by side and using the Cauchy-Schwarz inequality and integrating over $[s,t]$ yield the sought estimate \eqref{Eq:LocalNRJ-DU}.
	\end{proof}
	The following lemma is also important for our analysis.
	\begin{Lem}\label{Lem:LocalNRJ-Ball}
			Let $(v_0,d_0)\in \ve \times D(\hA)$, $(f,g)\in L^2(0,T;L^2\times \rH^1)$, $r_0$ and $\eps_0$ be as in Lemma \ref{Lem:EstimateofDeltaD+L4norms}. Also, let $(v,d) \in X_{T_\ast} $ be a maximal local regular solution to the problem \eqref{1b}.
	
	 Then, there exists a constant $K_4>0$ such that for all $\eps\in (0,\eps_0)$, $R\in (0,r_0]$ and $t\in [0,T(\eps,R)] $
		\begin{equation}\label{Eq:EstNRJonBalls}
		\begin{split}
		\frac12	\int_{B(x,R)} \left( \lvert v(t) \rvert^2 + \lvert \nabla d (t) \rvert^2 + 2 \phi(d(t))\right) dy - \frac12 \int_{B(x,2R)} \left( \lvert v_0 \rvert^2 + \lvert \nabla d_0 \rvert^2 + 2\phi(d_0) \right) dy
		\\
		\le K_4 t^\frac14 \left(1+ E_0 + \frac12 \int_0^t [\lvert f \rvert^2_{\rH^{-1}}+ \lvert g \rvert^2_{L^2}] ds +(1+ \frac{2\eps^2}{R^2}) t \right)^{\frac{5}{4}}\left(  R^{-\frac12} (\eps^\frac32 + \eps^\frac12)+ \eps^\frac12 \right) .
		\end{split}
		\end{equation}
	\end{Lem}
	\begin{proof}
		Let us fix $\eps\in (0,\eps_0)$, $R\in (0,r_0]$, $x\in \Omega$  and $t\in [0,T(\eps,R)]$. We also fix $\varphi \in C^\infty_c(\Omega; [0,1])$ such that
		\begin{equation}\label{Eq:Condvarphi}
		\mathds{1}_{B(x,R)} \le \varphi \le \mathds{1}_{B(x,2R)} \text{ and } 	\lvert \nabla \varphi \rvert\le \frac{c_4}{R},
		\end{equation}
		for some constant $c_4>0$.
		For such particular $\varphi$ we will estimate each term on the right hand side of \eqref{Eq:LocalNRJ-DU}. To start this quest we observe that
		\begin{equation*}
		\left(\int_0^t \int_{B(x,2R)} \lvert \nabla \varphi \rvert \;dx ds\right)^\frac14 \le \frac{c_4}{R} \lvert B(x,2R) \rvert^\frac14 t^\frac14\le c_5 \left(\frac{t}{R^2}\right)^\frac14.
		\end{equation*}
		We will use this inequality below without further notice.
		
		Using the H\"older inequality and \eqref{Eq:EstL4normsofU+NablaD} we get
		\begin{align}
		\int_0^t \int_\Omega \frac{\lvert \nabla d\rvert^2}{2} \lvert v\rvert \lvert \nabla \varphi\rvert \;dx dr \le \left( \int_0^t \lvert \nabla d \rvert^4_{L^4}\right)^\frac14 \left( \int_0^t \lvert v \rvert^4_{L^4}\right)^\frac14  \left( \int_0^t \lvert \nabla \varphi \rvert^4_{L^4}\right)^\frac14\nonumber \\
		\le C \eps^\frac32  \left(\frac{t}{R^2}\right)^\frac14 \left[E_0  + \Psi(t) + (1+ \frac{2\eps^2}{R^2}) t\right]^\frac{3}{4}.\label{Eq:LocalNRj-Start}
		\end{align}
		In a similar way we get
		\begin{equation}\notag
		\int_0^t \int_\Omega \lvert \nabla \varphi \rvert \phi(d) \lvert v \rvert \;dx dr \le C \eps^\frac12   \left(\frac{t}{R^2}\right)^\frac14 \left[E_0  + \Psi(t) + (1+ \frac{2\eps^2}{R^2}) t\right]^\frac{3}{4}.
		\end{equation}
		Here we used the boundedness of $\phi$ on $\mathbb{S}^2$.
		Similarly,
		\begin{align*}
		\int_0^t \int_\Omega \lvert \nabla \varphi \rvert \lvert v \rvert^3 \;dx dr \le C \eps^\frac32  \left(\frac{t}{R^2}\right)^\frac14 \left[E_0  + \Psi(t) + (1+ \frac{2\eps^2}{R^2}) t\right]^\frac{3}{4}\\
		\int_0\int_\Omega \lvert \nabla \varphi \rvert \lvert g \rvert \lvert \nabla d \rvert \;dx dr \le C \eps^\frac12  \left(\frac{t}{R^2}\right)^\frac14 \left[E_0  + \Psi(t) + (1+ \frac{2\eps^2}{R^2}) t\right]^\frac{3}{4}.
		\end{align*}
		Using the H\"older inequality, \eqref{Eq:EstL4normsofU+NablaD} and \eqref{Eq:NRJ-Inequality}
		\begin{align*}
		\int_0^t \int_\Omega \lvert \nabla \varphi \rvert \lvert \nabla v \rvert \lvert v \rvert \;dx dr \le & c_5 \left(\frac{t}{R^2}\right)^\frac14 \left(\int_0^t \lvert v \rvert^4_{L^4} dr \right)^\frac14 \left(\int_0^t \lvert \nabla v \rvert^2_{L^2} dr \right)^\frac12\nonumber \\
		\le & C \eps^\frac12  \left(\frac{t}{R^2}\right)^\frac14 \left[E_0  + \Psi(t) + (1+ \frac{2\eps^2}{R^2}) t\right]^\frac{3}{4}.
		\end{align*}
		To deal with the term containing $\partial_t d$ we argue as follows
		\begin{align*}
		\int_t\int_\Omega \lvert \partial_t d \rvert \lvert \nabla d\rvert \lvert \nabla \varphi \rvert \;dx dr \le c_5 \left(\frac{t}{R^2}\right)^\frac14 \left(\int_0^t \lvert \nabla d \rvert^4_{L^4}  dr \right)^\frac14
		\left(\int_0^t \lvert \partial_t d \rvert^2_{L^2}  dr \right)^\frac12\\
		\le c_5 \left(\frac{t}{R^2}\right)^\frac14 \left(\int_0^t \lvert \nabla d \rvert^4_{L^4}  dr \right)^\frac14
		\left(\int_0^t \left[\lvert R(d) \rvert^2_{L^2} + \lvert g \rvert^2_{L^2}\right] dr \right)^\frac12.
		\end{align*}
		Now, using \eqref{Eq:EstL4normsofU+NablaD} and \eqref{Eq:NRJ-Inequality}  we obtain
		\begin{align*}
		\int_t\int_\Omega \lvert \partial_t  d\rvert \lvert \nabla d\rvert \lvert \nabla \varphi \rvert \;dx dr \le
		C \eps^\frac12  \left(\frac{t}{R^2}\right)^\frac14 \left[E_0  + \Psi(t) + (1+ \frac{2\eps^2}{R^2}) t\right]^\frac{3}{4}.
		\end{align*}
		We now deal with term containing the pressure $\mathrm{p}$. First by the
		Using the H\"older and Poincar\'e inequalities and the estimates \eqref{Eq:NRJ-Inequality} we obtain
		\begin{align*}
		\int_0^t \int_\Omega \lvert \mathrm{p}- \mathrm{p}_\Omega \rvert \lvert v \rvert \lvert \nabla \varphi \rvert \;dx dr \le \sup_{0\le t < T_\ast} \lvert v(t) \rvert_{L^2} \left(\int_0^t \lvert \mathrm{p}-\mathrm{p}_\Omega \rvert_{L^4}^\frac4{2} dr  \right)^\frac34 \left( \int_0^t  \lvert \nabla \varphi \rvert^4_{L^4} \right)^\frac14 \\
		\le C \left(\frac{t}{R^2}\right)^\frac14  \left[E_0 + \Psi(t) \right]^\frac12 \left(\int_0^t \lvert \mathrm{p}-\mathrm{p}_\Omega \rvert_{L^4}^\frac4{3} dr  \right)^\frac34 \\
		\le C \left(\frac{t}{R^2}\right)^\frac14  \left[E_0 + \Psi(t) \right]^\frac12 \left(\int_0^t \lvert \nabla \mathrm{p} \rvert_{L^\frac43}^\frac4{3} dr  \right)^\frac34.
		\end{align*}
		From the last line and Lemma \ref{Lem:EstPressure} we infer that
		\begin{align*}
		\int_0^t \int_\Omega \lvert \mathrm{p}- \mathrm{p}_\Omega \rvert \lvert v \rvert \lvert \nabla \varphi \rvert \;dx dr \le  C \eps^\frac12  \left(\frac{t}{R^2}\right)^\frac14 \left[E_0  + \Psi(t) + (1+ \frac{2\eps^2}{R^2}) t\right]^\frac{5}{4}.
		\end{align*}
		We now deal with the terms containing $\lvert \varphi\rvert$. Applying the H\"older inequality and \eqref{Eq:EstL4normsofU+NablaD} yields
		\begin{align*}
		\int_0^t \int_\Omega \lvert \varphi \rvert \lvert \nabla g \rvert \lvert \nabla d \rvert \;dx dr \le &C \left(\int_0^t \int_\Omega \varphi^4 \;dx dr  \right)^\frac14 \left(\int_0^t \int_\Omega \lvert \nabla g \rvert^2 \;dx dr  \right)^\frac12 \left(\int_0^t \int_\Omega \lvert  \nabla d \rvert^4 \;dx dr  \right)^\frac14\nonumber \\
		\le& C t^\frac14 \eps^\frac12 \left[E_0  + \Psi(t) + (1+ \frac{2\eps^2}{R^2}) t\right]^\frac{3}{4}.
		\end{align*}
		In a similar way,
		\begin{align*}
		\int_0^t \int_\Omega \lvert \varphi \rvert \lvert  g \rvert \lvert \phi^\prime (d) \rvert \;dx dr \le C \left(\int_0^t \int_\Omega \varphi^4 \;dx dr  \right)^\frac14 \left(\int_0^t \int_\Omega \lvert  g \rvert^2 \;dx dr  \right)^\frac12 \left(\int_0^t \int_\Omega \lvert   d \rvert^4 \;dx dr  \right)^\frac14\nonumber\\
		\le C \left(\int_0^t \int_\Omega \varphi^4 \;dx dr  \right)^\frac14 \left(\int_0^t \int_\Omega \lvert  g \rvert^2 \;dx dr  \right)^\frac12 \left(1+ \sup_{r\in [0, t]} \lvert \nabla d (r)\rvert_{L^2} \right)\nonumber \\
		\le C t^\frac14 \eps^\frac12 \left[1+ E_0  + \Psi(t) + (1+ \frac{2\eps^2}{R^2}) t\right].
		\end{align*}
		We also have
		\begin{equation}
		\int_0^t \int_\Omega \lvert f \rvert \lvert v \rvert \lvert \varphi \rvert \;dx dr \le C t^\frac14 \eps^\frac12 \left[ E_0  + \Psi(t) + (1+ \frac{2\eps^2}{R^2}) t\right]^\frac34 .\label{Eq:LocalNRJ-Final}
		\end{equation}
		
		Taking $s=0$, dropping out the positive term $\int_0^t \int_\Omega \varphi \lvert R(d) \rvert^2 \;dx dr $, plugging the inequalities \eqref{Eq:LocalNRj-Start}-\eqref{Eq:LocalNRJ-Final} and using the first fact in \eqref{Eq:Condvarphi} in \eqref{Eq:LocalNRJ-DU} completes the proof of the Lemma \ref{Lem:LocalNRJ-Ball}
	\end{proof}
	We are now ready to give the proof of Proposition \ref{Prop:LocalSolwithSmallEnergy}.
	\begin{proof}[Proof of Proposition \ref{Prop:LocalSolwithSmallEnergy}]
We recall that under the assumption of Proposition \ref{Prop:LocalSolwithSmallEnergy} there exists a unique solution $(v,d) \in X_{T_\ast} $ to the problem \eqref{1b}, see Theorem \ref{th}. We can assume that $T_\ast>0$ is the maximal time of existence of  $(v,d)$. Let $r_0>0$ and $\eps_0>0$ be as in Lemma \ref{Lem:Struwe} and Lemma \ref{Lem:EstimateofDeltaD+L4norms}, respectively.   Let $R_0 \in (0,r_0)$ be chosen such that
\[
\eps_1^2= \mathcal{E}_{2R_0}(v_0,d_0)=\frac12 \sup_{x\in\Omega}\int_{\Omega\cap B_{2{R_{0}}}}\left(|v_{0}|^{2}+|\nabla d_{0}|^{2}+\phi(d_{0})\right)\;dx< \eps_0^2.
\]
Let us observe that since $\mathcal{E}(v_0,d_0)<\infty$ and
$\mu(A)=\int_{\Omega\cap A}\left[|v_{0}|^{2}+|\nabla d_{0}|^{2}+\phi(d_{0})\right]\;dx$ is absolutely continuous, then  it is possible to choose such $R_0$.
We also observe that  $\eps^2_1< E_0.$
	Now, 	let
		\[\theta_0(\eps_1, E_0):= \min\left\{\frac{K_4^{-4} \eps_1^6(1+r_0+\eps_1)^{-4}} {[1+ 2E_0 +\Psi(T) +T]^5}, \frac{3}{4}\right\}.
		\]
		and $T_0=T(\eps_1, R_0)$.	By Corollary \ref{Cor:SmallNRJBeforemaxTime} $T_0< T_\ast$.
		We will  now distinguish two cases.
		\begin{itemize}
			\item
			If  $T_0>R_0^2$, then because $\theta_0(\eps_1, E_0) \in (0, \frac34]$,   \[ T_0\ge  \theta_0(\eps_1, E_0)R_0^2,\].
			
			\item
			If  $T_0\le R_0^2$, then by Corollary \ref{Cor:SmallNRJBeforemaxTime},  the definition of $T_0$ and the continuity of $(v,d)$ at $t=T_0$ we infer that
			\begin{align}
			\mathcal{E}_{R_0}(v(T_0), d(T_0)) -\mathcal{E}_{2R_0}(v_0,d_0)= 2\eps_1^2 - \eps_1^2=\eps_1^2\nonumber
		\end{align}	
		Hence, by using the inequality \eqref{Eq:EstNRJonBalls} with $t=T_0$ and the fact $\eps^2_1< E_0$, we infer the existence of a universal constant  $K_7>0$ such that
		\begin{align*}
			\eps_1^2 \le & K_7 T_0^\frac14 \left(1+ E_0 + \Psi(T) + T+ \frac{T_0E_0}{R_0^2}  \right)^{\frac{5}{4}}\left(  R_0^{-\frac12} (\eps_1^\frac32 + \eps_1^\frac12)+ \eps_1^\frac12 \right) \\
			\le &  K_7 \eps_1^\frac12 (1+R_0^{-\frac12}[1+\eps_1]) T_0^\frac14  \left[1+2E_0  + \Psi(T) + T   \right]^\frac{5}{4},
			\end{align*}
			where $\Psi$ is defined in \eqref{Eq:DefPsi}. Since $R_0\le r_0$ we have $$ R_0^\frac12 +\eps_1+1\le 1+r_0^\frac12+\eps_1. $$
			Hence,
			\begin{align*}
			\eps_1^2 \le K_4 \eps_1^\frac12 R_0^{-\frac12} (1+r_0^\frac12+\eps_1) T_0^\frac14  \left[1+2E_0  + \Psi(T) + T   \right]^\frac{5}{4},
			\end{align*}
			from which we deduce   that
			\begin{equation*}
			T_0\ge  \frac{K_4^{-4} \eps_1^6(1+r_0^\frac12+\eps_1)^{-4}  } {[1+ 2 E_0 +\Psi(T) +T]^5} R_0^2 =\theta_0(\eps_1, E_0) R_0^2 .
			\end{equation*}
		\end{itemize}
		By the definition of $T_0=T(\eps_1, R_0)$, see \eqref{Eq:DefStoppingTime}, and the fact $T_0<T_\ast$ we automatically obtain \eqref{eqn-small estimates-B}.
		Thus, the proof of Proposition \ref{Prop:LocalSolwithSmallEnergy} is complete.
	\end{proof}
	
	\section{The existence  and the uniqueness  of  a global weak solution} \label{Sec:ExistMaxLocStrongSol}
	In this section we will prove global existence of a weak solutions to problem \eqref{1b}. Before we state and prove this  result let us define the concept of a weak  solution.
	\begin{Def}
		A global weak solution to \eqref{1b} is a pair of functions $(v,d):[0,T)\to \rH\times \rH^1$ such that
		\begin{enumerate}
			\item $(v,d) \in L^\infty(0,T; \rH\times \rH^1)$
			\item for all $t \in [0,T)$ the following integral equations
			\begin{align}
			v(t) =& v_0 +\int_0^t[Av(s) -B(v(s)) - \Pi(\Div [\nabla d(s)\odot  \nabla  d(s)] )  ] ds + \int_0^t \Pi f(s) ds,\nonumber\\
			d(t)=& d_0 +\int_0^t[\Delta d(s)+\lvert \nabla d(s)\rvert^2 d(s) -v(s)\cdot \nabla d(s) -\phi^\prime(d(s)) +(\phi^\prime(d(s) )\cdot d(s) ) d(s)   ] ds \nonumber \\
			&\qquad \qquad + \int_0^t (d(s)\times g(s)) ds,\nonumber 
			\end{align}
			hold in $D(A^{-\frac32})$ and $\rH^{-2}$, respectively.
			\item For all $t\in [0,T)$  $d(t) \in \mathcal{M}$,
			\item and $(\partial_tv, \partial_td)\in L^2(0,T_0; D(A^{-\frac32})\times \rH^{-2})$.
		\end{enumerate}
	\end{Def}
We also introduce the notion of local strong solution which will be needed to prove the existence of a global weak solution to our problem.
	\begin{Def}\label{Def:Strong-Sol}
		Let  $T_0\in (0, T]$. 
		A  function $(v,d):[0,T_0]\to H\times \rH^1 $ is a local strong solution to \eqref{1b} with initial data $(v(0), d(0))=(v_0,d_0)$ iff
		\begin{enumerate}
			\item $(v,d)\in C([0,T_0]; \rH\times \rH^1) \cap L^2(0,T_0; \ve\times D(\hA))$,
			\item for all $t \in [0,T_0]$ the equations \eqref{eq:LocVelo} and \eqref{eq:LocDir} hold in $\ve^\ast$ and $L^2$, respectively.
			\item For all $t\in [0,T_0]$\;  $d(t)\in \mathcal{M}$,
			\item and $(\partial_tv, \partial_td)\in L^2(0,T_0; \ve^\ast\times L^2)$.
			
		\end{enumerate}
		As usual we denote by $((v,d);T_0)$ a local strong solution defined on $[0,T_0]$.
		
		Similarly to Definition \ref{Def:Maximal-Sol}, one  we can also define the notion of a maximal local strong solution.
	\end{Def}

	We state the following important remark.
	\begin{Rem}
		From the definition it is clear that a maximal local solution $(v,d)$ defined on $[0,T_0)$ is a local solution on the  open interval $[0,T_0)$.
		
		In the definitions above, the condition $(\partial_tv, \partial_t d)\in L^2(0,T_0; D(A^{-\frac{j+1}{2}}) \times \rH^{j-2}) $, $j=0,2$, is equivalent to
		\begin{align*}
		F(v,d)=&  -\Pi(v\cdot \nabla v) -\Pi(\Div[\nabla d \odot \nabla d]) \in L^2(0,T_0; D(A^{-\frac{j+1}{2}}) ),\; j=0,2,\\
		G(v,d)=&\lvert \nabla d \rvert^2 d -v\cdot \nabla d -\phi^\prime(d) + (\phi^\prime(d) \cdot d) d \in L^2(0,T_0;  H^{j-2}) , j=0,2.
		\end{align*}
	\end{Rem}
	Let us now state the standing assumptions of this section.
	\begin{assume}\label{AssumptionMain}
		Let $T>0$ and assume that $(f,g)\in L^2(0,T; \rH^{-1}\times L^2)$.
		We also assume that $(v_0,d_0)\in H\times \rH^1$ satisfies
		
		\begin{align*}
		&	E_0=\mathcal{E}(v_0,d_0)=\int_\Omega (\lvert u_0\rvert^2 + \lvert \nabla d_0\rvert^2+ \phi(d_0)) dy <\infty,\\
		&	d_0 \in \mathcal{M}.
		\end{align*}
	\end{assume}
	The first main result of this section is the following uniqueness result.
	\begin{Prop}\label{Prop:Uniq}
		Let $(v_i, d_i)\in C([0,T]; \rh\times \rH^1 )\cap L^2(0,T; D(\rA^\frac12)\times D(\hA^\frac32 )$, $i=1,2$,  be two strong  solutions to \eqref{1b} defined on $[0,T]$.
		Then,
		\begin{equation}
		(v_1, d_1)= (v_2, d_2).\notag
		\end{equation}	
	\end{Prop}
	\begin{proof}
		In order to prove this result we closely follow the approach of \cite{JL+ET+ZX-2016}.
		
		Let  $(v_i,d_i)$ be a  two strong solutions to \eqref{1b}, $v=v_1-v_2$ and $d=d_1-d_2$. Hence,  $v$  satisfies the equation
		\begin{equation}
		\frac{d v}{dt} + \rA v + B(v, v_1)+B(v_2,v) = -\Pi\left(\Div[\nabla d \odot \nabla d_1 + \nabla d_2 \odot \nabla d  ]\right).\notag
		\end{equation}
		Let $w=\rA^{-1}v$. It is not difficult to show that $w$ satisfies
			\begin{equation}
		\frac{dw }{dt} +\rA w + A^{-1} \left(B(v, v_1 ) + B(v_2,v)\right) =-\rA^{-1} \Pi \left(\Div[\nabla d \odot \nabla d_1 + \nabla d_2 \odot \nabla d  ]\right).\notag
		\end{equation}
		 By parts (1)  and (4) of Definition 5.2,  we have $w \in  L^2(0,T; D(\rA^\frac32)) $ and $\partial_tw= A^{-1}\partial_t v\in L^2(0,T; \ve)\subset L^2(0,T; D(\rA)^\ast)$. Then, by applying the Lions-Magenes Lemma (\cite[Lemma III.1.2]{Temam_2001})  and  using the  facts that  $\rA$ is self-adjoint and $\Div w=0$  we infer that
		\begin{align}
		\frac12 \frac{d}{dt} \lvert \rA^\frac12 w \rvert^2_{L^2} + \lvert \rA w\rvert^2_{L^2}=- \langle B(v, v_1) + B(v_1, v), w  \rangle
		-\langle \Pi \left(\Div[\nabla d \odot \nabla d_1 + \nabla d_2 \odot \nabla d  ]\right), w\rangle \notag\\
		= - \langle B(v, v_1) + B(v_1, v), w  \rangle  -\langle \nabla d \odot \nabla d_1 + \nabla d_2 \odot \nabla d, \nabla w \rangle.\notag
		\end{align}
	We also used the integration by parts to obtain the second line. Let us now estimate the terms on the right hand side of the last line of the chain of identities above.
	
	 Hereafter we fix $\eps,\gamma>0$, the symbols $C_\eps, C_{\eps,\gamma}$ denote two positive constants depending only on $\eps$ and $\gamma$.
	
		Firstly, by using the H\"older, the Young inequalities and the Ladyzhenskaya inequality (\cite[Lemma III.3.3]{Temam_2001}) we infer  that
		\begin{align*}
		-\langle B(v, v_1), w\rangle =& \langle B(v, w), v_1\rangle\\
		\le & \lvert v \rvert_{L^2} \lvert \nabla w \rvert_{L^4} \lvert v_1\rvert_{L^4}\\
		\le & \eps \lvert v \rvert^2_{L^2} + C_{\eps}\lvert \nabla w\rvert_{L^2} \lvert \nabla^2 w\rvert_{L^2} \lvert v_1\rvert^2_{L^4}\\
		\le & \eps \lvert v \rvert^2_{L^2}+C_{\eps} \lvert \rA^\frac12 w\rvert_{L^2} \lvert \rA w\rvert_{L^2} \lvert v_1\rvert^2_{L^4}\\
		\le & \eps \lvert v \rvert^2_{L^2}+\eps \lvert \rA w \rvert^2_{L^2} + C_\eps \lvert \rA^\frac12 w \rvert^2_{L^2} \lvert v_1 \rvert^2_{L^4}.
		\end{align*}
		Observe that $\lvert v\rvert^2_{L^2} = \lvert \rA w \rvert^2_{L^2}$. Thus,
		\begin{align*}
		-\langle B(v, v_1), w\rangle
		\le 2 \eps \lvert \rA w \rvert^2_{L^2} + C_\eps \lvert \rA^\frac12 w \rvert^2_{L^2} \lvert v_1 \rvert^4_{L^4}.
		\end{align*}
		In a similar way, we can prove that
		\begin{align}
		-\langle B(v_2, v), w\rangle
		\le 2 \eps \lvert \rA w \rvert^2_{L^2} + {C_\eps} \lvert \rA^\frac12 w \rvert^2_{L^2} \lvert v_2 \rvert^4_{L^4}.\notag
		\end{align}
		
		 Secondly, making use of the Ladyzhenskaya inequality (\cite[Lemma III.3.3]{Temam_2001}), the H\"older and the Young inequalities we obtain
		\begin{align*}
		-\langle \nabla d \odot \nabla d_1 , \nabla w \rangle &\le \lvert \nabla d \rvert_{L^2} \lvert \nabla d_1 \rvert_{L^4} \lvert \nabla w \rvert_{L^4} \\
		&\le \gamma \lvert \nabla d \rvert^2_{L^2} + C_\gamma \lvert \nabla d_1 \rvert^2_{L^4} \lvert \nabla w \rvert_{L^2} \lvert \nabla^2 w\rvert_{L^2}\\
		&\le \gamma \lvert \nabla d \rvert^2_{L^2} + C_\gamma \lvert \nabla d_1 \rvert^2_{L^4} \lvert \rA^\frac12 w \rvert_{L^2} \lvert \rA  w\rvert_{L^2}\\
		& \le \gamma \lvert \nabla d \rvert^2_{L^2} + \eps \lvert \rA w \rvert^2_{L^2} + C_{\gamma, \eps} \lvert \rA^\frac12 w \rvert^2_{L^2} \lvert \nabla d_1 \rvert^4_{L^4}.
		\end{align*}
		Similarly,
		\begin{align*}
		-\langle \nabla d_2 \odot \nabla d , \nabla w \rangle
		\le \gamma \lvert \nabla d \rvert^2_{L^2} + \eps \lvert \rA w \rvert^2_{L^2} + C_{\gamma, \eps} \lvert \rA^\frac12 w \rvert^2_{L^2} \lvert \nabla d_2 \rvert^4_{L^4}.
		\end{align*}
		Collecting all these inequalities we obtain
		\begin{align}\label{Eq:DifferenceVelo-Fin}
		\frac12 \frac{d}{dt} \lvert \rA^\frac12 w \rvert^2_{L^2} + \lvert \rA w\rvert^2_{L^2} &\le \eps \lvert \rA w \rvert^2_{L^2} + \gamma \lvert \nabla d \rvert^2_{L^2}
\\ &+ C_{\eps,\gamma} \lvert \rA^\frac12 w \rvert^2_{L^2}\left(\lvert v_1 \rvert^4_{L^4} + \lvert v_2 \rvert^4_{L^4} + \lvert \nabla d_1 \rvert^4_{L^4} +  \lvert \nabla d_2 \rvert^4_{L^4}  \right).
		\nonumber \end{align}

		Let us turn our attention to the function $d=d_1-d_2$. We notice that $d$ satisfies
		\begin{equation*}
		\begin{split}
		\frac{d}{dt}d  +\hA d+ v\cdot \nabla d_1 + v_2\cdot \nabla d =\left( \lvert \nabla d_1 \rvert^2- \lvert \nabla d_2\rvert^2\right)d_1+ \lvert \nabla d_2 \rvert^2 d -[\phi^\prime(d_1) - \phi^\prime(d_2)]
		\\  + [\alpha(d_1)-\alpha(d_2)]d_1 + \alpha(d_2)d + d\times g\\
		=\left( \nabla d_1 -\nabla d_2 : \nabla d_1 + \nabla d_2   \right)d_1 + \lvert \nabla d_2 \rvert d  -[\phi^\prime(d_1) - \phi^\prime(d_2)]
		\\  + [\alpha(d_1)-\alpha(d_2)]d_1 + \alpha(d_2)d + d\times g\\
		=\left( \nabla d : \nabla d_1 + \nabla d_2   \right)d_1 + \lvert \nabla d_2 \rvert d -[\phi^\prime(d_1) - \phi^\prime(d_2)]
		\\  + [\alpha(d_1)-\alpha(d_2)]d_1 + \alpha(d_2)d + d\times g.
		\end{split}
		\end{equation*}
		Since $d_1,d_2\in  L^2(0,T; D(\hA))$, and $\partial_t d_1, \partial_t d_2 \in L^2(0,T; L^2)$, we infer by applying Lions-Magenes Lemma (\cite[Lemma III.1.2]{Temam_2001}),  and the facts $\langle v_2 \cdot \nabla d, d\rangle=0$ and $ \langle d\times g, d\rangle =0$ (because  $d\times g \perp_{\mathbb{R}^3} d$) that
		\begin{align}
\nonumber
		\frac12 \frac{d}{dt} \lvert d \rvert^2_{L^2} + \lvert \nabla d \rvert^2_{L^2}&= -\langle v \cdot \nabla d_1+[\nabla d : \nabla(d_1+d_2)]d_1 + \lvert \nabla d_2 \rvert^2 d -[\phi^\prime(d_1) \phi^\prime(d_2)]_+\alpha(d_2) d, d \rangle \\
		&+ \langle [\alpha(d_1)-\alpha(d_2)]d_1, d\rangle.
\label{Eq:DifferenceDinL2}		\end{align}
		Let us  estimate the terms in the right hand side of \eqref{Eq:DifferenceDinL2}. First, by using the H\"older, the Young inequalities and  the Gagliardo-Nirenberg  inequality (\cite[Section 9.8, Example C.3]{Brezis})
		we show that
		\begin{align*}
		-\langle v \cdot \nabla d_1 , d \rangle \le &\lvert v \rvert_{L^2} \lvert \nabla d_1 \rvert_{L^4} \lvert d \rvert_{L^4} \\
		\le &\eps \lvert v \rvert^2_{L^2} + C_\eps \lvert \nabla d_1\rvert^2_{L^4} \lvert d \rvert_{L^2}(\lvert d \rvert_{L^2} + \lvert \nabla d \rvert_{L^2} )\\
		\le& \eps \lvert v \rvert^2_{L^2} + \gamma \lvert \nabla d \rvert^2_{L^2}+ C_\eps \lvert \nabla d_1 \rvert^2_{L^4} \lvert d\rvert^2 +  \lvert d \rvert^2_{L^2} \lvert \nabla d_1\rvert^4_{L^4}\\
		\le & \eps \lvert \rA w \rvert^2_{L^2} + \gamma \lvert \nabla d \rvert^2_{L^2}+ C_{\eps, \gamma}  \  \lvert d \rvert^2_{L^2}\left( 1+ \lvert \nabla d_1\rvert^4_{L^4}\right).
		\end{align*}
		With the same idea,  we prove  that
		\begin{align*}
		\langle [\nabla d : \nabla (d_1 +d_2) ]d_1, d \rangle \le & \lvert \nabla d \rvert_{L^2} \lvert d\rvert_{L^4} [\lvert \nabla d_1 \rvert_{L^4}+ \lvert \nabla d_2\rvert_{L^4} ] \lvert d_1 \rvert_{L^\infty}\\
		\le& \gamma \lvert \nabla d \rvert^2_{L^2} +C_\gamma  \lvert  d \rvert_{L^2} [\lvert d\rvert_{L^2} + \lvert \nabla d \rvert_{L^2} ] [ \lvert \nabla d_1 \rvert^2_{L^4} + \lvert \nabla d_2 \rvert^2_{L^4}  ]\lvert d_1 \rvert^2_{L^\infty}\\
		\le & 2 \gamma \lvert \nabla d  \rvert^2_{L^2} + C_\gamma \lvert d \rvert^2_{L^2} [  \lvert \nabla d_1 \rvert^2_{L^4} + \lvert \nabla d_2 \rvert^2_{L^4}  + \lvert \nabla d_1 \rvert^4_{L^4} + \lvert \nabla d_2 \rvert^4_{L^4}]\\
		\le  & 2 \gamma \lvert \nabla d  \rvert^2_{L^2} + C_\gamma \lvert d \rvert^2_{L^2} [ 1+  \lvert \nabla d_1 \rvert^4_{L^4} + \lvert \nabla d_2 \rvert^4_{L^4}]
		\end{align*}
		In the last line we used the fact that $\lvert d_1 \rvert_{L^\infty} \le 1$.
		
		\noindent Utilizing the H\"older, the Young inequalities and the Gagliardo-Nirenberg inequality  (\cite[Section 9.8, Example C.3]{Brezis})  we obtain
		\begin{align*}
		\langle \lvert \nabla d_2 \rvert^2 d, d \rangle \le & \lvert \nabla d_2 \rvert^2_{L^4} \lvert d \rvert^2_{L^4}\\
		\le & \lvert \nabla d_2 \rvert^2_{L^4} \lvert d \rvert_{L^2}(  \lvert d \rvert_{L^2} + \lvert \nabla d \rvert_{L^2})\\
		\le & \gamma \lvert \nabla d \rvert^2_{L^2} +C_\gamma \lvert d \rvert^2_{L^2} \left( 1+ \lvert \nabla d_2 \rvert^4_{L^4} \right).
		\end{align*}
		Since $\lvert \phi^{\prime\prime}\rvert\le M $, the map $\phi^\prime: \mathbb{R}^3 \to \mathbb{R}^3$ is Lipschitz and
		\begin{equation}
		- \langle \phi^\prime(d_1) - \phi^\prime(d_2), d\rangle\le M \lvert d \rvert^2_{L^2}. \label{Eq:DiffPhiprime}
		\end{equation}
		Using the definition of $\alpha(d_2)=(\phi^\prime(d_2) \cdot d_2)$, the fact $\lvert d_2 \rvert=1$ and \eqref{Eq:LInearGrowthPhiprime}  we have
		\begin{align}
		\langle \alpha(d_2) d, d \rangle = \langle (\phi^\prime(d_2) \cdot d_2) d, d\rangle
		\le 2 M\lvert d\rvert^2_{L^2}. \label{Eq:PhiprimeD}
		\end{align}
		Using again the definition of $\alpha(d_1)$ and $\alpha(d_2)$ we obtain
		\begin{align*}
		\langle [\alpha(d_1)-\alpha(d_2)]d_1, d \rangle = &\langle [\phi^\prime(d_1)\cdot d_1 -\phi^\prime(d_2)\cdot d_2]d_1, d\rangle \\
		=& \langle \left([\phi^\prime(d_1)-\phi^\prime(d_2) ]\cdot d_1 + \phi^\prime(d_2)\cdot d \right)d_1, d\rangle.
		\end{align*}
		Since $\phi^\prime$ is Lipschitz,  $d_i(t)\in \mathcal{M}$ for all $t\in [0,T]$,  we show with the same ideas as used in \eqref{Eq:DiffPhiprime} and \eqref{Eq:PhiprimeD} that
		\begin{align*}
		\langle [\alpha(d_1)-\alpha(d_2)]d_1, d \rangle \le 3 M \lvert d \rvert^2_{L^2}.
		\end{align*}
		Hence, collecting all these inequalities related to the terms in the right hand side of the equation \eqref{Eq:DifferenceDinL2} we obtain
		\begin{equation}\label{Eq:DifferenceOptDir-fin}
		\frac12 \frac{d}{dt} \lvert d \rvert^2_{L^2} + \lvert \nabla d \rvert^2_{L^2}\le \eps \lvert \rA w \rvert^2_{L^2}+ \gamma \lvert \nabla d \rvert^2_{L^2} + C_\gamma \lvert d \rvert^2_{L^2} \left(1 + \lvert \nabla d_1 \rvert^4_{L^4} + \lvert \nabla d_2 \rvert^4_{L^4}\right).
		\end{equation}
		Thus, summing \eqref{Eq:DifferenceVelo-Fin} and \eqref{Eq:DifferenceOptDir-fin} up, we have
		\begin{align}\nonumber
\label{Eq:DifferenceVeloOptDir}
			\frac12 \frac{d}{dt} \left(\lvert d \rvert^2_{L^2} + \lvert \rA^\frac12 w\rvert^2_{L^2} \right)+\lvert \rA w\rvert^2_{L^2} + \lvert \nabla d \rvert^2_{L^2} &\le \eps \lvert \rA w \rvert^2_{L^2}+ \gamma \lvert \nabla d \rvert^2_{L^2} + C_{\gamma, \eps} \lvert d \rvert^2_{L^2} \left(1 + \lvert \nabla d_1 \rvert^4_{L^4} + \lvert \nabla d_2 \rvert^4_{L^4}\right) \\
		&\hspace{-1truecm}+ C_{\eps,\gamma}
		\lvert \rA^\frac12 w\rvert^2_{L^2}\left( \lvert \nabla d_1 \rvert^4_{L^4} + \lvert \nabla d_2 \rvert^4_{L^4}+\lvert v_1 \rvert^4_{L^4} + \lvert v_2 \rvert^4_{L^4}\right).
		\end{align}
	Let choose $\eps=\gamma=\frac12$  and put
		\begin{align*}
		y(t)= &\lvert d(t) \rvert^2_{L^2} + \lvert \rA^\frac 12 w(t)\rvert^2_{L^2}, t\in [0,T], \\
		\Phi(t)=& 2C_{\frac12,\frac12} \left(  1+\lvert \nabla d_1(t) \rvert^4_{L^4} + \lvert \nabla d_2(t) \rvert^4_{L^4}+\lvert v_1(t) \rvert^4_{L^4} + \lvert v_2 (t) \rvert^4_{L^4}  \right), t\in [0,T].
		\end{align*}
		Then, we see from \eqref{Eq:DifferenceVeloOptDir} that $y$ satisfies
		\begin{equation}\label{Eq:NeedGornwall}
		\dot{y}(t) \le \Phi(t) y(t), t\in [0,T].
		\end{equation}
		Observe that since $(v_i,d_i)\in C([0,T], \rH\times \rH)\cap L^2(0,T; D(\rA^\frac12) \times D(\hA))$ we infer that
		\begin{equation*}
		\int_0^T \Phi(s) ds<\infty.
		\end{equation*}
		Thus, one can apply the Gronwall inequality to  \eqref{Eq:NeedGornwall} and deduce that
		\begin{equation*}
		y(t) \le y(0)e^{\int_0^t \Phi(s) ds }\le y(0)e^{\int_0^T\Phi(s) ds}, \; t\in [0,T].
		\end{equation*}
		Since $$y(0)=\lvert d_1(0) -d_2(0)\rvert^2_{L^2} + \lvert \rA^{-1\frac12} v_1(0)-\rA^{-\frac12} v_2(0)\rvert^2_{L^2}= 0,$$
		we have
		\begin{equation*}
		y(t)= \lvert d_1(t)-d_2(t)\rvert^2_{L^2} + \lvert v_1(t) -v_2(t) \rvert^2_{D(\rA^{-\frac12})}=0, \; t\in [0,T].
		\end{equation*}
		This completes the proof of the Proposition \ref{Prop:Uniq}.
	\end{proof}
	
	The second  main result of this paper is  the following theorem.
	\begin{Thm}\label{thm-main}
		Let $(v_0,d_0)\in \rh\times \rH^1$, $r_0>0$ and $\eps_0>0$ be the constants from Lemma \ref{Lem:Struwe} and \ref{Lem:EstimateofDeltaD+L4norms}, respectively. Then, there exist constants $\varrho_0\in (0, r_0]$ and $\eps_1\in (0,\eps_0)$ such that the following hold.
		
		If Assumption \ref{AssumptionMain} holds, then there are a number $L\in \mathbb{N}$, depending only on the norms of $(v_0,d_0)\in H\times \rH^1$ and $(f,g)\in L^2(0,T; \rH^{-1} \times L^2)$, a collection of times $0<T_1<\cdots<T_L\le T$ and  a global weak solution $(\mathbf{u},\mathbf{d})\in C_{w}([0, T]; \mathrm{H} \times \rh^1) \cap L^2(0, T; \ve\times D(\hA) )$ to \eqref{1b-0} such that
		\begin{enumerate}
			\item for each $i\in \{1, \ldots, L\}$, $(\mathbf{u},\mathbf{d})_{\lvert_{[T_{i-1}, T_i)}}\in C([T_{i-1}, T_i); \mathrm{H} \times (\rH^1\cap \mathcal{M})) $ with the left-limit at $T_i$,
  is a maximal local regular solution to \eqref{1b} with initial data $(v(T_{i-1}), d(T_{i-1}) )$.  Here we understand that $T_0=0$.
\item  If $T_L<T$, then
$(\mathbf{u},\mathbf{d})_{\lvert_{[T_{L}, T]}}$ belongs to $ C([T_{L}, T]; \mathrm{H} \times \rh^1)$ and  satisfies  the variational form  problem \eqref{1b-0} on the interval $[T_{L}, T]$  with initial data $(v(T_{L}), d(T_{L}) )$.
			\item For each $i\in \{1, \ldots, L\}$
			\begin{equation*}
			\lim_{t\toup T_i} \mathcal{E}_{R}(\mathbf{u}(t), \mathbf{d}(t)) \ge \eps_1^2,
			\end{equation*}
			for all $R\in (0,\varrho_0]$.
			\item At each $T_i$ there is a loss of energy at least $\eps_1^2\in (0, \eps_0^2)$, \textit{i.e.},
			\begin{equation*}
			\mathcal{E}(\mathbf{u}(T_i), \mathbf{d}(T_i)) \le \mathcal{E}(\mathbf{u}(T_{i-1}), \mathbf{d}(T_{i-1})) +\frac12 \int_{T_{i-1}}^{T_1} \left[\lvert f\rvert_{\rH^{-1}}^2 + \lvert g\rvert^2_{L^2}\right] dt -\eps_1^2.
			\end{equation*}
		\end{enumerate}
		
	\end{Thm}
	The proof of this theorem is established in several steps. The first of such steps is the proof of the existence of a maximal local strong solution for the Ericksen-Leslie system \eqref{1b} with data satisfying Assumption \ref{AssumptionMain}.
	
	\begin{Thm}\label{Ma1}
		Let $(u_0,d_0)\in \rh\times(\rH^1\cap \mathcal{M})$, $r_0>0$ and $\eps_0>0$ be the constants from Lemma \ref{Lem:Struwe} and \ref{Lem:EstimateofDeltaD+L4norms}, respectively.
		Then, there exist $\varrho_0\in (0, r_0]$ such that
		\begin{align}
		\eps_1^2=\mathcal{E}_{2\varrho_0}(u_0,d_0)<  \eps_0^2,\label{Eq:AssumSmallInitialNRJ-1-0},
		\end{align}
		and a maximal local strong solution  $((v,d); T_\ast)$ satisfying
		\begin{align*}
		\limsup_{t\toup T_{\ast}}\mathcal{E}_{R}(v(t), d(t)) \ge\varepsilon^{2}_{1},
		\end{align*}
		for all $R\in (0, \varrho_0]$.
	\end{Thm}
	For the proof of this theorem, we first prove the existence of a local strong solution which is given by the following proposition.
	\begin{Prop}\label{Ma2}
		There exist $\eps_0>0$ and a function 	
		\[\theta_0: (0,\eps_0)\times (0,\infty) \to (0,\frac34]
		\]
		which is non-increasing w.r.t. the  second variable and nondecreasing  w.r.t. the first one
		such that the following holds. \\
		Let $r_0>0$ be the constant from Lemma \ref{Lem:Struwe} and assume that  the initial data $(v_0,d_0)$ and the forcing $(f,g)$ satisfies Assumption \ref{AssumptionMain}.
		
		 Then, there exists $\varrho_0\in (0, r_0]$ such that
		\begin{equation}\label{Eq:AssumSmallInitialNRJ-1}
		\eps_1^2=2\mathcal{E}_{2\varrho_0}(u_0,d_0)<  \eps_0^2.
		\end{equation}
		Moreover, there exists a local strong solution $((v,d);T_0)$ satisfying
		\begin{equation}\label{eqn-t_o-2}
		\begin{split}
		&T_0 \ge  \frac{\varrho^2_0}{(\varrho_0^\frac12 + 1)^4} \theta_0(\eps_1,E_0)\\
		&\sup_{t\in [0,T_0]}\mathcal{E}_{\varrho_0}(v(t), d(t)) \le 2\eps_1^2.
		\end{split}
		\end{equation}
		Furthermore, there exists a constant $K_5>0$
		\begin{align}
		\sup_{t\in [0,T_0]} \mathcal{E}(v(t), d(t)) + \int_0^{T_0}  \lvert \nabla v (r) \rvert^2_{L^2} dr\le E_0 +\frac12 \int_0^T \left( \lvert f(s) \rvert^2_{\rH^{-1} } + \lvert g(s\rvert^2_{L^2}\right) ds \label{Eq:UnifEstStrongSol-1}\\
		\int_0^{T_0} \lvert \Delta d(t) \rvert^2_{L^2} dt \le K_5 \left[ E_0  + \frac12 \int_0^{T}\left(\lvert f (s) \rvert^2_{\rH^{-1}}+ \lvert g(s) \rvert^2_{L^2}\right) ds + \left(1+\frac{E_0}{R^2}\right)T \right], \;\; R \in (0,\varrho_0].\label{Eq:UnifEstStrongSol-2}
		\end{align}
		where $E_0:= \mathcal{E}(u_0,d_0)$.
		%
		%
		
		\delb{and $(v,d)\in C([0,T_{0}];H\times \rH^{1}_{w})\cap L^{2}(0,T_{0}; V\times \rH^{2})$.
			Also for all $t\in [0,T_0]$, a.e. $x\in \Omega$
			$|d(t,x)|=1$.}
		
	\end{Prop}
	
	\begin{proof}[Proof of Proposition \ref{Ma2}]
		There exists a sequence $\{d_{0}^{k}\}_{k=1}^{\infty}\subset C^{\infty}(\Omega;\mathbb{S}^{2})$ such that
		\[
		\lim_{k\to\infty}\|d_{0}^{k}-d_{0}\|_{H^{1}}=0,
		\]
		see \cite[Section 4]{Schoen-Uhlenbeck-83}.
		By definition,  $v_{0}\in \rH$ can be approximated by a sequence $\{v_{0}^{k}\}_{k=1}^{\infty}\subset \mathcal{V}$ such that
		\[
		\lim_{k\to\infty}|v_{0}^{k}-v_{0}|_{L^{2}}=0.
		\]
		Since embedding $L^2 \hookrightarrow \rH^{-1}$ and $\rH^1\subset L^2$ are dense, then one can also approximate $(f,g)\in L^2(0,T; \rh^{-1} \times L^2)$ by a sequence $((f_k, g_k))_{k \in \mathbb{N}} \subset L^2(0,T; L^2 \times \rh^1)$ in the following sense
		\begin{equation}
		(f_k, g_k) \rightarrow (f,g) \text{ strong in }  L^2(0,T; \rh^{-1} \times L^2).
		\end{equation}
		
		Let $\eps_0>0$ and $R_0\in (0,r_0]$ be the constants from Proposition \ref{Prop:LocalSolwithSmallEnergy}.  Since $\mathcal{E}(v_0,d_0)<\infty$ and
		$\mu(A)=\int_{\Omega\cap A}\left[|v_{0}|^{2}+|\nabla d_{0}|^{2}+\phi(d_{0})\right]\;dx$ is absolutely continuous, then  there exists $\tilde{R_{0}}>0$ such that
		\[
		\eps^2_1:=  \sup_{x\in\Omega}\int_{\Omega\cap B_{2\tilde{R_{0}}}}\left(|v_{0}|^{2}+|\nabla d_{0}|^{2}+\phi(d_{0})\right)\;dx< \eps_0^2.
		\]
		Choosing $\varrho_0=\tilde{R}_0 \wedge R_0$ yields  \eqref{Eq:AssumSmallInitialNRJ-1}.
		%
		Let $\eps>0$ be an arbitrary real number.  Since $(v_{0}^{k}, d_{0}^{k})$ strongly converges to $(v_{0},d_{0})$ in $ L^{2}\times \rH^{1}$,  there exists a number $k_0\in \mathbb{N}$  such that for all $k\ge k_0$
		\begin{align}
		\frac12\int_{\Omega\cap B_{2\varrho_0}(x)}|v_{0}^{k}|^{2}\;dx
		&=\frac12  \int_{\Omega\cap B_{2\varrho_0}(x)}|v_{0}^{k}-v_{0}+v_{0}|^{2}\notag\\
		&\le \int_{\Omega\cap B_{2\varrho_0(x)}}|v_{0}^{k}-v_{0}|^{2}\;dx+  \int_{\Omega\cap B_{2\varrho_0(x)}}|v_{0}|^{2}\;dx\notag\\
		&\le \frac\eps3+  \int_{\Omega\cap B_{2\varrho_0(x)}}|v_{0}|^{2}\;dx.\notag
		\end{align}
		In a similar way one can prove that,
		\[
		\int_{\Omega\cap B_{2\varrho_0}(x)}|\nabla d_{0}^{k}|^{2}\;dx\le \frac\eps3+ \int_{\Omega\cap B_{2\varrho_0(x)}}|\nabla d_{0}|^{2}\;dx.
		\]
		Observe that
		\begin{align}
	\int_{\Omega\cap B_{2\varrho_0}(x)}\phi(d_{0}^{k})\;dx
		&= \int_{\Omega\cap B_{2\varrho_0}(x)}\left(\phi(d_{0}^{k})-\phi(d_{0})+\phi(d_{0})\right)\;dx\notag\\
		&\le \int_{\Omega\cap B_{2\varrho_0}(x)}|\phi(d_{0}^{k})-\phi(d_{0})|\;dx +\int_{\Omega\cap B_{2\varrho_0}(x)}\phi(d_{0})\;dx\notag\\
		&\le \int_{\Omega\cap B_{2\varrho_0}(x)}|\phi(d_{0}^{k})-\phi(d_{0})|\;dx + \int_{\Omega\cap B_{2\varrho_0}(x)}\phi(d_{0})\;dx\notag\\
		&\le  \int_{\Omega\cap B_{2\varrho_0}(x)}|\phi(d_{0}^{k})-\phi(d_{0})|\;dx + \int_{\Omega\cap B_{2\varrho_0}(x)}\phi(d_{0})\;dx.\notag
		\end{align}
		Since $|d_{0}^{k}|=|d_{0}|=1$, $$ \sup_{n \in \mathbb{S}^2 }  \lvert\phi^\prime(n) \rvert\le M $$ and $d_0^k \to d_0$ in $\rH^1$  we deduce that there exists $k_0\in \mathbb{N}$  such that for all $k\ge k_0$
		\begin{align}
		\int_{\Omega\cap B_{2\varrho_0}(x)}\phi(d_{0}^{k})\;dx
		&\le M \int_{\Omega} \lvert d_0^k - d_0\rvert \;dx +  \int_{\Omega\cap B_{2\varrho_0}(x)}\phi(d_{0})\;dx\nonumber\\
		&\le \frac\eps3 +  \int_{\Omega\cap B_{2\varrho_0}(x)}\phi(d_{0})\;dx.\notag
		\end{align}
		Hence,  there exists a constant $k_0\in \mathbb{N}$ such that for all $k\ge k_0$
		\begin{equation}
		\mathcal{E}_{2\varrho_0}(v_0^k, d_0^k) \le \eps +  2 \mathcal{E}_{2\varrho_0}(v_0, d_0).\notag
		\end{equation}
		Since $\eps$ is arbitrary we infer that
		\begin{equation}
		\mathcal{E}_{2\varrho_0}(v_0^k, d_0^k) \le  \eps_1^2 < \eps_0^2.\notag
		\end{equation}
		Without loss of generality, we will assume that for all $k\ge 1$
		\[
		\mathcal{E}_{2\varrho_0}(v_0^k, d_0^k)\le \varepsilon_{1}^{2}.
		\]
		
		By Proposition \ref{Prop:LocalSolwithSmallEnergy} there exist a function $\theta_0:(0,\eps_0) \times [0, \infty)\to (0,\frac34]$ satisfying the properties stated in Proposition \ref{Ma2}, a sequence of time $T_0^k$ satisfying \eqref{eqn-t_o-2} and a sequence of regular solutions $(v^{k},d^{k}):[0,T_{0}^{k}]\to \ve \times \rH^{2}$, with initial condition
		\[
		(v^{k}_{0}(0),d^{k}_{0}(0))=(v_{0}^{k},d_{0}^{k}).
		\]
		Moreover,
		\begin{equation}\label{Eq:NRGSmall}
		\sup_{t\in [0,T_{0}^{k}]}\mathcal{E}_{\varrho_0}(v^k(t), d^k(t))\le  2\varepsilon_{1}^{2}.
		\end{equation}
		We recall that for all $k\ge 1$ and $t\in [0,T_{0}^{k})$, we have
		\begin{align}
		&\sup_{t\le T^{k}_{0}}\mathcal{E}(v^{k}(t), d^{k}(t)) + \int_{0}^{T^{k}_{0}}\left[|\nabla v^{k}|^{2}_{L^{2}}+|\Delta d^{k}-\phi^\prime(d^{k})+\alpha(d^{k})d^{k}+|\nabla d^{k}|^{2}d^{k}|^{2}_{L^{2}}\right]dt\notag\\
		&\le \mathcal{E}(v_{0}^{k},d^{k}_{0})+ \frac12 \int_0^{T^{k}_{0} } (\lvert f_k  \rvert^2_{\rH^{-1}} + \lvert g_k\rvert^2_{L^2}  ) \,dt\le E_{0} + \frac12 \int_0^{T^{k}_{0}} (\lvert f  \rvert^2_{\rH^{-1}} + \lvert g\rvert^2_{L^2}  ) \,dt .\notag
		\end{align}
		Hereafter, we put
		\begin{equation}
		\Xi_0=E_{0} + \frac12 \int_0^{T^{k}_{0}} (\lvert f  \rvert^2_{\rH^{-1}} + \lvert g\rvert^2_{L^2}  ) \,dt.\notag
		\end{equation}
		We also recall that there exists $C>0$ such that  for all $k\in \mathbb{N}$ and $R\in (0,\varrho_0]$
		\begin{align}
		&\int_{0}^{T^{k}_{0}}|\Delta d^{k}|^{2}_{L^{2}}dt\le C \left[ E_0  + \frac12 \int_0^{T}\left(\lvert f (s) \rvert^2_{\rH^{-1}}+ \lvert g(s) \rvert^2_{L^2}\right) ds + \left(1+\frac{E_0}{R^2}\right)T \right],\label{Eq:UnifEstLaplace}\\
		&\int_{0}^{T^{k}_{0}}\left(|v^{k}|^{4}_{L^{4}}+|\nabla d^{k}|^{4}_{L^{4}}\right)dt\le  C \eps_1^2\left[ E_0  + \frac12 \int_0^{T}\left(\lvert f (s) \rvert^2_{\rH^{-1}}+ \lvert g(s) \rvert^2_{L^2}\right) ds + \left(1+\frac{E_0}{R^2}\right)T \right],\label{Eq:UnifEstL4}\\
		&\int_{0}^{T^{k}_{0}}|\nabla v^{k}|^{2}_{L^{2}}dt\le E_{0}+\frac12 \int_0^{T}\left(\lvert f (s) \rvert^2_{\rH^{-1}}+ \lvert g(s) \rvert^2_{L^2}\right) ds, \label{Eq:UnifEstL2}
		\end{align}
		and
		\begin{equation}\label{Eq:UnifEstNRG}
		\sup_{t\le T^{k}_{0}}\left[|v^{k}(t)|^{2}_{L^{2}}+|\nabla d^{k}(t)|^{2}_{L^{2}}+\int_{\Omega}\phi(d^{k}(t))\;dx\right]\le  E_{0}+\frac12 \int_0^{T}\left(\lvert f (s) \rvert^2_{\rH^{-1}}+ \lvert g(s) \rvert^2_{L^2}\right) ds.
		\end{equation}
		We now estimate the time derivatives. Let us put
		\begin{equation}
		\mathcal{E}(d(s)) =\frac12\lvert \nabla d(s) \rvert^2_{L^2}+ \int_\Omega \phi(d(s,x)) \;dx.\notag
		\end{equation}
		Then, we  deduce from Claim \ref{Claim:perp} that
		\begin{align}
		\int_{0}^{T^{k}_{0}}|\partial_sd^{k}(s)|^{2}_{L^{2}}ds =&\int_0^{T^k_0} \langle R(d(s)), \partial_s d(s)\rangle ds -\int_0^{T^k_0} \langle v^k(s)\cdot \nabla d^k(s), \partial_s d^k(s) \rangle ds\notag
		\\
		=&\int_{0}^{T^{k}_{0}}\frac{d}{ds}\mathcal{E}(d^{k}(s))ds-\int_{0}^{T^{k}_{0}}(v^{k}(s)\cdot \nabla d^{k}(s),\partial_sd^{k}(s) )ds.\notag
		\end{align}
		Hence
		\begin{align}
		\int_{0}^{T^{k}_{0}}|\partial_sd^{k}(s)|^{2}_{L^{2}} ds
		&\le \sup_{t\le T^{k}_{0}}\mathcal{E}(d^{k}(t))-\mathcal{E}(d^{k}(0))
		+\int_{0}^{T^{k}_{0}}(v^{k}(s)\cdot \nabla d^{k}(s),\partial_sd^{k}(s))ds\notag\\
		&\le\mathcal{E}(v_{0}^{k},d_{0}^{k})-\mathcal{E}(d^{k}(0))+\int_{0}^{T^{k}_{0}}(v^{k}(s)\cdot \nabla d^{k}(s),\partial_sd^{k}(s))\,ds\notag\\
		&\le \frac{1}{2}|v_{0}^{k}|^{2}_{L^{2}}+\int_{0}^{T^{k}_{0}}(v^{k}(s)\cdot \nabla d^{k}(s),\partial_sd^{k}(s))\,ds\notag\\
		&\le E_{0}+\int_{0}^{T^{k}_{0}}(v^{k}\cdot \nabla d^{k},\partial_sd^{k})\,ds.\notag
		\end{align}
		Now we estimate the integral on the right hand side of the last line as follow:
		\begin{align}
		\int_{0}^{T^{k}_{0}}(v^{k}(s)\cdot \nabla d^{k}(s),\partial_{s}d^{k}(s)) ds
		&\le \int_{0}^{T^{k}_{0}}|v^{k}(s)\cdot \nabla d^{k}(s)|_{L^{2}}|\partial_{s}d^{k}(s)|_{L^{2}} ds \notag\\
		&\le\frac{1}{2}\int_{0}^{T^{k}_{0}}|\partial_{s}d^{k}(s)|^{2}_{L^{2}}\,ds+\frac{1}{2}\int_{0}^{T^{k}_{0}}|v^{k}(s)\cdot \nabla d^{k}(s)|^{2}_{L^{2}}\,ds\notag\\
		&\le\frac{1}{2}\int_{0}^{T^{k}_{0}}|\partial_{s}d^{k}(s)|^{2}_{L^{2}}\,ds+
		\frac{1}{2}\left(\int_{0}^{T^{k}_{0}}|v^{k}(s)|^{4}_{L^{4}}\,ds\right)^{\frac{1}{2}}\left(\int_{0}^{T^{k}_{0}}|\nabla d^{k}(s)|^{4}_{L^{4}}\,ds\right)^{\frac{1}{2}}.\notag\\
		\end{align}
		By employing  \eqref{Eq:UnifEstL4} in the last line we get
		\begin{equation}\notag
		\int_{0}^{T^{k}_{0}}(v^{k}\cdot \nabla d^{k},\partial_{t}d^{k})	 dt\le\frac{1}{2}\int_{0}^{T^{k}_{0}}|\partial_{t}d^{k}|^{2}_{L^{2}}+\frac{C}{2}\varepsilon^{2}_{1}\Xi_{0}.
		\end{equation}
		Summing up  these discussion, we get
		\begin{equation}
		\sup_{k \in \mathbb{N}}\int_{0}^{T^{k}_{0}}|\partial_{t}d^{k}|^{2}_{L^{2}}dt \le E_{0}+\frac{1}{2}\sup_{k \in \mathbb{N}}\int_{0}^{T^{k}_{0}}|\partial_{t}d^{k}|^{2}_{L^{2}}\,dt+ C\varepsilon^{2}_{1}\Xi_{0}.\notag
		\end{equation}
		Thus, we obtain the following uniform estimate of $\partial_t d$
		\begin{equation}\notag
		\sup_{k \in \mathbb{N}}\int_{0}^{T^{k}_{0}}|\partial_{t}d^{k}|^{2}_{L^{2}} \,dt\le 2E_{0} + 2C\varepsilon_{1}^{2} \Xi_{0}.
		\end{equation}
		We now estimate the time derivative $\partial_{t}v^{k}$.
		Let $\varphi\in$ V. We have
		\begin{equation}\notag
		(\partial_{t}v^{k},\varphi)=-(\nabla v^{k},\nabla\varphi)-\int_{\Omega}v^{k}\cdot \nabla v^{k}\varphi \;dx-\int_{\Omega}Div(\nabla d^{k}\odot\nabla d^{k})\varphi \;dx +\int_{D}f(t) \;dx.
		\end{equation}
		Hence
		\begin{align}
		|(\partial_{t}v^{k},\varphi)|
		&\le |\nabla v^{k}|_{L^{2}}|\nabla\varphi|_{L^{2}}+\int_{\Omega}|v^{k}|^{2}|\nabla\varphi|\;dx+\int_{\Omega}|\nabla d^{k}|^{2}|\nabla\varphi|\;dx\notag\\
		&\le |\nabla\varphi|_{L^{2}}\left[|\nabla v^{k}|_{L^{2}}+|v^{k}|^{2}_{L^{4}}+|\nabla d^{k}|^{2}_{L^{4}}+|f(t)|_{L^{2}}\right]\notag
		\end{align}
		This altogether with \eqref{Eq:UnifEstL4} and \eqref{Eq:UnifEstL2}  imply that there exists $C>0$ such that for all $k\in \mathbb{N}$
		\begin{align}
		\sup_{k\in \mathbb{N}}\int_{0}^{T^{k}_{0}}|\partial_{t}v^{k}|^{2}_{\ve^{\ast}}
		&\le 4\left[\int_{0}^{T^{k}_{0}}|\nabla v^{k}|^{2}_{L^{2}} \,ds+\int_{0}^{T^{k}_{0}}|v^{k}|^{4}_{L^{4}}\,ds+\int_{0}^{T^{k}_{0}}|\nabla d^{k}|^{4}_{L^{4}}\,ds+\int_{0}^{T^{k}_{0}}|f(t)|_{L^{2}}^{2} dt\right]\notag\\
		&\le C.\notag
		\end{align}
		
		Now, let us set $$T_{0}=\inf_{k\ge 1}T^{k}_{0}.$$ Since for all $k\ge 1$ $$T^{k}_{0}\ge \theta_{0}(\eps_1, E_0) \frac{\varrho^2_0}{(\varrho_0^\frac12 + 1)^4} ,$$ then by the definition of $T_{0}$, we get $$T_{0}\ge\theta_{0}(\eps_1, E_0) \frac{\varrho^2_0}{(\varrho_0^\frac12 + 1)^4} .$$
		It then follows from the previous analysis that the sequence $\left((v^{k},d^{k})\right)_{k \in \mathbb{N}}$ is bounded in $C([0,T_{0}];\rH\times \rH^{1})\cap L^{2}(0,T_{0};\ve \times \rH^{2})$ and the sequence $\left((\partial_{t}v^{k},\partial_{t}d^{k})\right)_{k \in \mathbb{N}}$ is bounded in $L^{2}(0,T_{0};\ve^{\ast}\times L^{2})$.
		Hence by Aubin-Lions compactness lemma and Banach-Alaoglu theorem, one can extract a subsequence $\left((v^{k_{j}},d^{k_{j}})\right)_{j \in \mathbb{N}}$ from $\left((v^{k},d^{k})\right)_{k \in \mathbb{N}}$ and find $(v,d)$ such that as $j \to \infty$
		\begin{equation}
	\begin{split}
		&(v^{k_{j}},d^{k_{j}})\to (v,d)~ \text{weakly in}~ L^{2}(0,T_{0};\ve\times \rH^{2}),\\
		&(v^{k_{j}},d^{k_{j}})\to (v,d)~ \text{weakly star}~ \text{in}~ L^{\infty}(0,T_{0}; \rH\times \rH^{1}),\\
		&(v^{k_{j}},d^{k_{j}})\to (v,d)~ \text{strongly in}~ L^{2}(0,T_{0};D(\rA^{\frac{\theta}{2}})\times \rH^{1+\theta})
		~\text{for any}~ \theta\in [0,1)\label{fiu}.
	\end{split}
		\end{equation}
		Let $t\in [0,T_0]$. Then, the sequence $(v^{k}(t), d^k(t))_{k \in \mathbb{N}}$ is bounded in $\rH\times \rH^1$. Hence, thanks to the compact embedding $\rH^1\hookrightarrow L^2$ we can and we will assume that the subsequence $\left((v^{k_{j}},d^{k_{j}})\right)_{j \in \mathbb{N}}$ satisfies, for all $t\in [0,T_0]$
		\begin{equation}
		d^{k_j}(t) \to d(t) \text{ strongly in } L^2.\notag
		\end{equation}
		This and the fact $d^{k_j}(t)\in \mathcal{M}$ for all $j \in \mathbb{N}$, $t \in [0,T_0]$ implies that  there exists a constant $C>0$
		such that for all $j \in \mathbb{N}$, $t \in [0,T_0]$
		\begin{align}
		\int_\Omega \lvert 1-\rvert d(t,x) \rvert^2\rvert \;dx =& \int_\Omega \lvert \rvert d^{k_j}(t,x)\rvert^2 -\lvert d(t,x)\rvert^2\rvert \;dx\notag\\
		\le & \lvert d^{k_j}(t,x) -d(t,x)\rvert_{L^2} \lvert d^{k_j}(t,x) + d(t,x)\rvert_{L^2}\notag\\
		\le & C \lvert d^{k_j}(t,x) - d(t,x)\rvert_{L^2} .\notag
		\end{align}
		Passing to the limit as $j\to \infty$ yields
		\begin{equation*}
		\int_\Omega \lvert 1-\rvert d(t,x) \rvert^2\rvert \;dx=0, \text{  for all $t\in [0,T_0]$}.\notag
		\end{equation*}
		Hence, for all $t\in [0,T_0]$
		\begin{equation}\label{Eq:SphereCondition}
	d(t)\in \mathcal{M}.
		\end{equation}
		
		Our next step is to show that the limit $(v,d)$ satisfies the system \eqref{1b}. Hence, we need to pass to the limit in the nonlinear terms. In order to do this, we firstly observe that the
		convergences
		\begin{align}
		&v^{k_{j}}\cdot \nabla v^{k_{j}}\rightharpoonup v\cdot \nabla v~ \text{in}~ L^{2}(0,T_{0};\ve^{\ast})\notag\\
		&-\Div(\nabla d^{k_{j}}\odot\nabla d^{k_{j}})\rightharpoonup -\Div(\nabla d\odot\nabla d)~ \text{in}~ L^{2}(0,T_{0};\ve^{\ast})\notag\\
		&	v^{k_{j}}\cdot \nabla d^{k_{j}}\rightharpoonup v\cdot \nabla d ~\text{in}~ L^{2}(0,T_{0}; L^{2}).\notag
		\end{align}
		are now well-known, see, for instance, \cite{Temam_2001} or \cite{ZB+EH+PR-SPDE_2019}, hence we omit their proof.
		
		Secondly, the most difficult point is the convergence
		\begin{equation}
		\lim_{k\to\infty}\int_{0}^{T_{0}}||\nabla d^{k}|^{2}d^{k}-|\nabla d|^{2}d|_{L^{2}}dt=0, \label{fiu1}
		\end{equation}
		and hence we prove it here. For this quest we notice that  there exists a constant $C>0$ such that  for all $k\ge 1$
		\begin{align}
		\int_{0}^{T_{0}}||\nabla d^{k}|^{2}d^{k}-|\nabla d|^{2}d|^{2}_{L^{2}}dt
		&\le C\int_{0}^{T_0}|(\nabla d^{k}-\nabla d):(\nabla d^{k}+\nabla d)d^{k}|_{L^{2}} dt \notag\\
		&+ C\int_{0}^{T_0}||\nabla d|(d^{k}-d)|_{L^{2}}dt.\notag
		\end{align}
		By the H\"older  inequality, we have
		\begin{align}
		\int_{0}^{T_{0}}||\nabla d^{k}|^{2}d^{k}-|\nabla d|^{2}d|_{L^{2}}\,ds
		&\le C\left(\int_{0}^{T_{0}}|\nabla d^{k}-\nabla d|^{2}_{L^{4}}\,ds\right)^{\frac{1}{2}}
		\left(\int_{0}^{T_{0}}(|\nabla d^{k}|^{2}_{L^{4}}+|\nabla d|^{2}_{L^{4}})\,ds\right)^{\frac{1}{2}}\sup_{t\in [0,T_0]}|d^{k}(t)|_{L^{\infty}} \notag\\
		&+ \left(\int_{0}^{T_{0}}|\nabla d|^{2}_{L^{4}}\,ds\right)^{\frac{1}{2}}\left(\int_{0}^{T_{0}}|d^{k}-d|^{2}_{L^{\infty}}\,ds\right)^{\frac{1}{2}}\notag
		\end{align}
	By the Ladyzhenskaya inequality (\cite[Lemma III.3.3]{Temam_2001}) and the Sobolev embedding $\rH^{1+\theta}\hookrightarrow L^{\infty} (\theta\in (0,1))$, we arrive at
		\begin{align}
		\int_{0}^{T_{0}}||\nabla d^{k}|^{2}d^{k}-|\nabla d|^{2}d|_{L^{2}}\,ds
		&\le C\left(\int_{0}^{T_{0}}|\nabla d^{k}-\nabla d|_{L^{2}}|\Delta d^{k}-\Delta d|_{L^{2}}\,ds\right)^{\frac{1}{2}}\notag\\
		&\qquad +C\left(\int_{0}^{T_{0}}|d^{k}-d|^{2}_{\rH^{1+\theta}}\,ds\right)^{\frac{1}{2}}\notag\\
		&\le \left(\int_{0}^{T_{0}}|\nabla d^{k}-\nabla d|^{2}_{L^{2}}\,ds\right)^{\frac{1}{2}}\left(\int_{0}^{T_{0}}|\Delta d^{k}-\Delta d|^{2}_{L^{2}}\,ds\right)^{\frac{1}{2}} \notag \\
		&\qquad + C\left(\int_{0}^{T_{0}}|d^{k}-d|^{2}_{\rH^{1+\theta}}\,ds\right)^{\frac{1}{2}}\notag
		\end{align}
		where we have also used the fact that
		\[
		\sup_{k}\int_{0}^{T}|\nabla d^{k}|^{2}_{L^{4}}\,ds<\infty~ \text{and}~ d_{k}(t)\in \mathcal{M},\, t\in [0,T_0).
		\]
		The strong convergence (\ref{fiu}) and the  fact that
		\[
		\sup_{k}\left(\int_{0}^{T_{0}}|\Delta d^{k}-\Delta d|^{2}_{L^{2}}\,ds\right)^{\frac{1}{2}}<\infty
		\]
		completes the proof of  (\ref{fiu1}).
		
		\noindent We now study the convergence of the term $\phi^\prime (d^{k})$.
		Since $d^{k_{j}}$ $\to$ d strongly in $L^{2}(0,T_{0};\rH^{1})$, we can assume that
		$d^{k_{j}}\to$ $d$ ~a.e.~ $(t,x)\in [0,T_{0}]\times\Omega$. Thus,
		by the continuity of $\phi^\prime(.)$, we get
		\begin{equation}
		\phi^\prime(d^{k_{j}})\to \phi^\prime(d)~ a.e.~ (t,x)\in [0,T_{0})\times\Omega.\notag
		\end{equation}
		Since $d^{k_{j}}(t)\in \mathcal{M}, t\in [0,T_0)$  and, by assumption,  $|\phi^\prime(d^{k_{j}})|\le M$, the Lebesgue dominated convergence theorem implies that
		\begin{equation}
		\phi^\prime(d^{k_{j}})\to \phi^\prime(d)~ \text{in}~ L^{2}([0,T_{0})\times\Omega)=L^{2}(0,T_{0};L^{2}).\notag
		\end{equation}
	
		Since $d^{k_{j}}(t) \in \mathcal{M}, \,t\in [0,T_0)$, $|\phi^\prime(d^{k_{j}})|\le M$, we have
		\begin{align}
		&\int_{0}^{T_{0}}|(d^{k_{j}},\phi^\prime(d^{k_{j}}))d^{k_{j}}-(d,\phi^\prime(d))d|^{2}_{L^{2}}\notag\\
		&\le \int_{0}^{T_{0}}|(d^{k_{j}}-d,\phi^\prime(d^{k_{j}}))d^{k_{j}}+ (d,\phi^\prime(d^{k_{j}})-\phi^\prime(d))d^{k_{j}}|^{2}_{L^{2}}\notag\\
		&+\int_{0}^{T_{0}}|(d,\phi^\prime(d))(d^{k_{j}}-d)|^{2}_{L^{2}}\notag\\
		&\le M\int_{0}^{T_{0}}|d^{k_{j}}-d|^{2}_{L^{2}}dt+\int_{0}^{T_{0}}|\phi^\prime(d^{k_{j}})-\phi^\prime(d)|^{2}_{L^{2}}dt\notag\\
		& +\int_{0}^{T_{0}}|d^{k_{j}}-d|^{2}_{L^{2}}dt.\notag
		\end{align}
		By the convergence $\phi^\prime(d^{k_{j}})\to\phi^\prime(d)$ strongly in $L^{2}(0,T_{0}; L^{2})$ and $d^{k_{j}}\to d$ strongly in $L^{2}(0,T_{0};L^{\infty})$, we obtain
		\begin{equation}
		\alpha(d^{k_{j}})d^{k_{j}}\to \alpha(d)d ~\text{in}~ L^{2}(0,T_{0};L^{2}).\notag
		\end{equation}
		Since
		\begin{align}
		&\int_{0}^{T_{0}}|d^{k_{j}}\times g(t)-d\times g(t)|_{L^{2}}dt\notag\\
		&=\int_{0}^{T_{0}}|(d^{k_{j}}-d)\times g(t)|_{L^{2}}dt\notag\\
		&\le\int_{0}^{T_{0}}|d^{k_{j}}-d|_{L^{\infty}}|g(t)|_{L^{2}}dt\notag\\
		&\le\int_{0}^{T_{0}}|d^{k^{j}}-d|^{2}_{L^{\infty}}dt\int_{0}^{T_{0}}|g(t)|^{2}_{L^{2}}dt,\notag
		\end{align}
		and  $d^{k_{j}}\to d$ strongly in $L^{2}(0,T_{0};L^{\infty})$, we get
		\begin{equation}
		d^{k_{j}}\times g \to d\times g~ \text{in}~ L^{1}(0,T_{0};L^{\infty}).\notag
		\end{equation}
		
		We will now prove that $(v,d)$ satisfies the initial conditions and that $(v,d)\in C([0,T_0]; \rH\times \rH^1)$. Towards these goals, we first observe that
		since $(v,d)\in L^{\infty}(0,T_{0};\rH\times \rH^{1})$ and $$(\partial_{t}v^{k_{j}},\partial_{t}d^{k_{j}}) \rightharpoonup (\partial_{t}v,\partial_{t}d)\text{ in } L^{2}(0,T_{0};\ve^{\ast}\times L^{2}),$$  then by the Strauss theorem, see \cite[Lemma III.1.2]{Temam_2001},  we get
		\[
		(v,d)\in C([0,T_{0}];\rH_{w}\times \rH^{1}_{w}),
		\]
		
		and $d\in C([0,T_{0}];L^{2})$.  Hence,
		\begin{align}
		\lim_{t\to 0}(v(t),\varphi)=(v_{0},\varphi), \forall \varphi\in \rH,\label{Eq:Weak-ConvInitDatavelo}\\
		\lim_{t\to 0}(\nabla d(t),\Psi)=(\nabla d_{0},\Psi), \forall \Psi\in L^{2},\label{Eq:Weak-ConvInitDataOptDir}\\
		\lim_{t\to 0}\phi(d(t))=\phi(d_{0})~ \text{in}~ L^{2}.\label{Eq:StrongConve}
		\end{align}
		From all these passages to the limits we see that the limit $v$ and $d$ satisfy the equations \eqref{eq:LocVelo} and \eqref{eq:LocVelo} in $\ve^\ast$ and $L^2$, respectively.
		
		The estimates \eqref{Eq:UnifEstStrongSol-1} and \eqref{Eq:UnifEstStrongSol-2} are established by passing to the limit and  using the weak lower semicontinuity of the norms in
		the estimates \eqref{Eq:UnifEstNRG}, \eqref{Eq:UnifEstL2} and \eqref{Eq:UnifEstLaplace}.
		
		What remains to prove is the continuity of $(v,d):[0,T_0] \to \rH\times D(\hA^\frac12)$.
		For this we will firstly establish that
		\begin{equation}\label{Eq:EstDerrivativeintime-1}
		(\partial_tv, \partial d) \in L^2(0,T; \ve^\ast \times L^2).
		\end{equation}
		Towards this aim we recall that it was proved in \cite[Proofs of  Eqs (3.27), (3.28), (3.31) and (3.32)]{ZB+EH+PR-SPDE_2019}  that there exists a constant $C>0$ such that for $(v,d)\in C([0,T_0);\rH \times D(\hA^\frac12) )\cap L^2(0,T_0; \ve\times D(\hA))$
		\begin{align}
		\lvert -\rA v -B(v,v)-\Pi(\Div [\nabla d \odot \nabla d]) \rvert_{L^2(0,T_0; \ve^\ast)} \le C,\notag\\
		\lvert -\hA d -v \cdot \nabla d + d\times g \rvert_{L^2(0,T_0; L^2)} \le C.\notag
		\end{align}
		Note that the estimates in \cite[Eqs (3.27), (3.28), (3.31) and (3.32)]{ZB+EH+PR-SPDE_2019} are for $L^2(0,T_0; \ve^\ast)$- and $L^2(0,T; L^2)$-valued random variables, but they remain valid for deterministic $L^2(0,T_0; \ve^\ast)$ and $L^2(0,T; L^2)$ functions.
		Since $d(t)\in \mathcal{M}$, $t\in [0,T_0)$,  we easily infer that there exists a constant $C>0$ such that
		\begin{align}
		\lvert \lvert \nabla d \rvert^2 d \rvert^2_{L^2(0,T_0;L^2)} \le \lvert \nabla d \rvert^4_{L^4(0,T_0; L^4)} \le \sup_{t\in [0,T_0]} \lvert \nabla d(t) \rvert^2 \int_0^{T_0} \lvert d(s) \rvert^2_{\rH^2} ds< C.\notag
		\end{align}
		Using \eqref{Eq:LInearGrowthPhiprime} and the constraint $d(t)\in \mathcal{M}$, $t\in [0,T_0)$,  easily implies that there exists a constant $C>0$ such that
		\begin{equation}
		\lvert -\phi^\prime(d) + (\phi^\prime(d) \cdot d) d \rvert_{L^2(0,T_0; L^2)}\le C.\notag
		\end{equation}
		These estimates completes the proof of  \eqref{Eq:EstDerrivativeintime-1}.
		
		Secondly, from \eqref{Eq:EstDerrivativeintime-1}, the fact $(v,d)\in L^2(0,T_0; \ve\times D(\hA))$ and \cite[Lemma 3.1.2]{Temam_2001} we infer that
		$$ (v,d) \in C([0,T_0]; \rH\times D(\hA^\frac12).$$
		This completes the proof of Proposition \ref{Ma2}.

	\end{proof}
	Now, we proceed to the proof of Theorem \ref{Ma1}.
	\begin{proof}[Proof of Theorem \ref{Ma1}]
		In order to prove the theorem, let us denote by $\Sigma$ the set of local  solutions to problem \eqref{1b}.  By Proposition \ref{Ma2}, the set $\Sigma$ is non empty and we can and will assume that the time of existence $T_0$ of any local solution $(v,d)\in \Sigma$ satisfies the property \eqref{eqn-t_o-2} stated in Proposition \ref{Ma2}.
		Let us define the relation $\lesssim$ on $\Sigma$ by
		$$(y_{1};\sigma_{1})\lesssim (y_{2};\sigma_{2})\text{ if } \sigma_{1}\le\sigma_{2} \text{ and } y_{2}=y_{1} \text{ on } [0,\sigma_{1}] , \text{ for all }
		(y_{i};\sigma_{i}):=((u_{i},d_{i});\sigma_{i})\in\Sigma, i=1,2.$$
		We briefly show below that the relation $\lesssim$ is reflexive, antisymmetric and transitive.
		
		\noindent	\textbf{Reflexivity}. Let $(y_{1};\sigma_{1})\in\Sigma$. Then $(y_{1};\sigma_{1})\lesssim (y_{1};\sigma_{1})$ by definition.
		
		\noindent	\textbf{Antisymmetry}. Let $(y_{i};\sigma_{i})\in\Sigma$, $i=1,2$. Suppose that $(y_{1};\sigma_{1})\lesssim (y_{2};\sigma_{2})$ and $(y_{2};\sigma_{2})\lesssim (y_{1};\sigma_{1})$. This implies that
		$y_{2}=y_{1}$ on $[0,\sigma_{1}]$, $y_{1}=y_{2}$ on $[0,\sigma_{2}]$ and $\sigma_{1}=\sigma_{2}$. This proves the antisymmetric property of $\lesssim$.
		
		\noindent	\textbf{Transitivity}. Let $(y_{i},\sigma_{i})\in\Sigma$, $i=1,2,3$. Suppose that $(y_{1};\sigma_{1})\lesssim (y_{2};\sigma_{2})$ and $(y_{2};\sigma_{2})\lesssim (y_{3},\sigma_{3})$. We will prove that $(y_{1};\sigma_{1})\lesssim (y_{3};\sigma_{3})$. For this purpose, we observe that
		\begin{align}
		&\sigma_{1}\le\sigma_{2}~ \text{and} ~y_{2}=y_{1}~ \text{on}~ [0,\sigma_{1}],\notag\\
		&\sigma_{2}\le\sigma_{3}~ \text{and}~ y_{3}=y_{2}~ \text{on}~ [0,\sigma_{2}].\notag
		\end{align}
		Hence
		\begin{equation}
		\sigma_{1}\le \sigma_{3},\label{Ma5}
		\end{equation}
		\begin{equation}
		y_{3}=y_{2}=y_{1}~ \text{on}~ [0,\sigma_{1}].\label{Ma6}
		\end{equation}
		The facts (\ref{Ma5}) and (\ref{Ma6}) prove the transitivity property of $\lesssim$.
		
		In order to prove the existence of a maximal element in $\Sigma$ we shall use the Kuratowski-Zorn Lemma. Hence, we need to prove that all increasing chain in $\Sigma$ has an upper bound. For this purpose, let $((v_k, d_k);\sigma_{k})_{k \in \mathbb{N}}$ be an increasing chain in $\Sigma$.
		We will show that this sequence has an upper bound.  In order to do that we set
		$$ \sigma=\sup_{k \in \mathbb{N} } \sigma_k$$ and define $(v,d):[0,\sigma) \to \rH\times \rH^1$ by
		\begin{equation}
		(v,d)_{\vert_{[0,\sigma_k)}}=(v_k,d_k).\notag
		\end{equation}
		From these definitions it is clear that for all $k$ $\sigma_k\le\sigma $ and $(v,d)=(v_k,d_k)$ on $[0,\sigma_k)$ for all $k \in \mathbb{N}$.  Hence, $((v,d);\sigma)\in \Sigma$ and it is an upper bound of the increasing chain $((v_k, d_k);\sigma_{k})_{k \in \mathbb{N}}$.
		Therefore, it now follows from the Kuratowski-Zorn lemma that $\Sigma$ has a maximal element which is a maximal local solution to \eqref{1b}.  This completes the proof of Theorem \ref{Ma1}.
	\end{proof}
	The following result  gives an important property of the energy of the maximal solution $((v,d); T_{\ast})$ near the point $T_\ast$.
	\begin{Prop}
		Let $\eps_0>0$, $\varrho_0>0$ and $\theta_0$ be the constants and the function from  Proposition \ref{Ma2} and
		$$ \eps_1^2= 2\mathcal{E}_{2\varrho_0}(v_0,d_0).$$ Let   $((v,d);T_{\ast})$  be the maximal solution defined in Theorem \ref{Ma1}. Then,
		\begin{equation}\label{Eq:LossNRJ}
		\limsup_{t\toup T_{\ast}}\sup_{x\in\Omega}\int_{B_{R}(x)}\left[|v(t)|^{2}+|\nabla d(t)|^{2}+\phi(d(t))\right]dy\ge\varepsilon^{2}_{1},
		\end{equation}
		for all $R\in (0, \varrho_0]$.
	\end{Prop}
	\begin{proof}
		We prove the proposition by contradiction. Suppose that there exists $R>0$ such that
		\[
		\limsup_{t\toup T_{\ast}}\sup_{x\in\Omega}\int_{B_{R}(x)}\left[|v(t)|^{2}+|\nabla d(t)|^{2}+\phi(d(t))\right]dy<\varepsilon^{2}_{1}.
		\]
		Thus, there exists an increasing sequence $(t_{n})_{n\in\mathbb{N}}$ such that
		\begin{equation}
		t_{n}\toup T_{\ast}~ \text{as}~ n\to\infty,\label{Ma10-0}
		\end{equation}
		and
		\begin{equation}\label{Ma10}
		\lim_{n\to\infty}\sup_{x\in\Omega}\int_{B_{R}(x)}\left[|v(t_{n})|^{2}+
		|\nabla d(t_{n})|^{2}+\phi(d(t_{n}))\right]dy<\varepsilon^{2}_{1}.
		\end{equation}
		From (\ref{Ma10-0}) and \eqref{Ma10} we deduce that
		\begin{itemize}
			\item[(a)] $\forall$ $\delta>0$, there exists $m\in\mathbb{N}$ such that $0<T_{\ast}-t_{m}<\delta$,
			\item[(b)] we can and will  assume that  for all $n\in\mathbb{N}$
			\begin{equation}
			\sup_{x\in\Omega}\int_{B_{R}(x)}\left[|v(t_{n})|^{2}+|\nabla d(t_{n})|^{2}+\phi(d(t_{n}))\right]dy<\varepsilon^{2}_{1}.\notag
			\end{equation}
		\end{itemize}
		By the global energy inequality \eqref{Eq:NRJ-Inequality}, we get
		\begin{equation}
		E_{\ast}=\sup_{t<T_{\ast}}\mathcal{E}(v(t),d(t))<\infty.\label{Ma10-1}
		\end{equation}
		By (a), for $\delta=\frac{1}{2}\theta_{0}(\varepsilon^{2}_{1},E_{\ast})R^{2}$, there exists $m\in\mathbb{N}$ such that
		\begin{equation}
		0<T_{\ast}-t_{m}<\frac{1}{2}\theta_{0}(\varepsilon^{2}_{1},E_{\ast})R^{2}_{0}.\label{Ma11}
		\end{equation}
		By (b), (\ref{Ma10-1}) and \eqref{Ma10} we get
		\begin{equation}
		\mathcal{E}(v(t_{m}),d(t_{m}))\le E_{\ast},\notag
		\end{equation}
		and
		\begin{equation}
		\sup_{x\in\Omega}\int_{B_{R}(x)}\left[|v(t_{m})|^{2}+|\nabla d(t_{m})|^{2}+\phi(d(t_{m}))\right]dy<\varepsilon^{2}_{1}.\notag
		\end{equation}
		Hence,  by  Proposition \ref{Ma2}, there exists a solution $(\tilde{v},\tilde{d})$ defined on $[t_{m},\tau +t_{m}]$ with
		$$\tau\ge\theta_{0}(\varepsilon^{2}_{1},\mathcal{E}(v(t_{m}),d(t_{m}))R^{2}_{0}.$$ But observe that $\theta_{0}(\varepsilon^{2}_{1},E_{0})$ is a non-increasing function of the initial energy $E_{0}$. Hence, by (\ref{Ma11})
		\begin{equation}
		\tau\ge \theta_{0}(\varepsilon^{2}_{1},\mathcal{E}(v(t_{m}),d(t_{m}))R^{2}_{0}\ge\theta_{0}(\varepsilon^{2}_{1},E_{\ast})R^{2}_{0}>2(T_{\ast}-t_{m}).\notag
		\end{equation}
		By doing elementary calculation we obtain
		\begin{align}
		t_{m}+\tau
		&\ge 2(T_{\ast}-t_{m})+t_{m}\notag\\
		&\ge T_{\ast}+T_{\ast}-t_{m}  \notag\\
		&>T_{\ast}~ (~\text{because}~  T_{\ast}-t_{m}>0, \text{ see  \eqref{Ma11}}).\notag
		\end{align}
		Hence, we get the existence of a local solution $(\tilde{v},\tilde{d}):[0,\tau+t_{m})\to \rH\times \rH^{1}$ with $\tau+t_{m}>T_{\ast}$ and
		$(\tilde{v},\tilde{d})=(v,d)~ \text{on}~ [0,T_{\ast})$. This contradicts the fact that $((v,d);T_{\ast})$ is a maximal solution.
	\end{proof}
	
	We now give the promised proof of Theorem \ref{thm-main}.
	\begin{proof}[Proof of Theorem \ref{thm-main}]
		Let $((v,d),T_{\ast})$ be the maximal local strong solution to \eqref{1b} constructed from Theorem \ref{Ma1}. Firstly, we set $T_\ast=T_1$ and prove the following result.
		\begin{Lem}\label{Lem:Continuity}
			The maximal local strong solution $((v,d),T_1)$ satisfies
			\begin{align}
			(\partial_t v, \partial_td) \in L^2(0,T_1; D(\rA^{-\frac32}) \times D(\hA)^{\ast} ),\notag\\
			(v,d)\in C([0,T_1 ];\rH \times L^{2}(\Omega)).\notag
			\end{align}
		\end{Lem}
	\begin{proof}[Proof of Lemma \ref{Lem:Continuity}]
		We recall that
		\begin{equation*}
		\partial_{t}v+Av + \Pi(v\cdot\nabla v)=-\Pi(\divv (\nabla d\odot\nabla d)) + f.\notag
		\end{equation*}
		Firstly, let us prove that $\partial_{t}v\in L^{2}(0,T_{1};D(\rA^{-\frac32}))$. For this aim, let $\varphi\in D(\rA^\frac32)$  be fixed. Then,
		\begin{align*}
		|\langle {\partial_t } v ,\varphi\rangle|
		&=\left|\int_{\Omega}(\nabla v\cdot \nabla\varphi+v\cdot \nabla v.\varphi)\;dx-\int_{\Omega}\nabla d\odot\nabla d:\nabla\varphi \;dx +\int_{\Omega}f\varphi \;dx\right|\notag\\
		&\le \|\nabla v\|_{L^{2}}\|\nabla\varphi\|_{L^{2}}+\|v\|_{L^{2}}\|\nabla v\|_{L^{2}}\|\varphi\|_{C^{0}(\Omega)}+\|\nabla d\|^{2}_{L^{2}}\|\nabla\varphi\|_{L^{2}}+\|f(t)\|_{\rH^{-1}}\|\nabla \varphi\|_{L^{2}}\notag\\
		&\le\left(\|\nabla v\|_{L^{2}}+\|v\|_{L^{2}}\|\nabla v\|_{L^{2}}+\|\nabla d\|^{2}_{L^{2}}+\|f(t)\|_{\rH^{-1}}\right)\|\varphi\|_{\rH^{3}}\notag\\
		&\le\left(\|\nabla v\|_{L^{2}}+\|v\|_{L^{2}}\|\nabla v\|_{L^{2}}+\|\nabla d\|^{2}_{L^{2}}+\|f(t)\|_{\rH^{-1}}\right)\|\varphi\|_{D(\rA^{\frac32})},\notag
		\end{align*}
		where we  used the fact $\rH^{3}(\Omega)\subset C^{0}(\Omega)$ and $D(\rA^\frac32)\hookrightarrow \rH^3$.
		Then, we deduce that
		\begin{equation*}
		|\partial_tv|_{D(\rA^{-\frac32})}\le \|\nabla v\|_{L^{2}}+\|v\|_{L^{2}}\|\nabla v\|_{L^{2}}+\|\nabla d\|^{2}_{L^{2}}+\|f \|_{\rH^{-1}}.
		\end{equation*}
		The last line, the Assumption \ref{AssumptionMain} and the facts $(v,d)\in L^\infty(0,T_1;\rH\times \rH^1)$ and $v\in L^2(0,T_1; \ve)$ imply that $\partial_t v \in L^{2}(0,T_{1};D(\rA^{-\frac32}))$. Hence, since $v\in L^{2}(0,T_{1};\ve)$ we have  $v\in C([0,T_{1}];L^{2})$.
		
		Secondly, we estimate $\partial_{t}d$. For this purpose, let $\Psi$ $\in$ $D(\hA)$ be fixed. Recall that
		\begin{equation}
		\partial_{t}d=\Delta d + |\nabla d|^{2}d+v\cdot \nabla d-\phi^\prime(d) +\alpha(d)d+d\times g.\notag
		\end{equation}
		Then,
		\begin{align}
		&\left|\langle \Delta d+|\nabla d|^{2}d-v\cdot \nabla,\Psi\rangle\right|\notag\\
		&=\left|\int_{\Omega}(\nabla d\cdot \nabla\Psi + v\cdot \nabla d.\Psi)\;dx-\int_{\Omega}|\nabla d|^{2}d.\Psi \;dx\right|\notag\\
		&\le \left(\|\nabla d\|_{L^{2}}+\|v\|_{L^{2}}\|\nabla d\|_{L^{2}}+\|\nabla d\|^{2}_{L^{2}}\right)\|\Psi\|_{\rH^{2}}.\label{F1}
		\end{align}
		By the boundedness of $\phi^\prime(d)$ on $\mathbb{S}^{2}$, we have
		\begin{equation}
		\left|\langle -\phi^\prime(d)+\alpha(d)d, \Psi\rangle\right|\le C\|\Psi\|_{\rH^{2}}.\label{F2}
		\end{equation}
		We also get
		\begin{equation}
		\left|\langle d\times g,\Psi\rangle\right|\le C|g(t)|_{L^{2}}\|\Psi\|_{\rH^{2}}. \label{F3}
		\end{equation}
			The estimates (\ref{F1})-(\ref{F3}) imply that $\partial_{t}d\in L^{2}(0,T_{1};D(\hA)^\ast)$ because $d\in L^{\infty}(0,T_{1};\rH^{1})$ and $v\in L^{\infty}(0,T_{1};L^{2})$ and  $D(\hA)  \hookrightarrow \rH^2$. The fact $\partial_{t}d\in L^{2}(0,T_{1};D(\hA)^\ast)$ combined  with $d\in L^{2}(0,T_{1};\rH^{1})$ imply that $d\in C^{0}([0,T_{1}];L^{2}(\Omega))$, see \cite[Lemma III.1.4 ]{Temam_2001}.  This ends the proof of the lemma.
			

\end{proof}
		We can now continue with the proof of Theorem \ref{Ma1}.  From Lemma \ref{Lem:Continuity}, we can define
		\begin{equation}
		(v(T_{1}),d(T_{1}))=\lim_{t\toup  T_{1}}\left(v(t),d(t)\right)~ \text{in}~ L^{2}(\Omega)\times L^{2}(\Omega).\notag
		\end{equation}
		Since $\nabla d\in L^{\infty}(0,T_{1};L^{2}(\Omega))$, then $\nabla d(t)\rightharpoonup \nabla d(T_{1})$ weakly in $L^{2}(\Omega)$. This and Theorem \ref{Thm:StraussThm} implies that
		$v(T_{1})\in \rH$ and $d(T_{1})\in \rH^{1}$ and $(v,d)\in C_w([0,T_1]; \rH\times \rH^1)$. Moreover,
		thanks to the strong convergence $d(t)\to d(T_{1})$ in $L^{2}(\Omega)$, we show  using the same idea as in the proof of \eqref{Eq:SphereCondition} that
		\begin{equation}
		d(T_{1})\in \mathcal{M}.\label{F5}
		\end{equation}
	Also, it is not difficult to prove that
		\begin{align}
		\mathcal{E}(v(T_{1}),d(T_{1}))
		&\le \lim_{t\toup  T_{1}}\mathcal{E}(v(t),d(t))\notag\\
		&\le \mathcal{E}(v_{0},d_{0}) + C\left(|f|^{2}_{L^{2}(0,T;\rH^{-1})}+|g|^{2}_{L^{2}(0,T_{1};L^{2})}\right)<\infty. \label{F4}
		\end{align}
		We also prove that
		\begin{equation}\label{F8}
		\mathcal{E}(v(T_{1}),d(T_{1}))\le  \mathcal{E}(v_{0},d_{0}) + \frac12 \left(\int_{0}^{T_{1}}|f(t)|^{2}_{\rH^{-1}}dt+\int_{0}^{T_{1}}|g(t)|^{2}_{L^2}dt\right)-\eps_{1}^{2}.
		\end{equation}
		In fact, the inequality \eqref{Eq:LossOfNRJ} implies that there exists a sequence $t_{n}\toup  T_{1}$ and $x_{0}\in {\Omega}$ such that
		\begin{equation}
		\limsup_{t_{n}\toup  T_{1}}\int_{B_{R}(x_{0})}\left[|v(t_{n})|^{2}+|\nabla d(t_{n})|^{2}+\phi(d(t_{n}))\right]\;dx\ge \eps^{2}_{1}, \forall R\in (0, \varrho_0].\notag
		\end{equation}
		Therefore,
		\begin{align}
		\mathcal{E}(v(T_{1}),d(T_{1}))
		&=\lim_{R\todown 0}\int_{\Omega \setminus B_{R}(x_{0})}\mathcal{E}(v(T_{1}),d(T_{1}))dy\notag\\
		&\le \lim_{R\todown 0}\liminf_{t_{n}\toup  T_{1}}\int_{\Omega \setminus B_{R}(x_{0})}\mathcal{E}(v(t_{n}),d(t_{n}))dy\notag\\
		&\le \lim_{R\todown 0}\left[\liminf_{t_{n}\toup  T_{1}}\int_{\Omega}\mathcal{E}(v(t_{n}),d(t_{n}))dy-\limsup_{t_{n}\toup  T_{1}}\int_{B_{R}(x_{0})}\mathcal{E}(v(t_{n}),d(t_{n}))dy\right]\notag\\
		&\le \liminf_{t_{n}\toup  T_{1}}\mathcal{E}(v(t_{n}),d(t_{n}))-\eps_{1}^{2}\notag\\
		&\le \mathcal{E}(v_0,d_0)+ \frac12 \left(\int_{0}^{T_{1}}|f(t)|^{2}dt +\int_{0}^{T_{1}}|g(t)|^{2}dt\right)-\eps_{1}^{2}.\notag
		\end{align}
		This completes the proof of (\ref{F8}).
		
		With (\ref{F4}) and (\ref{F5}) at hand, one can apply Proposition \ref{Ma2} to construct a local strong solution $(\tilde{v},\tilde{d}):[T_{1}.T_{2}]\to \rH\times \rH^{1}$ with the initial data $(\tilde{v}(T_{1}),\tilde{d}(T_{1}))=(v(T_{1}),d(T_{1}))\in \rH\times \rH^{1}$.
		Moreover, there are constants $\varrho_0>0$,  $\eps_2\in (0,\eps_0)$ such that
		\begin{equation}
		\limsup_{t\toup  T_{2}}\sup_{x\in{\Omega}}\int_{B_{R}(x)}\mathcal{E}(\tilde{v}(t),\tilde{d}(t))\ge \eps^{2}_{2},~~ \forall R\in (0, \varrho_0].\label{F6}
		\end{equation}
		Furthermore,
		\begin{equation}\label{F8-b}
		\mathcal{E}(\tilde{v}(T_{2}),\tilde{d}(T_{2}))< \mathcal{E}(v(T_1),d(T_1)) + \frac12 \left(\int_{T_1}^{T_{2}}|f(t)|^{2}_{\rH^{-1}}dt+\int_{T_1}^{T_{2}}|g(t)|^{2}_{L^2}dt\right)-\eps_{2}^{2}.
		\end{equation}
		Hence, we can construct a sequence of maximal local strong solutions $((v_i, d_i);T_i)_{i=1}^L$ satisfying:  there are constants\footnote{In fact, $ \tilde{\eps}_1=\min\{\eps_i;\,1\le i\le L \}$} $\varrho_0>0$,  $\tilde{\eps}_1\in (0,\eps_0)$ such that for all $1\le i\le L$
		\begin{align}
		\limsup_{t\toup  T_{i}}\sup_{x\in{\Omega}}\int_{B_{R}(x)}\mathcal{E}(v_i(t),d_i(t))\ge \tilde{\eps}^{2}_{1},~~ \forall R\in (0, \varrho_0].\label{F6-b}\\
		\mathcal{E}(v_i(T_{i}),d_i(T_{i}))\le \mathcal(v_i(T_{i-1}),d_i(T_{i-1}))+ \frac12\left( \int_{T_{i-1}}^{T_{i}}|f(t)|^{2}_{\rH^{-1}}dt +\int_{T_{i-1}}^{T_{i}}|g(t)|^{2}_{L^2}dt\right)-\tilde{\eps}_1^2.\label{Eq:LossOfNRJ-0}
		\end{align}
		In order to construct the global solution we consider a function $(\mathbf{v},\mathbf{d})$ defined
		\begin{equation}\label{Eq:DefGlobalWeakSol}
		(\mathbf{v},\mathbf{d})_{\lvert_{[T_{i-1}, T_i)} } =(v_i,d_i), \;\; 1\le i \le L,
		\end{equation}
		and  prove that $L<\infty$. Towards this aim, we first deduce from \eqref{Eq:LossOfNRJ-0} that
		\begin{equation}\label{Eq:LossOfNRJ}
		\mathcal{E}(\mathbf{v}(T_{i}),\mathbf{d}(T_{i}))\le \mathcal(\mathbf{v}(T_{i-1}),\mathbf{d}(T_{i-1}))+ \frac12\left( \int_{T_{i-1}}^{T_{i}}|f(t)|^{2}_{\rH^{-1}}dt +\int_{T_{i-1}}^{T_{i}}|g(t)|^{2}_{L^2}dt\right)-\tilde{\eps}_1^2.
		\end{equation}
		for $i=1,\ldots ,L$. Here $T_{0}=0$.
		Iterating the estimate \eqref{Eq:LossOfNRJ} yields
		\begin{equation}
		\mathcal{E}(\mathbf{v}(T_{L}), \mathbf{d}(T_{L}))\le \mathcal{E}(v_{0},d_{0})+ C\left(\int_{0}^{T_{L}}(|f(t)|^{2}_{\rH^{-1}}+|g(t)|^{2}_{L^{2}})dt\right) -L\tilde{\eps}^{2}_{1}.
		\end{equation}
		This implies that
		\begin{equation}
		L\le\frac{1}{\tilde{\eps}_{1}^{2}}\left[\mathcal{E}(v_{0},d_{0})+ C\left(\int_{0}^{T}(|f(t)|^{2}_{\rH^{-1}}+|g(t)|^{2}_{L^{2}})dt\right)\right]<\infty.\notag
		\end{equation}
		In order to  complete the proof we need to check that $(\mathbf{v},\mathbf{d})$ is indeed a global weak solution. But this follows from the definition \eqref{Eq:DefGlobalWeakSol},
		the fact that each $((v_i,d_i);T_i)$ are maximal local strong solution defined on $[T_{i-1}, T_i)$ and satisfying
		\begin{align}
		&	(v_i,d_i)\in C([T_{i-1}, T_i);\rH\times \rH^1  ),\notag\\
		&	(\partial_tv_i, \partial_t d_i) \in L^2(T_{i-1}, T_i; D(\rA^{-\frac32})\times D(\hA)^\ast),\notag\\
		&	 d_i(t) \in \mathcal{M} \text{ for all } t \in [ T_{i-1}, T_i).\notag
		\end{align}
		
	\end{proof}
	\section{On the regularity and the set of singular times of the solutions when the data is small}\label{Sec:RegCompactSmallData}
	In this section we prove that the set of singular time reduces to the final time horizon $T$ when the data are small enough.
	Let us start with the following conditional regularity of a strong solution $((v,d);T_\ast)$.
	
	\begin{Prop}\label{prop-reg-zb}
		Let $(v_0,d_0)\in \ve\times D(\hA)$, $T>0$,  $f\in L^2(0,T; L^2)$, $g\in L^2(0,T; \rH^1)$ and $(v,d)$ be a  strong solution to \eqref{1b}  such that
		\begin{equation}\notag
		\begin{split}
		v\in C([0,T];\rH) \cap L^2(0,T;\rV)\\
		d\in C([0,T];\rH^1) \cap L^2(0,T;D(\hA)).
		\end{split}
		\end{equation}
		Then,
		\begin{equation}\notag
		\begin{split}
		v\in C([0,T];\rV) \cap L^2(0,T;D(\rA))\\
		d\in C([0,T]; D(\hA)) \cap L^2(0,T;D(\hA^{3/2})).
		\end{split}
		\end{equation}
	\end{Prop}

	\begin{proof}[Proof of Proposition \ref{prop-reg-zb}]
		We start the proof with the following claim.
		\begin{Clm}\label{Clm:Regular-Sol}
			There exist constants $K>0$ and $K_7>0$  depending only on the norms of $(v_0,d_0)\in \ve \times D(\hA^\frac12)$ and the norms of $(f,g)\in L^2([0,T]; \rH\times \rH^1)$  such that the following holds.
		If  $(v,d)$ is a  regular solution on
		some interval $[0,T_\ast] \subset [0,T]$ such  that
			\begin{equation}\label{EQ:AssumL4Norms}
			\int_0^T \lvert \nabla d(s)\rvert^4_{L^4} ds \le K_7 \text{ and } \int_0^T \lvert v(s)\rvert^4_{L^4} ds \le K_7,
			\end{equation}
		then,
		\begin{equation}\label{Eq:Derivativehigheroerdernorms-Fin-04}
		\begin{split}
		\lvert \Delta d (t) \rvert^2_{L^2} + \lvert \nabla v(t) \rvert^2_{L^2}   \le K,\;\;  t \in [0,T_\ast],\\
		\int_0^{T_\ast}\left(\lvert \rA v (r) \rvert^2_{L^2} + \lvert \nabla \Delta d(r) \rvert^2_{L^2}\right)dr \le K.
		\end{split}
		\end{equation}
		\end{Clm}
	\begin{proof}[Proof of Claim \ref{Clm:Regular-Sol}]
			Let $(v,d)$ be a regular solution to \eqref{1b} on $[0,T_\ast]\subset [0,T]$. Then,  we infer from Lemma \ref{Lem:EstinHigherorderNorms} that there exists a constant $K_3>0$ which depends on the norms of $(v_0,d_0)\in \ve\times \rH^1$ and $(f,g)\in L^2(0,T; H\times \rH^1)$  such that
		\begin{equation}\label{Eq:EstinHigherorderNorms-6}
		\begin{split}
		\sup_{0\le s\le \tau} \left(\lvert \nabla v(s) \rvert^2_{L^2} + \lvert \Delta d(s) \rvert^2_{L^2} \right) + 2 \int_0^{\tau} \left( \lvert \nabla \Delta d(s) \rvert^2_{L^2} + \lvert A v(s) \rvert^2_{L^2}\right)ds \\
		\le K_3e^{K_3 \int_0^{\tau}[\lvert \nabla d(r) \rvert^4_{L^4}+ \lvert v(r) \rvert^4_{L^4}+ \lvert \nabla d(r) \rvert^2_{L^4}+ \lvert v(r) \rvert^2_{L^4}]dr   }.
		\end{split}
		\end{equation}
		This and the assumption \eqref{EQ:AssumL4Norms}  implies the desired inequality \eqref{Eq:Derivativehigheroerdernorms-Fin-04}. This proves the claim.
	\end{proof}

		Now we give the proof of Proposition.
		Since $(v_0,d_0)\in \ve\times D(\hA)$, $f\in L^2(0,T; L^2)$, $g\in L^2(0,T; \rH^1)$, we can apply Theorem \ref{th} to assert that there exists a time $T_0 \leq T$ which depends on $K$ and
		the norms of $(f,g)$ (on [0,T])  and a unique solution $(\tilde{v},\tilde{d}):[0,T_0] \to \ve\times D(\hA)$ to \eqref{1b} such that
		\begin{equation}\notag
		\begin{split}
		\tilde v\in C([0,T_0],\rV) \cap L^2(0,T_0;D(\rA))\\
		\tilde d\in C([0,T_0],D(\hA)) \cap L^2(0,T_0;D(\hA^{3/2})).
		\end{split}
		\end{equation}
		
		
		Hence, by Proposition \ref{Prop:Uniq} we infer that
		$$ (\tilde{v}, \tilde{d})=(v,d) \text{ on } [0,T_\ast].$$
		Thus, since $(v,d)$ is a strong solution on $[0,T_0]$, it follows from the estimates \eqref{Eq:UnifEstStrongSol-1} and \eqref{Eq:UnifEstStrongSol-2} that the regular solution $(v,d)$ on $[0,T_0]$ satisfies \eqref{EQ:AssumL4Norms} on $[0,T_0]$. Hence, the above claim implies that
		\begin{equation}\notag
		(v,d)\in C([0,T_0]; \ve \times D(\hat{\rA} )) \cap L^2(0,T_0; D(\rA) \times D(\hat{\rA}^\frac32).
		\end{equation}
		Thanks to the claim again, we can repeat the above procedure finitely many times, say, on $[T_0, T_1]$, \ldots $[T_N, T_\ast]$ to assert that
		\begin{equation}\notag
		(v,d)\in C([T_i,T_{i+1}]; \ve \times D(\hat{\rA} )) \cap L^2(0,T_0; D(\rA) \times D(\hat{\rA}^\frac32),\;\; i\in \{0, \ldots, N  \}.
		\end{equation}
		With this we conclude the proof of the proposition.
		%
	\end{proof}
	\begin{Thm}\label{Thm:NoSingSmall}
		Let $\eps_0>0$ and $\tilde{\eps}_1\in (0,\eps_0)$ be the constants  from Proposition \ref{Ma2} and Theorem \ref{thm-main}, respectively. If  the data $(v_0,d_0)\in \rH\times \rH^1$ and $(f,g)\in L^2(0,T; L^2\times \rH^1)$ satisfy
		\begin{equation}\label{Eq:AsummSmallData}
		\mathcal{E}(v_0,d_0) +\frac12\int_0^T\left(\lvert f \rvert^2_{\rH^{-1}} + \lvert g \rvert^2_{L^2}\right)ds\le 2 \tilde{\eps}_1^2,
		\end{equation}
		then \eqref{1b} has a unique global strong solution $(v,d):[0,T) \to \rh \times \rH^1$ satisfying
		\begin{equation}\label{Eq:MaxRegSmallData}
		(v,d)\in C([0,T), \rH\times (\rH^1\cap\mathcal{M}))\cap C([0,T]; \rH\times L^2) \cap L^2(0,T; \ve\times D(\hA)).
		\end{equation}	
		That is, the global strong solution does not have any singular times.
		
		Moreover, if $(v_0, d_0)\in V\times D(\hat{\rA})$ and $(f,g)\in L^2(0, T; \rh \times \rH^1)$, then
		\begin{equation}\label{Eq:MaxRegSmallData-2}
		(v,d)\in C([0,T), \ve\times (D(\hat{\rA})\cap\mathcal{M}) )\cap L^2(0,T; D{\rA}\times D(\hA^\frac32)).
		\end{equation}	
	\end{Thm}
	\begin{proof}
		Let $(v_0,d_0)\in \rH\times \rH^1$ and $(f,g)\in L^2(0,T; \rH^{-1} \times \rH^1)$. Then, by Theorem \ref{thm-main} there exist a finite set $S=\{0< T_1< \ldots<T_L< T\}$ and a global weak solution $(v,d)$ to \eqref{1b} satisfying the properties  (1)-(3) of Theorem \ref{thm-main}. We will show that $T_1=T$ if  the smallness condition \eqref{Eq:AsummSmallData} holds.  For this purpose, we argue by contradiction. Assume that $T_1<T$ and that $T_2=T$. Then, by  part (2) of Theorem \ref{thm-main}
		\begin{equation}\notag
		\mathcal{E}(v(T_2), d(T_2)) \le \mathcal{E}(v_0,d_0) + \frac12\int_{0}^{T} \left(\lvert f\rvert^2_{\ve^\ast}+ \lvert g\rvert^2_{L^2}   \right) ds -2\tilde{\eps}_1^2,
		\end{equation}
		which altogether with the assumption \eqref{Eq:AsummSmallData} yields
		\begin{equation}\notag
		\limsup_{t\to T_{2}}\sup_{x\in\Omega}\int_{B_{R}(x)}\left[|v(t,y)|^{2}+|\nabla d(t,y)|^{2}+\phi(d(t,y))\right]dy\le \mathcal{E}(v(T_2), d(T_2)) \le 0.
		\end{equation}
		This clearly contradicts the fact that
		\begin{equation}\notag
		\limsup_{t\to T_{2}}\sup_{x\in\Omega}\int_{B_{R}(x)}\left[|v(t,y)|^{2}+|\nabla d(t,y)|^{2}+\phi(d(t,y))\right]dy\ge\tilde{\varepsilon}^{2}_{1}>0,
		\end{equation}
		for all $R\in (0, \varrho_0]$. This completes the proof of the first part of the theorem.
		
		The second part of the theorem follows easily from the Proposition \ref{prop-reg-zb}. Hence, the proof of the theorem is completed.
	\end{proof}
	
	The last, but not the least, result of this section is about the precompactness of the orbit $(v(t), d(t))$, $t\in [0,T)$, in $\rH\times D(\hA)$. This result will require the following set of conditions on the map $\phi$.
	\begin{assume}\label{Assum:PhiSpecial}
		Let $\xi\in \mathbb{R}^3$ be a constant vector and $\phi: \mathbb{R}^3 \to [0,\infty)$ be the map defined by
		\begin{equation}\notag
		\phi(d)= \frac12 \lvert d -\xi \rvert^2, d\in \mathbb{R}^3.
		\end{equation}
	\end{assume}
	It is clear that if $\phi$ satisfies this assumption, the it also satisfies Assumption \ref{Assum:EnergyPotential}.
	\begin{Thm}\label{Thm:PrecompactInHH1}
		
		Let $\eps_0>0$ and $\eps_1\in (0,\eps_0)$ be the constants  from Proposition \ref{Ma2} and Theorem \ref{thm-main}, respectively.
		Then, there exists a constant $\kappa_1\in (0,\eps_1]$ such that the following hold.
		
		If $(v_0,d_0)\in \rH\times \rH^1$ and $(f,g)\in L^2(0,T; \ve^\ast \times \rH^1)$ satisfy
		\begin{equation}\label{Eq:SmallAssumSmooth}
		\mathcal{E}(v_0,d_0) +\frac12\int_0^T\left(\lvert f \rvert^2_{\rH^{-1}}+ \lvert g \rvert^2_{L^2}\right)ds\le 2 \kappa_1^2,
		\end{equation}
		then \eqref{1b} has a unique global strong solution $(v,d):[0,T) \to \rh \times \rH^1$ satisfying:
		\begin{enumerate}
			\item $(v,d)$ does not have any singular point,
			\begin{equation}\label{Eq:MaxRegSmallData-2-B}
			(v,d)\in C([0,T), \ve\times D(\hat{\rA}) )\cap L^2(0,T; D({\rA})\times D(\hA^\frac32)).
			\end{equation}	
			\item There exists a constant $K_1>0$ such that for all $T>0$
			\begin{equation}\label{Eq:EstVeloOptDirInVH2-0}
			\sup_{t\in [0,T)} (\lvert v(t) \rvert^2_{\ve} + \lvert d(t)\rvert^2_{D(\hA)} ) + \int_0^T \left( \lvert \rA v (s) \rvert^2_{L^2}  +\lvert \nabla \Delta d(s) \rvert^2_{L^2} \right) ds \le K_1.
			\end{equation}
			In particular, the orbit of $(v,d)$ is precompact in $\rH\times D(\hA^\frac12)$.
		\end{enumerate}
		%
	\end{Thm}
	\begin{proof}
		Let $(v_0,d_0)\in \ve\times D(\hA)$, $f\in L^2(0,T; L^2)$, $g\in L^2(0,T; \rH^1)$.
		Let $\eps_0>0$ and ${\eps}_1\in (0,\eps_0)$ be the constants  from Proposition \ref{Ma2} and Theorem \ref{thm-main}, respectively. Let $\kappa_0\in (0,\eps_1]$ be a number to be chosen in \eqref{Eq:ChoiceOfKappa} (see the proof of  Proposition  \ref{Thm:PrecompactinVastL2}) such that \eqref{Eq:SmallAssumSmooth} holds. Then, by
		%
		Theorem \ref{Thm:NoSingSmall} we deduce that the problem \eqref{1b}  has a solution $(v,d)$ which does not have singular times and satisfying
		\eqref{Eq:MaxRegSmallData-2-B}. The proof of \eqref{Eq:EstVeloOptDirInVH2-0} will be proved in Proposition \ref{Prop:EstVeloOptDirInVH2} below.
	\end{proof}
	\begin{Prop} \label{Thm:PrecompactinVastL2}
		Let $\eps_0>0$ and ${\eps}_1\in (0,\eps_0)$ be the constants  from Proposition \ref{Ma2} and Theorem \ref{thm-main}, respectively. Then, there exists $\kappa_1\in (0, \eps_0)$ and $K_4, K_0>0$  such that the following hold.
		
		If  $(v_0,d_0)\in \ve\times D(\hA)$, $f\in L^2(0,T; L^2)$, $g\in L^2(0,T; \rH^1)$ satisfy
		\begin{equation}\notag
		\mathbf{E}_0= \mathcal{E}(v_0,d_0)+\frac12 \int_0^\infty \left[\lvert f \rvert^2_{\rH^{-1}}+ \lvert g \rvert^2_{L^2}\right] dt<2 \kappa_1^2,
		\end{equation}
		then any regular solution  $(v,d)$  to \eqref{1b} defined on $[0,T)$ satisfies
		\begin{equation}\label{Eq:EstVeloOptDirInHH1}
		\begin{split}
		\frac12\sup_{t\in [0,T)} \left( \lvert v(t) \rvert^2_{L^2} + \lvert \nabla d(t)  \rvert^2 \right) + \int_0^{T}\left( \lvert \nabla v (r)\rvert^2_{L^2} + K_0 \lvert \nabla d(r)\rvert^2 + K_0 \lvert \Delta d (r)\rvert^2_{L^2}  \right)dr
		\le \mathbf{E}_0.
		\end{split}
		\end{equation}
		Furthermore,
		\begin{equation}\label{Eq:EstOptDirInL4}
		\int_0^{T} [\lvert \nabla d\rvert^4_{L^4} + \lvert \nabla d\rvert^2_{L^4} ]dt \le \frac{K_4}{K_0}(1+\mathbf{E}_0^2)
		\end{equation}
	\end{Prop}
	\begin{proof}
		Let $(v_0,d_0)\in \ve\times D(\hA)$, $f\in L^2(0,T; L^2)$, $g\in L^2(0,T; \rH^1)$.  Let $\eps_0>0$ and ${\eps}_1\in (0,\eps_0)$ be the constants  from Proposition \ref{Ma2} and Theorem \ref{thm-main}, respectively.  	Let $\kappa_0\in (0,{\eps}_1]$ be a number to be chosen later such that \eqref{Eq:SmallAssumSmooth} holds.
		Let $(v,d)$ be a regular solution  to \eqref{1b} defined on $[0,T)$.
		%
		Then, multiplying the velocity equation by $v$ and using the Cauchy-Schwarz and  Young inequalities  imply
		\begin{equation}\label{Eq:EstVeloinL2-1}
		\begin{split}
		\frac12 \frac{d}{dt} \lvert v\rvert^2_{L^2} + \frac12 \lvert \nabla v\rvert^2_{L^2} \le -\langle \Pi\Div(\nabla d \odot \nabla d), u \rangle + \frac12 \lvert f\rvert^2_{\rh^{-1}}.
		\end{split}
		\end{equation}
		We multiply the optical director equation by $-\Delta d$, then use Cauchy-Schwarz and Young inequalities and the constraint $d(t)\in \mathcal{M}$, $t\in [0,T]$,  and obtain
		\begin{equation*}
		\begin{split}
		\frac12 \frac{d}{dt} \lvert \nabla d \rvert^2_{L^2} + \frac12 \lvert \Delta d \rvert^2_{L^2} \le \langle v \cdot \nabla d, \Delta d \rangle - \langle \lvert \nabla d \rvert^2 d -\phi^\prime(d) +(\phi^\prime(d) \cdot d)d, \Delta d\rangle +\frac12 \lvert g \rvert^2_{L^2}.
		\end{split}
		\end{equation*}
		Using the fact $\phi^\prime(d)= d-\xi$ and the integration-by-parts on the torus we find
		\begin{equation}\label{Eq:EstOptDirInH1-1}
		\begin{split}
		\frac12 \frac{d}{dt} \lvert \nabla d \rvert^2_{L^2} + \frac12 \lvert \Delta d \rvert^2_{L^2} + \lvert\nabla d \rvert^2_{L^2} \le \langle v \cdot \nabla d, \Delta d \rangle +\frac12 \lvert g \rvert^2_{L^2}- \langle \lvert \nabla d \rvert^2 d  +(\phi^\prime(d) \cdot d)d, \Delta d\rangle.
		\end{split}
		\end{equation}
		Before proceeding further, let us estimate the  last term of the above inequality. Toward this end we divide the task into two parts. Firstly, we use fact that $\Delta d\cdot d = -\lvert \nabla d \rvert^2$, the H\"older, the Gagliardo-Nirenberg (\cite[Section 9.8, Example C.3]{Brezis}) , the Young inequalities and \cite[Theorem 3.4]{Simader} to get the following chain of inequalities
		\begin{align}
		\lvert \langle (\phi^\prime(d)\cdot d)d, \Delta d \rangle \rvert & = -\int_{\Omega} (\phi^\prime(d(x)) \cdot d(x) ) \lvert \nabla d(x)\rvert^2 \;dx  \notag\\
		& \le \lvert \phi^\prime(d) \rvert_{L^2} \lvert d \rvert_{L^\infty} \lvert \nabla d \rvert^2_{L^4}\notag\\
		&\le C_0\lvert d- \xi \rvert_{L^2} \left(\lvert \nabla d \rvert_{L^2} \lvert \nabla^2 d \rvert_{L^2}+ \lvert \nabla d \rvert^2_{L^2}   \right)\notag\\
		& \le C_0 \lvert d-\xi \rvert_{L^2} \left(\lvert \nabla d \rvert_{L^2} \lvert \Delta d \rvert_{L^2}+ \lvert \nabla d \rvert^2_{L^2}   \right)\notag\\
		&\le \frac12 \lvert \nabla d \rvert^2_{L^2} + 2C^2_0 \frac12\lvert d-\xi \rvert^2_{L^2} \left(\lvert \Delta d \rvert^2_{L^2} +   \lvert \nabla d \rvert^2_{L^2}\right).
		\label{Eq:EstAnisopNRJ}
		\end{align}
		Secondly, using the  Gagliardo-Nirenberg inequality (\cite[Section 9.8, Example C.3]{Brezis})  and \cite[Theorem 3.4]{Simader}    we obtain
		\begin{align}
		-\langle \lvert \nabla d\rvert^2 d, \Delta d \rangle&= \lvert \nabla d \rvert^4_{L^4}\notag\\
		&\le C_1 \lvert \nabla d \rvert^2_{L^2} (\lvert \Delta d \rvert^2_{L^2}+ \lvert \nabla d \rvert^2_{L^2} ).\label{Eq:EstTransportTerm-1}
		\end{align}
		Plugging \eqref{Eq:EstAnisopNRJ} and \eqref{Eq:EstTransportTerm-1} into the inequality \eqref{Eq:EstOptDirInH1-1} yields
		\begin{equation}\label{Eq:EstOptDirInH1-2}
		\begin{split}
		\frac12 \frac{d}{dt} \lvert \nabla d \rvert^2_{L^2}  + \lvert\nabla d \rvert^2_{L^2} \le \langle v \cdot \nabla d, \Delta d \rangle +\frac12 \lvert g \rvert^2_{L^2}-\left(\frac12-2C_0^2 \frac12  \lvert d-\xi \rvert^2_{L^2}-C_1 \lvert \nabla d\rvert^2_{L^2} \right)\lvert \Delta d\rvert^2_{L^2}\\
		-\left(\frac12 -2C_0^2 \frac12 \lvert d-\xi \rvert^2_{L^2}-C_1 \lvert \nabla d\rvert^2_{L^2} \right)\lvert \nabla d \rvert^2_{L^2}.
		\end{split}
		\end{equation}
		Now adding the inequalities \eqref{Eq:EstVeloinL2-1} and \eqref{Eq:EstOptDirInH1-2} side by side, and using \eqref{Eq:Dissip} imply
		\begin{equation}\notag
		\begin{split}
		\frac12 \frac{d}{dt} \left(\lvert \nabla d \rvert^2_{L^2} + \lvert v \rvert^2_{L^2}\right) \frac12 \lvert \nabla v \rvert^2_{L^2} \le-\left(\frac12-2C_0^2 \frac12  \lvert d-\xi \rvert^2_{L^2}-C_1 \lvert \nabla d\rvert^2_{L^2} \right)\lvert \Delta d\rvert^2_{L^2}\\
		-\left(\frac12 -2C_0^2 \frac12  \lvert d-\xi \rvert^2_{L^2}-C_1 \lvert \nabla d\rvert^2_{L^2} \right)\lvert \nabla d \rvert^2_{L^2}\\
		+ \frac12 \left(\lvert g \rvert^2_{L^2}+\lvert f \rvert_{\rH^{-1}}\right).
		\end{split}
		\end{equation}
		Thanks to the energy inequality \eqref{Eq:NRJ-Inequality} we obtain
		\begin{equation}\label{Eq:EstVeloOptDirInL2H1}
		\begin{split}
		\frac12 \frac{d}{dt} \left(\lvert \nabla d \rvert^2_{L^2} + \lvert v \rvert^2_{L^2}\right)  + \lvert\nabla d \rvert^2_{L^2}+ \frac12 \lvert \nabla v \rvert^2_{L^2} \le-\left(\frac12-2C_0^2 \mathbf{E}_0-C_1 \mathbf{E}_0 \right)\lvert \Delta d\rvert^2_{L^2}\\
		-\left(\frac12 -2C_0^2 \mathbf{E}_0-C_1\mathbf{E}_0 \right)\lvert \nabla d \rvert^2_{L^2}
		+ \frac12 \left(\lvert g \rvert^2_{L^2}+\lvert f \rvert_{\rH^{-1}}\right).
		\end{split}
		\end{equation}
		We now easily conclude the proof of \eqref{Eq:EstVeloOptDirInHH1} in the proposition by integrating   \eqref{Eq:EstVeloOptDirInL2H1}, taking the supremum over $t\in [0,T_0]$ and choosing
		\begin{equation}\label{Eq:ChoiceOfKappa}
		\kappa_1^2= \min\{\eps_1^2, \frac{1}{4(C_0^2 \lvert \Omega \rvert+ C_1)} \} \text{ and } K_0= \frac12- \mathbf{E}_0\left(C_0^2 \lvert \Omega \rvert +C_1\right).
		\end{equation}
		
		In order to prove the second estimate, we use the Gagliardo-Nirenberg inequality (see \cite[Section 9.8, Example C.3]{Brezis}) and \eqref{Eq:EstVeloOptDirInHH1} to obtain
		\begin{align*}
		\int_0^T( \lvert \nabla d \rvert^4_{L^4} +\lvert \lvert \nabla d\rvert^2_{L^4} )dt &\le C \sup_{t\in [0,T]} \lvert \nabla d(t)\rvert^2_{L^2} \int_0^T \lvert \Delta d(t) \rvert^2 dt +C \int_0^{T} \lvert \nabla d (s) \rvert^2_{L^2} ds \\
		& + C \int_0^T (\lvert \nabla d(s) \rvert^2+ \lvert \Delta d(s) \rvert^2_{L^2} ds ) \\
		&\le \frac{C}{K_0}[\mathbf{E}_0^2+\mathbf{E}_0 ].
		\end{align*}
		This completes the proof of the proposition.
	\end{proof}
	The next result is crucial for the proof of Theorem \ref{Thm:PrecompactInHH1}.
	\begin{Prop}\label{Prop:EstVeloOptDirInVH2}
		Let $\eps_0$ be as in Proposition \ref{Ma2} and $(v_0,d_0)\in \ve\times D(\hA)$, $f\in L^2(0,T; L^2)$, $g\in L^2(0,T; \rH^1)$. Assume that
		\begin{equation}\notag
		\mathbf{E}_0= \mathcal{E}(v_0,d_0)+\frac12 \int_0^\infty \left[\lvert f \rvert^2_{\rH^{-1}}+ \lvert g \rvert^2_{L^2}\right] dt< \kappa^2_1,
		\end{equation}
		where  $\kappa_1\in (0,\eps_0)$ is defined in Proposition \ref{Thm:PrecompactinVastL2}. Then, there exist a constant  $K_2>0$, which depends only on the norms of $(v_0,d_0)\in \ve\times D(\hA)$ and $(f,g)\in L^2(0,T; \rh\times \rH^1)$,  such that for a regular solution  $(v,d)$  to \eqref{1b} we have
		\begin{equation}\label{Eq:EstVeloOptDirInVH2}
		\begin{split}
		\frac12\sup_{t\in [0,T)} \left( \lvert \rA^\frac12 v(t) \rvert^2_{L^2} + \lvert \Delta d(t)  \rvert^2 \right) + \int_0^{T_\ast}\left( \lvert \rA v (r)\rvert^2_{L^2} + \lvert \Delta d(r)\rvert^2 +  \lvert \Delta d (r)\rvert^2_{L^2}  \right)dr
		\le K_2.
		\end{split}
		\end{equation}
	\end{Prop}
	
	\begin{proof}
		The proof of the estimate \eqref{Eq:EstVeloOptDirInVH2} is very similar to the proof of    \eqref{Eq:EstinHigherorderNorms}, hence we only sketch the proof.
		We start with the same idea as in the proof of \eqref{Eq:EstinHigherorderNorms}, i.e., we multiply the velocity and optical director equations by $Av$ and $\rA^2 d$, respectively. This procedure implies the equation \eqref{Eq:Derivativehigheroerdernorms}. In the present proof, we estimate the term ${}_{\rH^1}\langle\phi^\prime(d), -\hA^2 d\rangle_{(\rH^1)^\ast} $ as follows:
		\begin{align}\notag
		 \;{}_{\rH^1}\langle\phi^\prime(d), -\hA^2 d\rangle_{(\rH^1)^\ast}  =-\langle \Delta\phi^\prime(d), \Delta d \rangle = -\lvert \Delta d \rvert^2_{L^2}.
		\end{align}
		This gives raise to the term $\int_0^T \lvert \Delta d(r)\rvert^2  dr$ in the left hand side of \eqref{Eq:EstVeloOptDirInVH2}.
		Now, the remaining terms are estimated with exactly the same way as in the proof of  \eqref{Eq:Derivativehigheroerdernorms}. In particular, we infer that
		there exist constants  $K_3>0$, which depends on the norms of $(u_0,d_0)\in \ve\times D(\hA)$ and $(f,g)\in L^2(0,T; \rH\times D(\hA^{1/2}))$, and $K_4>0$ such that  for all $\tau\in [0,T)$ we have
		\begin{equation}\label{Eq:EstinHigherorderNorms-A}
		\begin{split}
		\sup_{0\le s\le \tau} \left(\lvert \nabla
		v(s) \rvert^2_{L^2} + \lvert \Delta d(s) \rvert^2_{L^2} \right) + 2 \int_0^{\tau} \left( \lvert \nabla \Delta d(s) \rvert^2_{L^2} +\lvert \Delta d(s)\rvert^2_{L^2} +\lvert A v(s) \rvert^2_{L^2}\right)ds \\
		\le K_3e^{K_4 \int_0^{\tau}[\lvert \nabla d(r) \rvert^4_{L^4}+ \lvert v(r) \rvert^4_{L^4}+\lvert \nabla d(r) \rvert^2_{L^4}+ \lvert v(r) \rvert^2_{L^4}  ]dr   }.
		\end{split}
		\end{equation}
	Since $L^4\cap \rH \subset \ve $, by using \eqref{Eq:EstVeloOptDirInHH1} and the Ladyzhenskaya inequality (\cite[Lemma III.3.3]{Temam_2001}) we infer that there exists $C>0$ such that
		\begin{equation}\notag
		\int_0^T[ \lvert v(s) \rvert^4_{L^4} +\lvert v(s) \rvert^2_{L^4}] ds\le  \int_0^T[ \lvert v(s) \rvert^4_{L^4} +\lvert \nabla v(s) \rvert^2] ds \le C (1+\mathbf{E}_0^2).
		\end{equation}
		Thus, plugging the latter estimate and \eqref{Eq:EstOptDirInL4} into \eqref{Eq:EstinHigherorderNorms-A} imply  \eqref{Eq:EstVeloOptDirInVH2}. This completes the proof of the proposition.
	\end{proof}
	\appendix
\section{Proof of Lemma \ref{Lem:RighthandSideinL2VxH1}}
In this section we will prove Lemma \ref{Lem:RighthandSideinL2VxH1}.
\begin{proof}[Proof of Lemma \ref{Lem:RighthandSideinL2VxH1}]
	Firstly, we infer from
	\cite[Lemma 2.5]{ZB+EH+PR-RIMS}	
	\footnote{ Note that in
		\cite[Lemma 2.5]{ZB+EH+PR-RIMS}
		the authors used $B_1(v_i,v_i)$  and $M(n_i,n_i)$ in places of $B(v_i,v_i)$ and $\Pi\left(\nabla n_i \Delta  n_i\right)$, respectively. }  that there exists a  constant $C>0$ such that for all $(v_i,d_i)\in D(\rA)\times D(\hA^\frac32)$
	\begin{equation*}
	\begin{split}
	\lvert F(v_1,n_1)-F(v_2, n_2) \rvert_{\rH} \le C \lvert\rA^\frac12[ v_1 -v_2]\rvert^\frac12_{L^2} \left(\lvert \rA[v_1 -v_2] \rvert^\frac12_{L^2} \lvert \rA^\frac12 u_2 \rvert_{L^2} + \lvert \rA^\frac12[v_1-v_2] \rvert^\frac12_{L^2} \lvert \rA^\frac12 u_1 \rvert^\frac12_{L^2}  \lvert \rA u_1 \rvert^\frac12_{L^2}  \right)\\
	+ \lvert n_1 -n_2 \rvert_{\rH^2} \lvert n_1 \rvert^\frac12_{\rH^2} \lvert n_1\rvert^\frac12_{\rH^3} + \lvert n_1-n_2 \rvert^\frac12_{\rH^2} \lvert n_1-n_2 \rvert^\frac12_{\rH^3} \lvert n_2 \rvert^\frac12_{\rH^2}
	\end{split}
	\end{equation*}
	This clearly implies that there exists a constant $C>0$ such that for all $(v_i,d_i)\in \mathbf{X}_T$, $i=1,2$,
	\begin{equation*}
	\begin{split}
	\lvert F(v_1,n_1)-F(v_2, n_2) \rvert^2_{L^2(0,T; \rH)} \le & C T^\frac12 \lvert v_1 -v_2\rvert_{C([0,T];\ve)}  \lvert  u_2 \rvert_{C([0,T];\ve)} \lvert \rA[v_1 -v_2] \rvert_{L^2(0,T;D(\rA))}  \\
	&+ C T^\frac12 \lvert \rA^\frac12[v_1-v_2] \rvert^2_{C([0,T]; \ve)} \lvert  u_1 \rvert_{C([0,T];\ve)}  \lvert \rA u_1 \rvert_{L^2(0,T;D(\rA))}  \\
	&+  C T^\frac12  \lvert n_1 -n_2 \rvert^2_{C([0,T];\rH^2)} \lvert n_1 \rvert_{C([0,T];\rH^2)} \lvert n_1\rvert_{L^2(0,T;\rH^3)} \\
	&+ CT^\frac12\lvert n_1-n_2 \rvert_{C([0,T]; \rH^2)} \lvert n_1-n_2 \rvert_{L^2(0,T;\rH^3)} \lvert n_2 \rvert_{C([0,T];\rH^2)}.
	\end{split}
	\end{equation*}
	Hence, we infer that there exists a constant $C>0$ such that for all $(v_i,d_i)\in \mathbf{X}_T$, $i=1,2$,
	\begin{equation*}
	\begin{split}
	\lvert F(v_1,n_1)-F(v_2, n_2) \rvert^2_{L^2(0,T; \rH)} \le C T^\frac12 \lvert v_1-v_2\rvert^2_{\mathbf{X}^1_T}[\lvert u_1\rvert^2_{\mathbf{X}^1_T}+\lvert u_2\rvert^2_{\mathbf{X}^1_T}  ]\\
	C T^\frac12 \lvert n_1-n_2\rvert^2_{\mathbf{X}^1_T}[\lvert d_1\rvert^2_{\mathbf{X}^2_T}+\lvert d_2\rvert^2_{\mathbf{X}^2_T}  ],
	\end{split}
	\end{equation*}
	from which we easily deduce the inequality   \eqref{Eq:FPT-ForcingVelo}.
	
	We will now proceed to the proof of \eqref{Eq:FPT-ForcingOptDir-1} which will be divided into four steps.  Firstly, because of the assumption  \eqref{Eq:BoundednessPhibis} the map $\phi^\prime:\mathbb{R}^3 \to \mathbb{R}^3$ is Lipschitz continuous. Hence, for all $n_i\in \rH^2$, $i=1,2$,
	\begin{align*}
	\lvert \phi^\prime(n_1)- \phi^\prime(n_2) \rvert_{\rH^1}=& \lvert \phi^\prime(n_1)- \phi^\prime(n_2) \rvert_{L^2}+\lvert \nabla \phi^\prime(n_1)- \nabla \phi^\prime(n_2) \rvert_{L^2}\\
	\le & M_2 \lvert n_1 -n_2\lvert_{L^2} + \lvert \phi^{\prime\prime}(n_1)\nabla n_1-\phi^{\prime\prime}(n_2) \nabla n_2\rvert_{L^2}\\
	\le & M_2 \lvert n_1 -n_2\lvert_{L^2} + \lvert [\phi^{\prime\prime}(n_1)-\phi^{\prime\prime}(n_2)]\nabla n_1+\phi^{\prime\prime}(n_2) \nabla (n_1-n_2)\lvert_{L^2}.
	\end{align*}
	Using \eqref{Eq:LipSchitzPhibis}, \eqref{Eq:BoundednessPhibis} and the Sobolev emebedding $\rH^1\subset L^{4}$ 	we obtain
	
	%
	
	\begin{align*}
	\lvert \phi^\prime(n_1)- \phi^\prime(n_2) \rvert_{\rH^1} \le & M_2\lvert n_1 -n_2\lvert_{L^2}+ \lvert \lvert n_1 -n_2\lvert \lvert \nabla n_1\rvert\rvert_{L^2}+ M_2 \lvert \nabla (n_1-n_2) \rvert_{L^2}\\
	\le & M_2 \lvert n_1 -n_2 \rvert_{\rH^1}+ \lvert n_1 -n_2 \rvert_{L^4} \lvert \nabla n_1 \rvert_{L^4}\\
	\le & M_2 \lvert n_1- n_2 \rvert_{\rH^1} + \lvert n_1 - n_2 \rvert_{\rH^1} \lvert n_1 \rvert_{\rH^2}.
	\end{align*}
	Hence, there exists a constant $C>0$ such that for all $n_i\in \mathbf{X}^2_T$, $i=1,2$
	\begin{equation} \label{Eq:EstimateL2H1PhiPrime}
	\lvert \phi^\prime(n_1)- \phi^\prime(n_2) \rvert_{L^2(0,T;\rH^1)}^2 \le M_2 T  \lvert n_1 -n_2\rvert^2_{\mathbf{X}^2_T} [1 + \lvert n_1 \rvert^2_{\mathbf{X}^2_T} ].
	\end{equation}
	Secondly, we estimate the term involving $(\lvert (\phi^\prime(n_i) \cdot n_i)n_i)$, $i=1,2$, as follows. Let  $n_i\in \rH^2$, $i=1,2$. Then,  we have
	\begin{align*}
	\lvert (\phi^\prime(n_1) \cdot n_1)n_1 -(\phi^\prime(n_2) \cdot n_2)n_2\rvert_{\rH^1} =& \lvert ([\phi^\prime(n_1) -\phi^\prime(n_2)] \cdot n_1)n_1+ (\phi^\prime(n_2)\cdot n_1)n_1-(\phi^\prime(n_2) \cdot n_2)n_2 \rvert_{\rH^1}\\
	\le & \lvert ([\phi^\prime(n_1) -\phi^\prime(n_2)] \cdot n_1)n_1\rvert_{\rH^1} +\lvert  (\phi^\prime(n_2)\cdot [n_1-n_2])n_1\rvert_{\rH^1} \\
	&\qquad + \lvert (\phi^\prime(n_2)\cdot n_2)[n_1-n_2] \rvert_{\rH^1}.
	\end{align*}
	{We recall the following classical fact whose proof is easy and omitted. }
	\begin{equation}\label{Eq:ProductH1Linfty}
	\lvert fg\rvert_{\rH^1}\le 2 \left( \lvert f\rvert_{L^\infty} \lvert g \rvert_{\rH^1} + \lvert f\rvert_{\rH^1}\lvert g \rvert_{L^\infty}\right),\;\; \text{ for all } f,g\in \rh^1\cap L^\infty.
	\end{equation}
	Since, $n_i\in \rh^2$ and $\rH^2\subset L^\infty$, it follows from  \eqref{Eq:LInearGrowthPhiprime} and \eqref{Eq:BoundednessPhibis} that $\phi^\prime(n_i)\in \rh^1\cap L^\infty$. Hence, we can  repeatedly use \eqref{Eq:ProductH1Linfty}, the Lipschitz property of $\phi^\prime$  to infer that for all $n_i\in \rH^2$, $i=1,2$,
	\begin{align*}
	\lvert \left(\left[\phi^\prime(n_1)-\phi^\prime(n_2)\right]\cdot n_1\right)n_1 \rvert_{\rH^1}\le C \lvert \phi^\prime(n_1) -\phi^\prime(n_2)\rvert_{L^\infty}  \lvert n_1\rvert_{L^\infty} \lvert n_1\rvert_{\rH^1} + \lvert \phi^\prime(n_1)-\phi^\prime(n_2) \rvert_{\rH^1} \lvert n_1\rvert_{L^\infty}^2.\\
	\le \lvert n_1 -n_2\rvert_{L^\infty} \lvert n_1 \rvert^2_{\rH^2} + \lvert \phi^\prime(n_1) -\phi^\prime(n_2)\rvert_{\rH^1} \lvert n_1\rvert^2_{\rH^2}\\
	\le \lvert n_1-n_2\rvert_{\rH^2} \lvert n_1\rvert^2_{\rH^2}+  \lvert \phi^\prime(n_1) -\phi^\prime(n_2)\rvert_{\rH^1} \lvert n_1\rvert^2_{\rH^2}.
	\end{align*}
	Thanks to the last line and \eqref{Eq:EstimateL2H1PhiPrime} we deduce that there exists $C>0$ such that for all $n_1,n_2\in \mathbf{X}^2_T$
	\begin{align*}
	\lvert \left(\left[\phi^\prime(n_1)-\phi^\prime(n_2)\right]\cdot n_1\right)n_1 \rvert^2_{L^2(0,T;\rH^1)}\le CT  \lvert n_1-n_2\rvert^2_{\mathbf{X}^2_T } \lvert n_1\rvert^4_{\mathbf{X}^2_T}\left[ 1 + \lvert n_1 \rvert^2_{\mathbf{X}^2_T} \right],
	\end{align*}

	\noindent Using \eqref{Eq:ProductH1Linfty} we infer that there exists a constant $C>0$ such that for all $a,b,c\in \rH^2$
	\begin{align*}
	\lvert (\phi^\prime(a)\cdot b)c\rvert_{\rH^1} \le & C \lvert\phi^\prime(a)  \rvert_{L^\infty}[ \lvert b \rvert_{L^\infty} \lvert c \rvert_{\rH^1}+\lvert b \rvert_{\rH^1} \lvert c \rvert_{L^\infty}   ]+ \lvert \phi^\prime(a)\rvert_{\rH^1} \lvert b \rvert_{L^\infty} \lvert a \rvert_{L^\infty}.\\
	\end{align*}
	This altogether with \eqref{Eq:LInearGrowthPhiprime}, \eqref{Eq:BoundednessPhibis} and  the continuous Sobolev embeddings $\rH^2\hookrightarrow L^\infty$and $\rH^2\subset \rH^1$ implies that 	for all $a,b,c\in \rH^2$
	\begin{align*}
	\lvert (\phi^\prime(a)\cdot b)c\rvert_{\rH^1} \le & C (1+\lvert a  \rvert_{\rH^2}) \lvert b \rvert_{\rH^2} \lvert c \rvert_{\rH^2},
	\end{align*}
	Thus, there exists a constant $C>0$ such that $n_1,n_2\in \rH^2$
	\begin{align*}
	\lvert  (\phi^\prime(n_2)\cdot [n_1-n_2])n_1\rvert_{\rH^1} \le C (1+\lvert n_2  \rvert_{\rH^2}) \lvert n_1-n_2 \rvert_{\rH^2} \lvert n_1 \rvert_{\rH^2}\\
	\lvert (\phi^\prime(n_2)\cdot n_2)[n_1-n_2] \rvert_{\rH^1}\le C (1+\lvert n_2 \rvert_{\rH^2}) \lvert n_2 \rvert_{\rH^2} \lvert n_1-n_2 \rvert_{\rH^2}.
	\end{align*}
	Thus, we deduce that there exists  constant $C>0$ such that for all
	$n_1,n_2\in \mathbf{X}^2_T$
	\begin{align}
	\lvert  (\phi^\prime(n_2)\cdot [n_1-n_2])n_1\rvert^2_{L^2(0,T;\rH^1)} \le C T (1+\lvert n_2  \rvert_{\mathbf{X}^2_T}^2) \lvert n_1-n_2 \rvert^2_{\mathbf{X}^2_T} \lvert n_1 \rvert^2_{\mathbf{X}^2_T}\label{Eq:alphaphiprime-1}\\
	\lvert (\phi^\prime(n_2)\cdot n_2)[n_1-n_2] \rvert^2_{L^2(0,T;\rH^1)}\le C T (1+\lvert n_2 \rvert^2_{\mathbf{X}^2_T}) \lvert n_2 \rvert_{\mathbf{X}^2_T}^2 \lvert n_1-n_2 \rvert^2_{\mathbf{X}^2_T}.\label{Eq:alphaphiprime-2}
	\end{align}
	Therefore, there exists a constant $C>0$ such that for all $n_i\in \mathbf{X}^2_T$ , $i=1,2$,
	\begin{align} \label{Eq:alphaphiprime-3}
	\lvert \left(\left[\phi^\prime(n_1)-\phi^\prime(n_2)\right]\cdot n_1\right)n_1 \rvert_{L^2(0,T; \rH^1)} \le C T \lvert n_1 -n_2\rvert^2_{\mathbf{X}^2_T} \left(1 + \lvert n_1 \rvert^4_{\mathbf{X}^2_T} + \lvert n_1\rvert^6_{\mathbf{X}^2_T} + \lvert n_2 \rvert^2_{\mathbf{X}^2_T} \right).
	\end{align}
	From \eqref{Eq:alphaphiprime-1}-\eqref{Eq:alphaphiprime-3} we infer that there exists a constant $C>0$ such that for all $n_i\in \mathbf{X}^2_T$, $i=1,2$,
	\begin{equation*}
	\lvert (\phi^\prime(n_1)\cdot n_1)n_1-(\phi^\prime(n_2)\cdot n_2)n_2 \rvert_{L^2(0,T; \rH^1)} \le C T \lvert n_1 -n_2\rvert^2_{\mathbf{X}^2_T} \left(1 + \sum_{i=1}^2  \lvert n_i \rvert^6_{\mathbf{X}^2_T}\right).
	\end{equation*}

	\noindent Thirdly, it was proved in \cite[Lemma 2.6]{ZB+EH+PR-RIMS}, where the notation $B_2(u_i,n_i)$ was used in place of $u_i\cdot \nabla n_i$,  that there exists  a constant $C>0$ such that for all $u_i\in \ve $ and $n_i\in D(\hA^\frac32)$, $i=1,2$, we have
	\begin{align*}
	\lvert u_1\cdot \nabla n_1 - u_2\cdot\nabla n_2\rvert_{\rH^1}\le  C
	\biggl(&\lvert \nabla(u_1-u_2)\rvert_{L^2} \lvert
	n_1\rvert_{\rH^2}^{\frac12}\lvert n_1\rvert_{\rH^3}^\frac12\\
	&+ \lvert n_1-n_2\lvert^{\frac12}_{\rH^2} \lvert n_1-n_2\lvert^{\frac12}_{\rH^3}\lvert
	\nabla u_2\lvert_{L^2}\biggr).
	\end{align*}
	From this inequality we easily infer that there exists a constant $C>0$ such that for all {$(u_i,n_i)\in \mathbf{X}^1_T\times \mathbf{X}^2_T $}, $i=1,2$
	\begin{align*}
	\lvert u_1\cdot \nabla n_1 - u_2\cdot\nabla n_2\rvert^2_{L^2(0,T; \rH^1)}\le  C
	\biggl(& \lvert u_1-u_2\rvert_{C([0,T]; \ve)}^2 \lvert
	n_1\rvert_{C([0,T];D(\hA))}\int_0^T \lvert n_1(s)\rvert_{\rH^3} ds \\
	&+ \lvert
	\nabla u_2\lvert^2_{C([0,T];\ve)} \lvert n_1-n_2\lvert_{C([0,T];D(\hA))}\int_0^T \lvert n_1(s)-n_2(s)\lvert_{\rH^3}ds \biggr)\\
	\le & C T^\frac12
	\biggl(\lvert u_1-u_2\rvert_{\mathbf{X}^1_T}^2 \lvert
	n_1\rvert_{\mathbf{X}^2_T}^2+ \lvert u_2\lvert^2_{\mathbf{X}^1_T} \lvert n_1-n_2\lvert_{\mathbf{X}^2_T}^2 \biggr)
	\end{align*}
	Fourthly, we prove that there exists a constant $C>0$ such that for all $n_i\in \mathbf{X}^2_T$, $i=1,2$,
	\begin{equation}\label{Eq:GradientNonlinearity-Lipschitz}
	\lvert \lvert \nabla n_1\rvert^2n_1-\lvert \nabla n_2 \rvert^2n_2 \rvert^2_{L^2(0,T; \rH^1)} \le C (T\vee T^\frac12)    \lvert n_1-n_2\rvert^2_{\mathbf{X}^2_T} \left(\lvert n_1\rvert^4_{\mathbf{X}^2_T}+\lvert n_2\rvert^4_{\mathbf{X}^2_T}  \right).
	\end{equation}
	We need the following claim to prove this.
	\begin{Clm}\label{Claim:Projection-Of-Laplacian}
		There exists $C>0$ such that for $a,b\in D(\hA^\frac32)$ and $c\in D(\hA)$ the following holds
		\begin{equation*}
		\lvert [\nabla a : \nabla b ] c \rvert_{\rH^1}^2 \le \lvert a \rvert_{\rH^2}^2 \lvert b\rvert_{\rH^2}^2\lvert c \rvert_{\rH^2}^2+ \lvert c \rvert_{\rH^2}^2[ \lvert a \rvert_{\rH^2}\lvert a \rvert_{\rH^3}\lvert b \rvert^2_{\rH^2}+\lvert a \rvert^2_{\rH^2}\lvert b \rvert_{\rH^2}\lvert b \rvert_{\rH^3}  ]
		\end{equation*}
	\end{Clm}
	\begin{proof}[Proof of the Claim \ref{Claim:Projection-Of-Laplacian}]
		We infer from  the Cauchy-Schwarz inequality and the Gagliardo-Nirenberg inequality (\cite[Section 9.8, Example C.3]{Brezis})  that there exists a constant $C>0$ such that for all $a,b\in D(\hA^\frac32)$ and $c\in D(\hA)$ we have
		\begin{align*}
		\lvert [\nabla a : \nabla b]c \rvert_{\rH^1}^2 =& \lvert [\nabla a : \nabla b]c \rvert_{L^2}^2 + \sum_{i=1}^2 \lvert \partial_i \left(  [\nabla a : \nabla b]c  \right) \lvert^2_{L^2}\\
		\le & C \lvert\nabla a \rvert_{L^4}^2 \lvert \nabla b \rvert_{L^4}^2 \lvert c \rvert_{L^\infty}^2 +C \lvert c \rvert^2_{L^\infty} \sum_{i=1}^2\left( \lvert \nabla \partial_i a \rvert^2_{L^4} \lvert \nabla b \rvert^2_{L^4}  + \lvert \nabla a \rvert^2_{L^4} \lvert \nabla \partial_i b \rvert^2_{L^4}\right) \\
		& \qquad \qquad +C \sum_{i=1}^2 \lvert \nabla a\rvert^2_{L^8} \lvert \nabla b   \rvert^2_{L^8 } \lvert \partial_i c \rvert^2_{L^4}  \\
		\le & C \lvert a \rvert_{\rH^2}^2 \lvert  b \rvert_{\rH^2}^2 \lvert c \rvert_{\rH^2}^2+C \lvert c \rvert^2_{\rH^2} \sum_{i=1}^2\left( \lvert \nabla \partial_i a \rvert_{L^2} \lvert \nabla \partial_i a \rvert_{\rH^1 }\lvert b \rvert^2_{\rH^2}  + \lvert \nabla a \rvert^2_{\rH^2} \lvert \nabla \partial_i b \rvert_{L^2} \lvert \nabla \partial_i b \rvert_{\rH^1} \right) \\
		& \qquad \qquad +C \sum_{i=1}^2 \lvert \nabla a\rvert^2_{\rH^1} \lvert \nabla b   \rvert^2_{\rH^1} \lvert \partial_i c \rvert^2_{L^4}  \\
		\le &  C \lvert a \rvert_{\rH^2}^2 \lvert  b \rvert_{\rH^2}^2 \lvert c \rvert_{\rH^2}^2 +  \lvert c \rvert^2_{\rH^2}\left(\lvert a \rvert_{\rH^2} \lvert a \rvert_{\rH^3} \lvert b \rvert^2_{\rH^2} + \lvert a \rvert^2_{\rH^2} \lvert b \rvert_{\rH^2} \lvert b\rvert_{\rH^3}   \right).
		\end{align*}
		Thus, the proof of Claim \ref{Claim:Projection-Of-Laplacian} is complete.
	\end{proof}
	Let us resume the proof of \eqref{Eq:GradientNonlinearity-Lipschitz}.  Applying the claim and integrating over $[0,T]$ yields
	\begin{align*}
	\lvert [ \nabla n_2: \nabla n_2 ] (n_1-n_2)\rvert^2_{L^2(0,T;\rH^1)} \le  C T \lvert n_2\rvert^4_{C([0,T;D(\hA)])}  \lvert n_1-n_2\rvert^2_{C([0,T;D(\hA)])} \\
	+ C \lvert n_1-n_2\rvert^2_{C([0,T;D(\hA)])}\Bigl(\lvert n_2\rvert^3_{C([0,T;D(\hA)])}\int_0^T \lvert n_2(s) \rvert_{\rH^3} ds \Bigr)\\
	\le
	C T^\frac12 \lvert n_1-n_2\rvert^2_{\mathbf{X}^2_T} \Bigl(\lvert n_2\rvert^2_{\mathbf{X}^2_T}  \lvert n_2 \rvert_{C([0,T;D(\hA)])}\left(\int_0^T \lvert n_2(s) \rvert^2_{\rH^3} ds\right)^\frac12 \Bigr)\\
	+ C T   \lvert n_2\rvert^4_{\mathbf{X}^2_T }  \lvert n_1-n_2\rvert^2_{\mathbf{X}^2_T},
	\end{align*}
	for some constant $C>0$ and for all $n_i\in \mathbf{X}^2_T$, $i=1,2$.
	The last line of the above chain of inequalities yields that there exists a constant $C>0$ such that for all $n_i\in \mathbf{X}^2_T$, $i=1,2$,
	\begin{align}\label{Eq:GradientNonlinearity-1}
	\lvert [ \nabla n_2: \nabla n_2 ] (n_1-n_2)\rvert^2_{L^2(0,T;\rH^1)} \le  C( T \vee T^\frac12) \lvert n_2\rvert^4_{\mathbf{X}^2_T}\lvert n_1-n_2\rvert^2_{\mathbf{X}^2_T}.
	\end{align}
	In a similar way, one can show that there exists a constant $C>0$ such that for all $n_i\in \mathbf{X}^2_T$, $i=1,2$,
	\begin{align}\label{Eq:GradientNonlinearity-2}
	\lvert [ \nabla (n_1-n_2): \nabla(n_1+n_2)] n_1\rvert_{L^2(0,T; \rH^1)}^2 \le  C(T\vee T^\frac12 )\lvert n_1-n_2\rvert^2_{\mathbf{X}^2_T}\left( \lvert n_2\rvert^4_{\mathbf{X}^2_T}+ \lvert n_1\rvert^4_{\mathbf{X}^2_T}\right).
	\end{align}
	We easily complete the proof of \eqref{Eq:GradientNonlinearity-Lipschitz} by using \eqref{Eq:GradientNonlinearity-1} and \eqref{Eq:GradientNonlinearity-2} in the following inequality
	\begin{align*}
	\lvert \lvert \nabla n_1\rvert^2n_1 -  \lvert \nabla n_2\rvert^2n_2\rvert^2_{L^2(0,T; \rH^1)} \le &
	2 \lvert [ \nabla (n_1-n_2): \nabla(n_1+n_2)] n_1\rvert_{L^2(0,T; \rH^1)} ^2 \\
	&+ 2 \lvert [\nabla n_2 : \nabla n_2] (n_2-n_1 ) \rvert^2_{L^2(0,T; \rH^1)} .
	\end{align*}
	In order to complete the proof of Lemma \ref{Lem:RighthandSideinL2VxH1} we need to establish the inequality \eqref{Eq:FPT-ForcingOptDir-2}. For this purpose, we firstly observe that it is not difficult to prove that there exists a constant $C>0$ such that for all $n_i$, $i=1,2$,
	\begin{align*}
	\lvert (n_1-n_2) \times g \rvert^2_{L^2(0,T; \rH^1)} \le & \lvert n_1- n_2 \rvert^2_{C([0,T]; D(\hA))} \lvert g \rvert^2_{L^2(0,T; \rh^1)} + \lvert \nabla (n_1-n_2) \rvert^2_{C([0,T]; L^4)} \int_0^T\lvert g(s)\rvert^2_{L^4}ds.
	\end{align*}
	Hence, there exists a constant $C>0$ such that for all $n_i\in \mathbf{X}^2_T$, $i=1,2$, we have
	\begin{equation*}
	\lvert (n_1-n_2) \times g \rvert^2_{L^2(0,T; \rH^1)} \le  \lvert n_1- n_2 \rvert^2_{\mathbf{X}^2_T} \lvert g \rvert^2_{L^2(0,T;\rH^1)},
	\end{equation*}
	which is the  inequality \eqref{Eq:FPT-ForcingOptDir-2}. Hence, the proof of Lemma \ref{Lem:RighthandSideinL2VxH1} is complete.
\end{proof}
\section{Weak solution of a modified viscous transport equation}
Let $d\in C([0,T];\rH^{1})\cap L^{2}(0,T;\rH^{2})$ and $u \in C([0,T]; \rH)\cap L^2(0,T;\ve )$. Consider the following problem
\begin{equation}
\begin{cases}
\partial_{t}z-\Delta z= 2|\nabla d|^{2}z-2(\phi^\prime(d)\cdot d)z-u\cdot \nabla z\\
{\frac{\partial z}{\partial \nu}}{\Big\lvert_{\partial\Omega}}=0,\\
z(0)=z_{0}.\label{tu1}
\end{cases}
\end{equation}
We introduce the definition of weak solution of problem \ref{tu1}.
\begin{Def}\label{Def:WeakSol-ViscousTransport}
	Let $z_0\in L^2$.  A function $z: [0,T] \to L^2$ is a  weak solution to (\ref{tu1}) iff
	\begin{enumerate}
		\item \label{Item:DefWeakSol-1} $z\in C([0,T]; L^2)\cap L^2(0,T; \rH^1)$;
		\item for all $t\in [0,T]$ and $\varphi \in \rH^1$
		\begin{equation*}
		(z(t) -z_0, \varphi)+ \int_0^t (\nabla z(s), \nabla \varphi) ds = \int_0^t \left(2 |\nabla d(s)|^{2}z(s)-2 (\phi^\prime(d(s))\cdot d(s) )z(s)-u(s)\cdot \nabla z(s), \varphi\right) ds
		\end{equation*}
	\end{enumerate}
\end{Def}
\begin{Rem}
	Let $d\in C([0,T];\rH^{1})\cap L^{2}(0,T;\rH^{2})$, $u \in C([0,T]; \rH)\cap L^2(0,T;\ve )$ and $z\in C([0,T]; L^2)\cap L^2(0,T; \rH^1)$ be a weak solution to \eqref{Eq:ViscousTransport-z}. Then, by using the H\"older inequality and Sobolev embeddings such as $\rH^1\hookrightarrow L^4$ and $\rH^2\hookrightarrow\infty$ we can easily shows that
	\begin{align*}
	\lvert \partial_t z\rvert_{(\rH^1)^\ast}=&\sup_{\varphi \in \rH^1 : \lvert \varphi \rvert_{\rH^1}=1 } \lvert\langle \partial_t z, \varphi \rangle \rvert\\
	\le & \lvert \nabla z \rvert_{L^2}+ \lvert \nabla d \rvert^2_{L^4} \lvert z\rvert_{L^4}+ \lvert u \rvert_{L^4} \lvert z\rvert_{L^4}+ (1+ \lvert d \rvert_{\rH^2})\lvert d \rvert_{L^4} \lvert z \rvert_{L^2}.
	\end{align*}
	Hence, since $u, \nabla d ,z\in C([0,T];L^2)\cap L^2(0,T; \rH^1)\subset L^4(0,T; L^4)$, we can show by using the H\"older  inequality and \eqref{tu3} that
	\begin{equation}\label{Item:DefWeakSol-3}
	\lvert \partial_t z\rvert_{L^2(0,T; (\rH^1)^\ast)}\le c,
	\end{equation}
	for a universal constant which depends only on $\Omega$ and $T$. 	
\end{Rem}
The following result gives the existence and uniqueness of problem (\ref{tu1}).
\begin{Prop}\label{tu4}
	Let $d\in C([0,T];\rH^{1})\cap L^{2}(0,T;\rH^{2})$, $u \in C([0,T]; \rH)\cap L^2(0,T;\ve )$ and $z_{0}\in L^{2}$. Let also  $\phi:\mathbb{R}^{3}\to \mathbb{R}^{+}$ be of class $C^{1}$ such that
	\begin{equation*}
	|\phi^\prime(d)|_{\mathbb{R}^{3}}\le c(1+|d|).
	\end{equation*}
	Then problem (\ref{tu1}) has a unique weak solution $z$. Moreover, there exists a constant $c>0$ such that
	\begin{equation}
	\sup_{0\le t\le T}|z(t)|^{2}_{L^{2}}+\int_{0}^{T}|\nabla z(t)|^{2}dt\le
	|z(0)\rvert^{2}_{L^{2}}e^{c\int_{0}^{T}\left[|\nabla d|^{4}_{L^{4}}+(1+|d|^{2}_{\rH^{2}})\right]dt}\label{tu3}
	\end{equation}
\end{Prop}
\begin{proof}
	Throughout this proof $c>0$ will denote an universal constant which depends only on $\Omega$ and may change from one term to the other. For the sake of simplicity we will omit the dependence on the space variable inside any integral over $\Omega$.
	
	Let $d\in C([0,T];\rH^{1})\cap L^{2}(0,T;\rH^{2})$, $u \in C([0,T]; \rH)\cap L^2(0,T;\ve )$ and $z_{0}\in L^{2}$.
	
	Since the problem \eqref{tu1} is linear, the proofs of the existence, which can be done via Galerkin and compactness methods, and the uniqueness are easy and omitted.
	So we only prove \eqref{tu3}.
	For this purpose,  $z\in C([0,T]; L^2)\cap L^2(0,T; \rH^1)$ be a weak solution to \eqref{Eq:ViscousTransport-z}. We firstly observe that since $\rH^2 \subset L^\infty$
	\begin{align*}
	\int_{\Omega}(\phi^\prime(d)\cdot d)|z|^{2}\;dx \le &
	c|d|_{L^{\infty}}(1+|d|_{L^\infty})\int_{\Omega}|z|^{2}\;dx\notag\\
	&\le c(1+|d|^{2}_{\rH^{2}})|z|^{2}_{L^{2}}.
	\end{align*}
	Also, since $u\in \rH^1_0$ and $\Div u =0$ we can prove that
	\begin{equation*}
	\langle u \cdot\nabla  z, z \rangle =\frac12 \int_\Omega u\cdot \nabla \lvert z\rvert^2 dx =0.
	\end{equation*}
	Hence, by the Gagliardo-Nirenberg inequality (\cite[Section 9.8, Example C.3]{Brezis})  and the Lions-Magenes lemma (\cite[Lemma III.1.2]{Temam_2001}), which is applicable because of Definition \ref{Def:WeakSol-ViscousTransport}\eqref{Item:DefWeakSol-1} and \eqref{Item:DefWeakSol-3},   we have the following inequalities
	\begin{align*}
	\langle \partial_t z, z\rangle=& \frac12 \frac{d}{dt}\lvert z\rvert^2_{L^2}\\
	=& - \frac{1}{2}|\nabla z|^{2}_{L^{2}}+ 2\int_\Omega \lvert \nabla d \rvert^2 \lvert z\rvert^2\; dx -2 \int_\Omega (\phi^\prime(d).d)|z|^{2}\;dx \\
	&\le - \frac{1}{2}|\nabla z|^{2}_{L^{2}}+ \int_{\Omega}|\nabla d|^{2}|z|^{2}\;dx + C(1+|d|^{2}_{\rH^{2}})|z|^{2}_{L^{2}}\notag\\
	&\le - \frac{1}{2}|\nabla z|^{2}_{L^{2}}+ c|\nabla d|^{2}_{L^{4}}|z|^{2}_{L^{4}}+ c(1+|d|^{2}_{\rH^{2}})|z|^{2}_{L^{2}}\notag\\
	&\le - \frac{1}{2}|\nabla z|^{2}_{L^{2}}+ c |\nabla d|^{2}_{L^{4}}|z|_{L^{2}}|\nabla z|_{L^{2}}+ \lvert \nabla d \rvert^2_{L^4} \lvert z\rvert^2_{L^2} + c(1+ |d|^{2}_{\rH^{2}})|z|^{2}_{L^{2}}\notag\\
	&\le- \frac{1}{2}|\nabla z|^{2}_{L^{2}}+\frac{1}{4}|\nabla z|^{2}_{L^{2}}+ c|\nabla d|^{4}_{L^{4}}|z|^{2}_{L^{2}}+ c(1+|d|^{2}_{\rH^{2}})|z|^{2}_{L^{2}}.
	\end{align*}
	Thus,
	\begin{align*}
	\frac12 \frac{d}{dt}\lvert z\rvert^2_{L^2}+\frac14 |\nabla z|^{2}_{L^{2}}\le c|\nabla d|^{4}_{L^{4}}|z|^{2}_{L^{2}}+ c(1+|d|^{2}_{\rH^{2}})|z|^{2}_{L^{2}},
	\end{align*}
	By applying the Gronwall lemma  we infer (\ref{tu3}).
	This also completes the proof of  the Proposition \ref{tu4}.
\end{proof}
\section{Proof of Claim \ref{Clm:EstimateHigherorder}}\label{App:EstHighOrder}
In this section we will the estimates \eqref{Eq:EstHighOrd-1}-\eqref{Eq:EstHighOrd-3} in Claim \ref{Clm:EstimateHigherorder}.
Throughout this section $C>0$ will denote an universal constant which may change from one term to the other.
%
%
\begin{proof}[Proof of inequality \eqref{Eq:EstHighOrd-1}]
	Let us choose and fix $v \in D(\rA)$ and $n \in D(\hA^\frac32)$.
	
	Since $D(\hA^\frac32)\subset \rH^3$ and $\rH^2$ is an algebra, it is not difficult to show that $\Div (\nabla n\odot \nabla n)\in L^2$. Thus, using the fact $\Pi:L^2 \to \rH$ is bounded, the H\"older inequality and Gagliardo-Nirenberg inequality (\cite[Section 9.8, Example C.3]{Brezis})  we infer that
	\begin{align*}
	\lvert \langle \rA v, \Pi [\Div(\nabla n \odot \nabla n)  ] \rangle \rvert=&\lvert \langle \rA v, \Pi[\nabla n)^{\mathrm{T}} \Delta n  ]\rangle \rvert\\
	& \le C  \lvert \rA v\rvert_{L^2} \lvert \nabla n \rvert_{L^4} \lvert \Delta n \rvert_{L^4} \\
	&	\le C \lvert \rA v \rvert_{L^2} \lvert \nabla n \rvert_{L^4} \lvert \Delta n \rvert_{L^2}^\frac12 \lvert \nabla \Delta n \rvert^\frac12_{L^2}.
	\end{align*}
	Now, applying the Young inequality twice  implies
	\begin{align*}
	\lvert \langle \rA v, \Pi_L [\Div(\nabla n \odot \nabla n)  ] \rangle \rvert
	\le \frac12( \lvert \rA v \rvert_{L^2}^2+  \lvert \nabla \Delta n \rvert^2_{L^2} )+ \lvert \nabla n \rvert_{L^4}^4\lvert \Delta n \rvert_{L^2}^2.
	\end{align*}
	This complete the proof of the inequality \eqref{Eq:EstHighOrd-1}.
\end{proof}
\begin{proof}[Proof of the inequality \eqref{Eq:EstHighOrd-2}]
	Let $v\in D(\rA)$ and $n \in D(\hA^\frac32)$ be fixed.
	
	Because $D(\rA)\subset \rH^2$, $D(\hA^\frac32)\subset \rH^3$ and $\rH^\theta$, $\theta>1$, is an algebra,  it is not difficult to show that $v\cdot \nabla n \in \rH^2\subset \rH^1$. Hence,  by using the Cauchy-Schwarz and Young inequalities we infer
	\begin{align*}
	{}_{(\rH^1)^\ast}\langle \rA^2 n, v\cdot \nabla n \rangle_{\rH^1} = &\langle \nabla \Delta n, \nabla (v\cdot \nabla n) \rangle \\
	&\le \frac1{24} \lvert \nabla \Delta n\rvert^2_{L^2} + C \lvert \nabla (v\cdot \nabla n) \rvert^2_{L^2}.
	\end{align*}
	Using the H\"older inequality in the last line  we infer that
	\begin{align*}
	{}_{(\rH^1)^\ast}\langle \rA^2 n, v\cdot \nabla n \rangle_{\rH^1} \le \frac1{24} \lvert \nabla \Delta n\rvert^2_{L^2} + C \lvert (\nabla v) (\nabla n) \rvert^2_{L^2}+C \lvert v\cdot \nabla(\nabla n)\rvert^2_{L^2}\\
	\le 	\frac1{24} \lvert \nabla \Delta n\rvert^2_{L^2} + C \lvert (\nabla v)\rvert^2_{L^4 } \rvert(\nabla n) \rvert^2_{L^4}+C \lvert v\rvert^2_{L^4} \lvert \nabla^2 n\rvert^2_{L^4}
	\end{align*}
	Now, by using \cite[Theorem 3.4]{Simader}  to estimate $\lvert \nabla^2 n \rvert_{L^4}$ by $\lvert \Delta n \rvert_{L^4}$, then by applying the Gagliardo-Nirenberg (\cite[Section 9.8, Example C.3]{Brezis}) , the Young inequalities we infer that
	\begin{align}
	{}_{(\rH^1)^\ast}\langle \rA^2 n, v\cdot \nabla n \rangle_{\rH^1}
	\le 	\frac1{24} \lvert \nabla \Delta n\rvert^2_{L^2} + C \lvert \nabla v\rvert_{L^2 }\lvert \rA v \rvert_{L^2} \rvert(\nabla n) \rvert^2_{L^4}+C \lvert v\rvert^2_{L^4} \lvert \Delta n\rvert^2_{L^4}\nonumber\\
	\le 	\frac1{24} \lvert \nabla \Delta n\rvert^2_{L^2} +\frac1{12} \lvert \rA v\rvert^2_{L^2}+ C \lvert \nabla v\rvert^2_{L^2 } \rvert(\nabla n) \rvert^4_{L^4}+C \lvert v\rvert^2_{L^4} [\lvert \Delta n\rvert_{L^2} \lvert \nabla \Delta n \rvert_{L^2} +\lvert \Delta n \rvert^2_{L^2} ]\nonumber \\
	\le 	\frac1{12} \lvert \nabla \Delta n\rvert^2_{L^2} +\frac1{12} \lvert \rA v\rvert^2_{L^2}+ C \lvert \nabla v\rvert^2_{L^2 } \rvert(\nabla n) \rvert^2_{L^4}+C [ \lvert v \rvert^2_{L^4}+\lvert v\rvert^4_{L^4} ]\lvert \Delta n\rvert_{L^2}.\label{Eq:EstHighOrd-1b}
	\end{align}
	This completes the proof of the inequality \eqref{Eq:EstHighOrd-2}.
\end{proof}
\begin{proof}[Proof of the inequality \eqref{Eq:EstHighOrd-3}]
	Let $v \in D(\rA)$ and $n \in D(\hA^\frac32)$. As in the proof of \eqref{Eq:EstHighOrd-1}, we can show that $\lvert \nabla n \rvert^2n \in \rH^1$. Hence, using  the Cauchy-Schwarz, the Young and the Gagliardo-Nirenberg inequalities (see \cite[Section 9.8, Example C.3]{Brezis}) we infer that
	\begin{align*}
	{}_{(\rH^1)^\ast}\langle \hA^2 n, \lvert \nabla n\rvert^2n \rangle_{\rH^1} = & \langle \nabla \Delta n, \nabla (\lvert \nabla n \rvert^2 n) \rangle \\
	&\le \frac1{24} \lvert \nabla \Delta  n \rvert^2_{L^2} + C \lvert \nabla (\lvert \nabla n \rvert^2 n) \rvert^2_{L^2}\\
	& \le \frac1{24} \lvert \nabla \Delta  n \rvert^2_{L^2} + C \lvert \nabla (\lvert \nabla n \rvert^2 n) \rvert^2_{L^2}\\
	&\le \frac{1}{24}\lvert \nabla \Delta n \rvert^2_{L^2} +C \lvert (\nabla n)(\nabla^2 n ) \rvert^2_{L^2}+ C \lvert \lvert \nabla n \rvert^2 \nabla n \rvert^2_{L^2}\\
	&\le \frac{1}{24}\lvert \nabla \Delta n \rvert^2_{L^2} +C \lvert (\nabla n)\rvert^2_{L^4} \lvert (\nabla^2 n ) \rvert^2_{L^4}+ C \lvert \nabla n  \rvert^6_{L^6}\\
	&\le \frac{1}{24}\lvert \nabla \Delta n \rvert^2_{L^2} +C \lvert (\nabla n)\rvert^2_{L^4} \lvert (\nabla^2 n ) \rvert^2_{L^4}+ C \lvert \nabla n  \rvert^4_{L^4}(\lvert \nabla n\rvert^2_{L^2} + \lvert \nabla^2 n \rvert^2_{L^2}).
	\end{align*}
	Proceeding as in the proof of \eqref{Eq:EstHighOrd-1b}, we show that
	\begin{align*}
	{}_{(\rH^1)^\ast}\langle \hA^2 n, \lvert \nabla n\rvert^2n \rangle_{\rH^1}
	&\le \frac{1}{12}\lvert \nabla \Delta n \rvert^2_{L^2} +C [\lvert (\nabla n)\rvert^4_{L^4}+ \lvert (\nabla n)\rvert^2_{L^4}]\lvert (\Delta n ) \rvert^2_{L^2}+ C \lvert \nabla n  \rvert^4_{L^4}(\lvert \nabla n\rvert^2_{L^2} + \lvert \Delta n \rvert^2_{L^2}).
	\end{align*}
	This completes the proof of \eqref{Eq:EstHighOrd-3}.
\end{proof}
	\section{A weak continuity of  Banach space valued functions}
	In this section we will state and prove of continuity theorem for Banach-valued map similar to \cite{Strauss}. This theorem was initially proven in the work in progress \cite{ZB+AP+UM}, but for the sake of completeness we repeat the proof here.
	\begin{Thm}\label{Thm:StraussThm}
		Let $X$ and $Y$ be two Banach spaces such that $X$ is reflexive, $X\subset Y$ and the canonical injection $i:X\to Y$ is dense and continuous.
		Let $T>0$ be fixed and $u \in L^\infty([0,T); X)$. Let also $b \in Y$ and $v:[0,T] \to Y$, defined by
		\begin{equation*}
		v(t)=
		\begin{cases}
		i(u(t)) \text{ if } t\in [0,T),\\
		b \text{ if } t=T,
		\end{cases}
		\end{equation*}
		be weakly continuous.
		
		Then, $b\in X$ and  the map $\tilde{u}: [0,T] \to X$ defined by
			\begin{equation}\label{Eq:ModifiedMap}
		\tilde{u}(t)=
		\begin{cases}
		u(t) \text{ if } t\in [0,T),\\
		b \text{ if } t=T,
		\end{cases}
		\end{equation}
		is weakly continuous.
	\end{Thm}
\begin{proof}
	Let $X$, $Y$, $b\in Y$ , $T>0$, $u\in L^\infty([0,T);X)$, $v:[0,T]\to Y$ be as in the statement of the theorem.  Let also $(t_n)_{n\in \mathbb{N}}\subset [0,T)$ be a sequence such that $t_n \toup T$.
	
	Let us prove the first part of the theorem, \text{i.e.}, that $b\in X$. For this aim we observe that by assumption there exists $M>0$ such that
	\begin{equation*}
	\lvert u(t) \rvert_X \le M \text{ for all } t\in [0,T).
	\end{equation*}
	Hence, by the Banach-Alaoglu theorem we can extract from $(t_n)_{n \in \mathbb{N}}$ a subsequence,   which is still denoted by $(t_n)_{n\in \mathbb{N}}$, such that $t_n\toup T$ and
	\begin{equation*}
	u(t_n) \to x \text{ weakly in } X.
	\end{equation*}
	Since, by assumption,  $X\subset Y$ we infer that $v(t) = i(u(t)) =u(t)$ for all $t\in [0,T]$. Hence, by the weak continuity of $v$ we infer that
	\begin{equation}\notag
	v(t_n)=u(t_n) \to b \text{ weakly in } Y.
	\end{equation}
	By the uniqueness of weak limit we infer that $x=b \in X$.
	
	It  remains to prove the second part of the theorem, \textit{i.e.}, we shall show that the map $\tilde{u}:[0,T] \to X$ defined in \eqref{Eq:ModifiedMap} is weakly continuous. For this purpose, we will closely follow the proof of \cite[Lemma III.1.4 ]{Temam_2001}.  We will divide the task into two cases.
	\begin{trivlist}
		\item[\textbf{Case 1}] Let $t, t_0 \in [0,T)$. Let $X^\ast$ and $Y^\ast$ be the dual spaces of $X$ and $Y$, respectively.
	Recall that since the embedding $X\subset Y$ is dense and continuous, the embedding $Y^\ast\subset X^\ast$ is dense and continuous, too. 	 Let us choose and fix $\eps>0$ and $\eta\in X^\ast$. Then, there exists $\eta_\eps\in Y^\ast$ such that
	\begin{equation}\notag
	\lvert \eta -\eta_\eps \rvert_{X^\ast}<\eps.
	\end{equation}
Thus, using the boundedness of $u=\tilde{u}\Big\vert_{[0,T)}$ we infer that
	\begin{align}
	\lvert {}_{X}\langle \tilde{u}(t) -\tilde{u}(t_0), \eta \rangle_{X^\ast}\rvert\le &\lvert {}_{X}\langle \tilde{u}(t) -\tilde{u}(t_0), \eta-\eta_\eps \rangle_{ X^\ast}\rvert+ 	\lvert {}_{Y}\langle \tilde{u}(t) -\tilde{u}(t_0), \eta \rangle_{Y^\ast}\rvert\notag\\
	\le & 2 M \lvert \eta -\eta_\eps\rvert_{X^\ast} + 	\lvert {}_{Y}\langle \tilde{u}(t) -\tilde{u}(t_0), \eta \rangle_{Y^\ast}\rvert.\notag
	\end{align}
	\end{trivlist}
	By the weak continuity of $v\Big\vert_{[0,T)}=u=\tilde{u}$ we have
	\begin{align}
	\lim_{t\to t_0}  \lvert {}_{X}\langle \tilde{u}(t) -\tilde{u}(t_0), \eta \rangle_{X^\ast}\rvert\le 2M \lvert \eta- \eta_\eps\rvert_{X^\ast} + \lim_{t\to t_0} 	\lvert {}_{Y}\langle v(t) -v(t_0), \eta \rangle_{Y^\ast} \rvert
	\le 2M \eps.
	\end{align}
	Since $\eps>0$ is arbitrary we infer that
	\begin{equation}
	\lim_{t\to t_0} \lvert {}_{X}\langle \tilde{u}(t) -\tilde{u}(t_0), \eta \rangle_{X^\ast}\rvert=0.\notag
	\end{equation}
	Hence, $\tilde{u}\Big\vert_{[0,T)}$ is weakly continuous.
	\item[\textbf{Case 2}.] We will prove that $\tilde{u}$ is weakly continuous at $t=T$. Towards this aims let $\eta\in X^\ast$, $\eps>0$ be fixed.
	From the proof of the first part of the theorem we infer that there exists $\delta>0$ such that for $t>0$ with $0<T-t<\delta$
	\begin{equation*}
	\lvert{}_{X}\langle u(t) -b, \eta \rangle_{X^\ast}  \rvert= \lvert{}_{X}\langle u(t) -\tilde{u}(T), \eta \rangle_{X^\ast} \rvert<\eps.
	\end{equation*}
	But for $t\in (0,T)$ we have $u(t)=\tilde{u}(t)$, hence for $t>0$ with $0<T-t<\delta$
	\begin{equation*}
	\lvert{}_{X}\langle \tilde{u}(t) -\tilde{u}(T), \eta \rangle_{X^\ast}  \rvert \le  	\lvert{}_{X}\langle \tilde{u}(t) -u(t), \eta \rangle_{X^\ast} \rvert+ 	\lvert{}_{X}\langle u(t) -\tilde{u}(T), \eta \rangle_{ X^\ast} \rvert
	< \eps.
	\end{equation*}
	Thus, $\tilde{u}$ is weakly continuous at $t=T$.  This completes the proof of Case 2, the second part of the theorem and hence the whole theorem.
\end{proof}

\end{document}